\documentclass[11pt,leqno]{article}
\usepackage{amsfonts}
\usepackage{amsmath}
\usepackage{amssymb}
\usepackage{amsthm}
\usepackage{latexsym}
\usepackage{mathrsfs}
\usepackage{amsthm}
\usepackage{mathtools}
\usepackage{bm}
\usepackage{verbatim}
\usepackage{color}

\newtheorem{theorem}{\bf Theorem}[section]
\newtheorem{lemma}[theorem]{\bf Lemma}

\newtheorem{proposition}[theorem]{\bf Proposition}

\newtheorem{remark}[theorem]{\bf Remark}
\newtheorem{definition}[theorem]{\bf Definition}

\numberwithin{equation}{section}

\newcommand{\dt}{\partial_t}
\newcommand{\ds}{\partial_s}

\makeatletter
\newcommand{\opnorm}{\@ifstar\@opnorms\@opnorm}
\newcommand{\@opnorms}[1]{%
  \left|\mkern-1.5mu\left|\mkern-1.5mu\left|
   #1
  \right|\mkern-1.5mu\right|\mkern-1.5mu\right|
}
\newcommand{\@opnorm}[2][]{%
  \mathopen{#1|\mkern-1.5mu#1|\mkern-1.5mu#1|}
  #2
  \mathclose{#1|\mkern-1.5mu#1|\mkern-1.5mu#1|}
}

\begin{document}
%----------------------------------------------------------------------------------------------------------------------
%----------------------------------------------------------------------------------------------------------------------
\vspace*{0ex}
\begin{center}
{\Large\bf
Well-posedness of the initial boundary value problem \\[0.5ex]
for the motion of an inextensible hanging string
}
\end{center}

\begin{center}
Tatsuo Iguchi and Masahiro Takayama\footnote{Corresponding author}
\end{center}

\begin{abstract}
We consider the motion of an inextensible hanging string of finite length under the action of the gravity. 
The motion is governed by nonlinear and nonlocal hyperbolic equations, which is degenerate at the free end of the string. 
We show that the initial boundary value problem to the equations of motion is well-posed locally in time in weighted Sobolev spaces 
at the quasilinear regularity threshold under a stability condition. 
This paper is a continuation of our preceding articles \cite{IguchiTakayama2024, IguchiTakayama2023-2}. 
\end{abstract}

%----------------------------------------------------------------------------------------------------------------------
%----------------------------------------------------------------------------------------------------------------------
\section{Introduction}
We consider the initial boundary value problem to the motion of an inextensible hanging string of finite length under the action of the gravity, 
where one end of the string is fixed and another one is free. 
We also consider the problem in the case without any external forces. 
After an appropriate scaling, the equations of motion have the form 
\begin{equation}\label{Eq}
\begin{cases}
 \ddot{\bm{x}}-(\tau\bm{x}')' = \bm{g} &\mbox{in}\quad (0,1)\times(0,T), \\
 |\bm{x}'|=1 &\mbox{in}\quad (0,1)\times(0,T),
\end{cases}
\end{equation}
under the boundary conditions 
\begin{equation}\label{BC}
\begin{cases}
 \bm{x}=\bm{0} &\mbox{on}\quad \{s=1\}\times(0,T), \\
 \tau=0 &\mbox{on}\quad \{s=0\}\times(0,T),
\end{cases}
\end{equation}
where $\bm{x}(s,t)=(x_1(s,t),x_2(s,t),x_3(s,t))$ is the position vector of the string at time $t$ with the arc length parametrization $s\in[0,1]$ 
measured from the free end of the string, $\tau(s,t)$ is the tension of the string, 
and $\bm{g}$ is the acceleration of gravity vector, which is assumed to be a constant unit vector or the zero vector; see Figure \ref{intro:hanging string}.
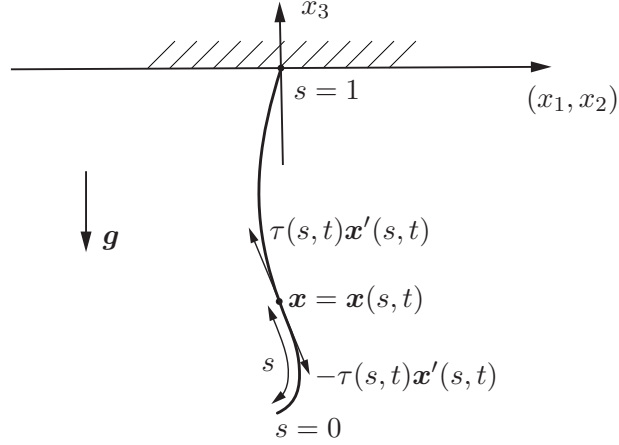
\begin{figure}[ht]
\setlength{\unitlength}{1pt}
\begin{picture}(0,0)
\put(320,-50){$(x_1,x_2)$}
\put(235,-15){$x_3$}
\put(232,-47){$s=1$}
\put(225,-175){$s=0$}
\put(220,-150){$s$}
\put(160,-102){$\bm{g}$}
\put(230,-126){$\bm{x}=\bm{x}(s,t)$}
\put(240,-155){$-\tau(s,t)\bm{x}'(s,t)$}
\put(223,-100){$\tau(s,t)\bm{x}'(s,t)$}
\end{picture}
\begin{center}
\includegraphics[width=0.45\linewidth]{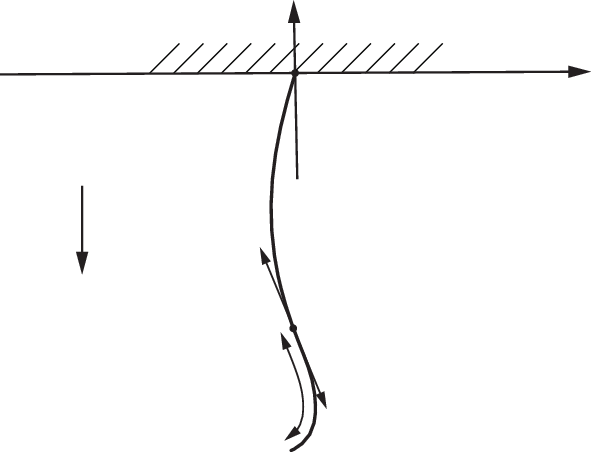}
\end{center}
\caption{Hanging String}
\label{intro:hanging string}
\end{figure}
Here, $\dot{\bm{x}}$ and $\bm{x}'$ denote the derivatives of $\bm{x}$ with respect to $t$ and $s$, respectively. 
For details on these equations, we refer to Reeken \cite{Reeken1979-1}, Yong \cite{Yong2006}, and Iguchi and Takayama \cite{IguchiTakayama2024}. 
We consider the initial boundary value problem of the above equations under the initial conditions 
\begin{equation}\label{IC}
(\bm{x},\dot{\bm{x}})|_{t=0}=(\bm{x}_0^\mathrm{in},\bm{x}_1^\mathrm{in}) \quad\mbox{in}\quad (0,1).
\end{equation}

\medskip
Our objective in this paper is to establish the well-posedness locally in time of the initial boundary value problem \eqref{Eq}--\eqref{IC} 
under the stability condition 
\begin{equation}\label{SC}
\frac{\tau(s,t)}{s}\geq c_0>0
\end{equation}
for $(s,t)\in(0,1)\times(0,T)$ in the class $\bm{x}\in \mathscr{X}_T^m=\bigcap_{j=0}^m C^j([0,T];X^{m-j})$ with $m\geq4$, 
where $X^m$ is a weighted Sobolev space of order $m$ on the interval $(0,1)$; see Section \ref{sect:pre} for the definition of the space $X^m$. 
Here, we note that in this problem $\tau$ is an unknown quantity as well as $\bm{x}$. 
As is well-known, once $\bm{x}'(\cdot,t)$ and $\dot{\bm{x}}'(\cdot,t)$ are given, $\tau(\cdot,t)$ is uniquely determined as a unique solution 
of the two-point boundary value problem 
\begin{equation}\label{BVP}
\begin{cases}
 -\tau''(\cdot,t)+|\bm{x}''(\cdot,t)|^2\tau(\cdot,t) = |\dot{\bm{x}}'(\cdot,t)|^2 &\mbox{in}\quad (0,1), \\
 \tau(0,t)=0, \quad \tau'(1,t)=-\bm{g}\cdot\bm{x}'(1,t).
\end{cases}
\end{equation}
Therefore, $\tau$ can be regarded as a function of $\bm{x}'$ and $\dot{\bm{x}}'$, 
so that the first equation in \eqref{Eq} can be regarded as nonlinear wave equations for $\bm{x}$, which are degenerate at the free end where $\tau=0$. 
However, this dependence is nonlocal in space and causes a nonlocal character of the equations.

\medskip
Our main result in this paper is given by the following theorem.

\begin{theorem}\label{th:main}
For any integer $m\geq4$ and any positive constants $M_0$ and $c_0$, 
there exists a small positive time $T$ such that if the initial data satisfy 
\begin{equation}\label{CondID}
\begin{cases}
 \|\bm{x}_0^\mathrm{in}\|_{X^m}+\|\bm{x}_1^\mathrm{in}\|_{X^{m-1}} \leq M_0, \\
 |\bm{x}_0^{\mathrm{in}\prime}(s)|=1, \quad \bm{x}_0^{\mathrm{in}\prime}(s)\cdot\bm{x}_1^{\mathrm{in}\prime}(s)=0, \quad
  \frac{\tau_0^\mathrm{in}(s)}{s} \geq 2c_0 \quad\mbox{for}\quad 0<s<1, 
\end{cases}
\end{equation}
where $\tau_0^\mathrm{in}(s)=\tau(s,0)$ is the initial tension, and the compatibility conditions up to order $m-1$ in the sense of Definition \ref{def:CC}, 
then the initial boundary value problem \eqref{Eq}--\eqref{IC} has a unique solution $(\bm{x},\tau)$ 
satisfying $\bm{x}\in\mathscr{X}_T^m$ and the stability condition \eqref{SC}. 
Moreover, we have $\tau' \in \mathscr{X}_T^{m-1,*}$ in the case $m\geq5$ and $\tau' \in \mathscr{X}_T^{3,*,\epsilon}$ for any $\epsilon>0$ in the case $m=4$. 
\end{theorem}

\begin{remark}\label{re:MainTh}
\begin{enumerate}
\item[\rm(1)]
The initial tension $\tau_0^\mathrm{in}$ is uniquely determined from the initial data $(\bm{x}_0^\mathrm{in},\bm{x}_1^\mathrm{in})$ as a unique solution 
of the two-point boundary value problem \eqref{BVP} in the case $t=0$. 
Moreover, as was shown in Iguchi and Takayama \cite[Lemma 3.4]{IguchiTakayama2024}, we have 
\[
\frac{\tau_0^\mathrm{in}(s)}{s} \geq \left( -\bm{g}\cdot\bm{x}_0^{\mathrm{in}\prime}(1)
 + \|\sigma^\frac12\bm{x}_1^{\mathrm{in}\prime}\|_{L^2}^2 \exp\bigl( -\|\sigma^\frac12\bm{x}_0^{\mathrm{in}\prime\prime}\|_{L^2}^2 \bigr) \right)
 \exp\bigl( -\|\sigma^\frac12\bm{x}_0^{\mathrm{in}\prime\prime}\|_{L^2}^2 \bigr),
\]
if the right-hand side is non-negative. 
Therefore, if the string is in fact hanging initially, that is to say, if $-\bm{g}\cdot\bm{x}_0^{\mathrm{in}\prime}(1)>0$, 
then the stability condition is initially satisfied. 

\item[\rm(2)]
In the case $\bm{g}=\bm{0}$, that is, in the case without any external forces, if $\bm{x}_1^\mathrm{in}=\bm{0}$, 
then the initial boundary value problem \eqref{Eq}--\eqref{IC} has a unique global solution $(\bm{x}(s,t),\tau(s,t))=(\bm{x}_0^\mathrm{in}(s),0)$, 
which does not satisfy the stability condition \eqref{SC}; see \cite[Theorem 2.4]{IguchiTakayama2024}. 
On the other hand, if $\bm{x}_1^\mathrm{in}\ne\bm{0}$, then the stability condition is initially and automatically satisfied. 

\item[\rm(3)]
In the case $\bm{g}=\bm{0}$, the existence of a unique solution $(\bm{x},\tau)$ to the initial boundary value problem \eqref{Eq}--\eqref{IC} 
has already been established by Preston \cite{Preston2011} in the class $\bm{x}\in L^\infty(0,T;X^4) \cap W^{1,\infty}(0,T;X^3)$. 
Our result asserts not only the existence of a unique solution but also the strong continuity of the solution in time, that is, 
$\bm{x}\in C^0([0,T];X^4) \cap C^1([0,T];X^3)$. 

\item[\rm(4)]
As was explained in \cite[Remark 2.2]{IguchiTakayama2024}, the requirement $m\geq4$ corresponds to the quasilinear regularity in the sense that 
$m=4$ is the minimal integer regularity index $m$ that ensures the embedding $C^0([0,T];X^m)\cap C^1([0,T];X^{m-1}) \hookrightarrow C^1([0,1]\times[0,T])$. 
Therefore, $m=4$ is a critical regularity index in the classical sense. 

\item[\rm(5)]
As for the compatibility conditions and the definition of the spaces $\mathscr{X}_T^{m-1,*}$ and $\mathscr{X}_T^{3,*,\epsilon}$, 
we refer to Sections \ref{sect:CC} and \ref{sect:pre}, respectively. 
\end{enumerate}
\end{remark}

As was shown in \cite[Theorem 2.5]{IguchiTakayama2024}, in the class of solutions stated in Theorem \ref{th:main} 
the initial boundary value problem \eqref{Eq}--\eqref{IC} is equivalent to 
\begin{equation}\label{HP2}
\begin{cases}
 \ddot{\bm{x}}-(\tau\bm{x}')' = \bm{g} &\mbox{in}\quad (0,1)\times(0,T), \\
 \bm{x}=\bm{0} &\mbox{on}\quad \{s=1\}\times(0,T),
\end{cases}
\end{equation}
\begin{equation}\label{BVP2}
\begin{cases}
 -\tau''+|\bm{x}''|^2\tau = |\dot{\bm{x}}'|^2 &\mbox{in}\quad (0,1)\times(0,T), \\
 \tau=0 &\mbox{on}\quad \{s=0\}\times(0,T), \\
 \tau'=-\bm{g}\cdot\bm{x}' &\mbox{on}\quad \{s=1\}\times(0,T),
\end{cases}
\end{equation}
\begin{equation}\label{IC2}
(\bm{x},\dot{\bm{x}})|_{t=0}=(\bm{x}_0^\mathrm{in},\bm{x}_1^\mathrm{in}) \quad\mbox{in}\quad (0,1)
\end{equation}
under the restrictions $|\bm{x}_0^{\mathrm{in}\prime}(s)|\equiv1$ and $\bm{x}_0^{\mathrm{in}\prime}(s)\cdot\bm{x}_1^{\mathrm{in}\prime}(s)\equiv0$ 
on the initial data and the stability condition \eqref{SC}. 
In other words, we can remove the constraint $|\bm{x}'|=1$ from the equations. 
Therefore, it is sufficient to show the well-posedness of this transformed problem \eqref{HP2}--\eqref{IC2}.

\medskip
Well-posedness of a linearized problem to the nonlinear one \eqref{HP2}--\eqref{IC2} was established by Iguchi and Takayama \cite{IguchiTakayama2023-2} 
in the class $\mathscr{X}_T^m$ of the weighted Sobolev space. 
However, the map $(\bm{x},\tau) \mapsto (\bm{y},\nu)$ reveals a loss of twice derivatives, 
where $(\bm{y},\nu)$ denote variations in the linearization to the nonlinear problem around $(\bm{x},\tau)$, 
so that the standard Picard iteration cannot be applicable directly to show the existence of solution to the nonlinear problem. 
To overcome this difficulty, we implement a quasilinearization procedure and reduce the nonlinear problem into a quasilinear one equivalently 
introducing new unknowns $(\bm{y},\nu)$ by 
\begin{equation}\label{defNU}
\bm{y}=\ddot{\bm{x}}, \quad \nu=\ddot{\tau}.
\end{equation}
Then, we can solve the quasilinear problem for unknowns $(\bm{x},\bm{y},\tau,\nu)$ by applying the existence theory established in \cite{IguchiTakayama2023-2} and 
by a standard fixed point argument. 
This procedure works well under the additional regularity assumption $m\geq6$.

\medskip
To show the well-posedness in the case $m=4,5$, we approximate the initial data by regular ones, which must satisfy higher order compatibility conditions, 
and construct a sequence of solutions for the approximate initial data. 
Thanks to the a priori estimate of the solution obtained in Iguchi and Takayama \cite{IguchiTakayama2024}, we obtain a uniform bound of the solutions. 
Then, by the standard compactness argument, we can extract a subsequence of solutions which converges the desired solution. 
However, we face a difficulty in the approximation process of the initial data due to the nonlocal character of the problem, 
so that we need to study the compatibility conditions in detail.

\medskip
It is well-known that in the theory of the initial boundary value problem for hyperbolic systems with an unknown $u$, if the boundary is non-characteristic, 
then the initial data $u_0^\mathrm{in}\in H^m$ satisfying compatibility conditions up to order $m-1$ can be approximated by 
a sequence of more regular data $\{u_0^{\mathrm{in}(n)}\}_{n=1}^\infty \subset H^{m+p}$ with a positive integer $p$, 
which converges to $u_0^\mathrm{in}$ in $H^m$ and satisfy the compatibility conditions up to order $m+p-1$, 
where $H^m$ is the standard $L^2$ Sobolev space of order $m$. 
See, for example, Rauch and Massay \cite{RauchMassey1974}. 
To construct such approximate data, we first approximate the initial data by $\{v_0^{\mathrm{in}(n)}\}_{n=1}^\infty \subset H^{m+p}$, 
which converges to $u_0^\mathrm{in}$ in $H^m$. 
Generally, $v_0^{\mathrm{in}(n)}$ does not satisfy the compatibility conditions. 
Therefore, we need to modify the data by adding an appropriate function $w_0^{\mathrm{in}(n)}$ so that $\{w_0^{\mathrm{in}(n)}\}_{n=1}^\infty$ 
converges to $0$ in $H^m$ and that $u_0^{\mathrm{in}(n)}=v_0^{\mathrm{in}(n)}+w_0^{\mathrm{in}(n)}$ satisfies the compatibility conditions. 
To this end, we usually use the following structure of the system: 
$u_j^\mathrm{in}=(\dt^ju)|_{t=0}$ $(j=1,2,3,\ldots)$ are determined from the initial data $u_0^\mathrm{in}$ through the hyperbolic system 
and have the form 
\begin{align*}
u_j^\mathrm{in}
&= F_j(u_0^\mathrm{in}, \ds u_0^\mathrm{in}, \ldots, \ds^j u_0^\mathrm{in}) \\
&= A_j(u_0^\mathrm{in})\ds^ju_0^\mathrm{in} + G_j(u_0^\mathrm{in}, \ds u_0^\mathrm{in}, \ldots, \ds^{j-1}u_0^\mathrm{in}),
\end{align*}
where $A_j=A_j(u_0^\mathrm{in})$ is an invertible matrix. 
In this situation, by adding $\frac{(s-1)^j}{j!}a_j$ with an appropriate constant vector $a_j$ to the initial data $u_0^\mathrm{in}$ 
we can modify $u_j^\mathrm{in}|_{s=1}$ for whatever we want keeping the values of $u_k^\mathrm{in}|_{s=1}$ unchanged for $k=0,1,\ldots,j-1$. 
For our problem \eqref{HP2}--\eqref{BVP2}, roughly speaking, $u_j^\mathrm{in}=(\dt^ju)|_{t=0}$ $(j=1,2,3,\ldots)$ have the form 
\begin{align*}
u_j^\mathrm{in}|_{s=1}
&= F_j(u_0^\mathrm{in}|_{s=1}, \ldots, (\ds^j u_0^\mathrm{in})|_{s=1}, \theta_0(u_0^\mathrm{in}), \ldots, \theta_j(u_0^\mathrm{in})) \\
&= A_j(u_0^\mathrm{in}|_{s=1},\theta_0(u_0^\mathrm{in}))(\ds^ju_0^\mathrm{in})|_{s=1} \\
&\quad\;
 + G_j(u_0^\mathrm{in}|_{s=1}, \ldots, (\ds^{j-1}u_0^\mathrm{in})|_{s=1},\theta_0(u_0^\mathrm{in}), \ldots, \theta_j(u_0^\mathrm{in})),
\end{align*}
where $\theta_0,\ldots,\theta_j$ are functionals defined in appropriate weighted Sobolev spaces, 
so that $u_j^\mathrm{in}|_{s=1}$ is not linear with respect to $(\ds^ju_0^\mathrm{in})|_{s=1}$ due to these nonlocal terms 
$\theta_0(u_0^\mathrm{in}), \ldots, \theta_j(u_0^\mathrm{in})$. 
This nonlocal character is caused by the tension of the string. 
Moreover, if we add $\frac{(s-1)^j}{j!}a_j$ to the initial data $u_0^\mathrm{in}$ to modify the value of $(\ds^ju_0^\mathrm{in})|_{s=1}$, 
then all the values of $u_k^\mathrm{in}|_{s=1}$ for $k=0,1,\ldots,j-1$ would change due to the nonlocal terms. 
Therefore, the standard technique for approximation of the initial data cannot be applicable directly to our problem.

\medskip
Our strategy to overcome this difficulty is to approximate the initial data in the form 
\[
v_0^{\mathrm{in}}(s) + \sum_{j=0}^{m-2}\delta^{-j}\psi_j^\delta(s)a_j + \delta^{-(m-2)}\psi_{m-1}^\varepsilon(s)a_{m-1},
\]
where $v_0^{\mathrm{in}}$ is a smooth approximation of the initial data $u_0^{\mathrm{in}}$, $\psi_j^\delta$ is defined for $\delta>0$ by 
$\psi_j^\delta(s)=\frac{(s-1)^j}{j!}\psi\left(\frac{s-1}{\delta}\right)$ with a cut-off function $\psi\in C_0^\infty(\mathbb{R})$ satisfying 
$\psi(s)=1$ for $|s|\leq1$, $\delta$ and $\varepsilon$ are small positive parameters; 
this function $\psi_j^\delta$ is naturally extended for $\delta\leq0$ as $\psi_j^\delta(s)\equiv0$. 
Then, we can regard the compatibility conditions as equations for $a_0,a_1,\ldots,a_{m-1}$, 
which are nonlinear contrary to the standard initial boundary value problems for nonlinear hyperbolic systems. 
However, the nonlinearity becomes as weak as we want by choosing the parameters $\delta$ and $\varepsilon$ sufficiently small. 
Therefore, we can apply the implicit function theorem to solve the equations for $a_0,a_1,\ldots,a_{m-1}$ 
if $v_0^{\mathrm{in}}$ is sufficiently close to $u_0^{\mathrm{in}}$ in an appropriate norm. 
To carry out this strategy, 
we use the continuity of the map $\mathbb{R}\ni\delta\mapsto\psi_j^\delta\in H^j$ as well as the discontinuity of the map 
$\mathbb{R}\ni\delta\mapsto\psi_j^\delta\in H^{j+1}$ at $\delta=0$. 
We also need to study the nonlocal terms $\theta_0(u_0^\mathrm{in}), \ldots, \theta_j(u_0^\mathrm{in})$ in detail.

\medskip
The contents of this paper are as follows. 
In Section \ref{sect:pre}, which is a preliminary section, we introduce weighted Sobolev spaces $X^m$ and $Y^m$ for non-negative integer $m$ 
and present basic properties of these spaces and related calculus inequalities. 
We also present estimates of solutions to a two-point boundary value problem related to \eqref{BVP2}. 
In Section \ref{sect:CC} we state precisely the compatibility conditions for the initial boundary value problem \eqref{Eq}--\eqref{IC} 
and for the problem \eqref{HP2}--\eqref{IC2}. 
Then, we state Propositions \ref{prop:AppID1} and \ref{prop:AppID2}, which ensure that the initial data for the problem \eqref{Eq}--\eqref{IC} 
and for the problem \eqref{HP2}--\eqref{IC2} can be approximated by smooth initial data satisfying higher order compatibility conditions. 
Since our proofs of these propositions are technical and long, we postpone them until Sections \ref{sect:AppID1} and \ref{sect:AppID2}. 
In Section \ref{sect:QLS} we derive a quasilinear system of equations for unknowns $(\bm{x},\bm{y},\tau,\nu)$, 
where $\bm{y}$ and $\nu$ are defined by \eqref{defNU}, and state Proposition \ref{prop:QLWP} which ensures the well-posedness locally in time 
of the initial boundary value problem for the quasilinear system in the class $\bm{x},\dot{\bm{x}},\bm{y} \in\mathscr{X}_T^m$ and 
$\tau',\dot{\tau}',\nu'\in\mathscr{X}_T^{m-1,*}$ with $m\geq4$. 
The proposition is proved by applying our previous result \cite[Theorem 2.8]{IguchiTakayama2023-2} on the well-posedness of the initial boundary value problem 
to a linearized system for the nonlinear problem \eqref{Eq}--\eqref{IC} and by a standard fixed point argument. 
In Section \ref{sect:Proof} we prove our main result in this paper, that is, Theorem \ref{th:main}. 
The proof is divided into three cases: the case $m\geq6$, the case $m=5$, and the critical case $m=4$. 
The case $m\geq6$ is proved by showing that the solution for the quasilinear system obtained in Section \ref{sect:QLS} with appropriately prepared initial data 
is in fact a solution of the reduced nonlinear problem \eqref{HP2}--\eqref{IC2} and satisfies the desired regularity. 
In order to show the theorem in the case $m=4$ and $5$, we first approximate the initial data by regular ones satisfying higher order compatibility conditions 
by using Proposition \ref{prop:AppID2} and then construct corresponding approximate solutions by using the result in the first case. 
Thanks to the a priori estimate obtained in \cite[Theorem 2.1]{IguchiTakayama2024} with a slight improvement
and a standard compactness argument, 
we can show that the approximate solutions converge to the desired solution. 
A difficulty in showing the strong continuity in time of the solution appears in the critical case $m=4$. 
To overcome this difficulty, we make use of the energy estimate. 
In Section \ref{sect:AppID1} we give a proof of Proposition \ref{prop:AppID1} on an approximation of the initial data for the problem \eqref{HP2}--\eqref{IC2}. 
In Section \ref{sect:AppID2} we give a proof of Proposition \ref{prop:AppID2} on an approximation of the initial data for the problem 
\eqref{Eq}--\eqref{IC} by modifying the calculations in Section \ref{sect:AppID1} in order that the approximate initial data also satisfy the constraints 
$|\bm{x}_0^{\mathrm{in}\prime}(s)|=1$ and $\bm{x}_0^{\mathrm{in}\prime}(s)\cdot\bm{x}_1^{\mathrm{in}\prime}(s)=1$ for $0<s<1$. 
Finally, in Section \ref{sect:APE} we give a remark on the improvement of the a priori estimate of the solution. 

\medskip
\noindent
{\bf Notation}. \ 
For $1\leq p\leq\infty$, we denote by $L^p$ the Lebesgue space on the open interval $(0,1)$. 
For non-negative integer $m$, we denote by $H^m$ the $L^2$ Sobolev space of order $m$ on $(0,1)$. 
The norm of a Banach space $B$ is denoted by $\|\cdot\|_B$. 
The inner product in $L^2$ is denoted by $(\cdot,\cdot)_{L^2}$. 
We put $\dt=\frac{\partial}{\partial t}$ and $\ds=\frac{\partial}{\partial s}$. 
The norm of a weighted $L^p$ space with a weight $s^\alpha$ is denoted by $\|s^\alpha u\|_{L^p}$, so that 
$\|s^\alpha u\|_{L^p}^p=\int_0^1s^{\alpha p}|u(s)|^p \mathrm{d}s$ for $1\leq p<\infty$. 
It is sometimes denoted by $\|\sigma^\alpha u\|_{L^p}$, too. 
This would cause no confusion. 
For a map $F$ defined in an open set $\Omega$ in a Banach space $B_1$ and with a value in a Banach space $B_2$, 
we denote by $DF(u)$ the Fr\'echet derivative of $F$ at $u\in\Omega$. 
The application of the linear bounded operator $Df(u)$ to $\hat{u}\in B_1$ is denoted by $DF(u)[\hat{u}]$. 
$[P,Q]=PQ-QP$ denotes the commutator. 
We denote by $C(a_1, a_2, \ldots)$ a positive constant depending on $a_1, a_2, \ldots$. 
$f\lesssim g$ means that there exists a non-essential positive constant $C$ such that $f\leq Cg$ holds. 
$f\simeq g$ means that $f\lesssim g$ and $g\lesssim f$ hold. 
$a_1 \vee a_2 = \max\{a_1,a_2\}$.

\medskip
\noindent
{\bf Acknowledgement} \\
T. I. is partially supported by JSPS KAKENHI Grant Number JP23K22404.

%----------------------------------------------------------------------------------------------------------------------
%----------------------------------------------------------------------------------------------------------------------
\section{Preliminaries}\label{sect:pre}
In this preliminary section, we recall the definition of the weighted Sobolev spaces $X^m$ and $Y^m$ and present basic properties of these spaces 
and related calculus inequalities. 
They were proved in Takayama \cite{Takayama2018} and Iguchi and Takayama \cite{IguchiTakayama2024, IguchiTakayama2023-2}.

%----------------------------------------------------------------------------
\subsection{Function spaces and known properties}
For a non-negative integer $m$, following Reeken \cite{Reeken1979-1, Reeken1979-2}, Takayama \cite{Takayama2018}, and Iguchi and Takayama \cite{IguchiTakayama2024}, 
we define a weighted Sobolev space $X^m$ as a set of all function $u=u(s)\in L^2$ equipped with a norm $\|\cdot\|_{X^m}$ defined by 
\[
\|u\|_{X^m}^2 =
\begin{cases}
 \displaystyle
 \|u\|_{H^k}^2 + \sum_{j=1}^k\|s^j\ds^{k+j}u\|_{L^2}^2 &\mbox{for}\quad m=2k, \\
 \displaystyle
 \|u\|_{H^k}^2 + \sum_{j=1}^{k+1}\|s^{j-\frac12}\ds^{k+j}u\|_{L^2}^2 &\mbox{for}\quad m=2k+1.
\end{cases}
\]

\noindent
Preston \cite{Preston2011} used a weighted Sobolev space $N_m$ equipped with a norm $\|\cdot\|_{N_m}$ defined by 
$\|u\|_{N_m}^2 = \sum_{j=0}^m\|s^\frac{j}{2}\ds^ju\|_{L^2}^2$. 
Apparently the weights of these norms $\|\cdot\|_{X^m}$ and $\|\cdot\|_{N_m}$ are different, they are actually equivalent so that we have $X^m=N_m$. 
We will use another weighted Sobolev space $Y^m$ equipped with a norm $\|\cdot\|_{Y^m}$, which is defined so that 
$\|u\|_{X^{m+1}}^2=\|u\|_{L^2}^2+\|u'\|_{Y^m}^2$ holds. 
For more details, we refer to \cite{IguchiTakayama2024}.

For a function $u=u(s,t)$ depending also on time $t$ and for integers $m$ and $l$ satisfying $0\leq l\leq m$, 
we introduce norms $\opnorm{\cdot}_{m,l}$ and $\opnorm{\cdot}_{m,l}^\dag$ by 
$\opnorm{u(t)}_{m,l}^2 = \sum_{j=0}^l \|\dt^j u(t)\|_{X^{m-j}}^2$ and $\bigl( \opnorm{u(t)}_{m,l}^\dag \bigr)^2 = \sum_{j=0}^l \|\dt^j u(t)\|_{Y^{m-j}}^2$, 
respectively, and put $\opnorm{\cdot}_m=\opnorm{\cdot}_{m,m}$, $\opnorm{\cdot}_{m,*}=\opnorm{\cdot}_{m,m-1}$, and 
$\opnorm{\cdot}_m^\dag=\opnorm{\cdot}_{m,m}^\dag$, $\opnorm{\cdot}_{m,*}^\dag=\opnorm{\cdot}_{m,m-1}^\dag$. 
Corresponding to these norms, we define the spaces 
$\mathscr{X}_T^{m,l} = \bigcap_{j=0}^l C^j([0,T],X^{m-j})$, $\mathscr{X}_T^m = \mathscr{X}_T^{m,m}$, $\mathscr{X}_T^{m,*} = \mathscr{X}_T^{m,m-1}$, and similarly, 
$\mathscr{Y}_T^{m,l} = \bigcap_{j=0}^l C^j([0,T],Y^{m-j})$, $\mathscr{Y}_T^m = \mathscr{Y}_T^{m,m}$, $\mathscr{Y}_T^{m,*} = \mathscr{Y}_T^{m,m-1}$. 
We will solve the initial boundary value problem \eqref{Eq}--\eqref{IC} in the class $(\bm{x},\tau')\in\mathscr{X}_T^m\times\mathscr{X}_T^{m-1,*}$. 
However, in the critical case $m=4$, we need to use a weaker norm than $\opnorm{\cdot}_{3,*}$ to evaluate $\tau'$. 
For $\epsilon>0$, we introduce norms $\|\cdot\|_{X_\epsilon^k}$ for $k=1,2,3$ as 
\[
\|u\|_{X_\epsilon^k}^2 =
\begin{cases}
 \|s^\epsilon u\|_{L^\infty}^2 + \|s^{\frac12+\epsilon}u'\|_{L^2}^2 &\mbox{for}\quad k=1, \\
 \|u\|_{L^\infty}^2 + \|s^\epsilon u'\|_{L^2}^2 + \|s^{1+\epsilon} u''\|_{L^2}^2 &\mbox{for}\quad k=2, \\
 \|u\|_{L^\infty}^2 + \|s^\epsilon u'\|_{L^\infty}^2 + \|s^{\frac12+\epsilon}u''\|_{L^2}^2 + \|s^{\frac32+\epsilon}u'''\|_{L^2}^2
  &\mbox{for}\quad k=3,
\end{cases}
\]
and put $\opnorm{u(t)}_{3,*,\epsilon}^2 = \|u(t)\|_{X_\epsilon^3}^2 + \|\dt u(t)\|_{X_\epsilon^2}^2 + \|\dt^2 u(t)\|_{X_\epsilon^1}^2$ and 
$\mathscr{X}_T^{3,*,\epsilon}=\bigcap_{j=0}^2C^j([0,T];X_\epsilon^{3-j})$. 
As we will see in Lemma \ref{lem:NormEquiv}, for each $\epsilon\in(0,\frac12)$ the norm $\|\cdot\|_{X_\epsilon^k}$ satisfies the equivalence 
\begin{equation}\label{NormEquiv}
\|u\|_{X_\epsilon^k}^2 \simeq
\begin{cases}
 \|u\|_{L^2}^2 + \|s^{\frac12+\epsilon}u'\|_{L^2}^2 &\mbox{for}\quad k=1, \\
 \|u\|_{L^2}^2 + \|s^\epsilon u'\|_{L^2}^2 + \|s^{1+\epsilon} u''\|_{L^2}^2 &\mbox{for}\quad k=2, \\
 \|u\|_{H^1}^2 + \|s^{\frac12+\epsilon}u''\|_{L^2}^2 + \|s^{\frac32+\epsilon}u'''\|_{L^2}^2
  &\mbox{for}\quad k=3.
\end{cases}
\end{equation}

The weighted Sobolev space $X^m$ is characterize as follows: 
Let $D$ be the unit disc in $\mathbb{R}^2$ and $H^m(D)$ the $L^2$ Sobolev space of order $m$ on $D$. 
For a function $u$ defined in the open interval $(0,1)$, we define $u^\sharp(x_1,x_2)=u(x_1^2+x_2^2)$ which is a function on $D$.

\begin{lemma}[{\cite[Proposition 3.2]{Takayama2018}}]\label{lem:NormEq}
Let $m$ be a non-negative integer. 
The map $X^m\ni u \mapsto u^\sharp \in H^m(D)$ is bijective and it holds that $\|u\|_{X^m} \simeq \|u^\sharp\|_{H^m(D)}$ 
for any $u\in X^m$. 
\end{lemma}

\begin{lemma}[{\cite[Lemma 4.5]{IguchiTakayama2024}}]\label{lem:embedding2}
For a non-negative integer $m$, we have $\|su'\|_{X^m} \lesssim \|u\|_{X^{m+1}}$, $\|u'\|_{X^m} \leq \|u\|_{X^{m+2}}$, 
and $\|\ds^mu\|_{L^\infty} \lesssim \|u\|_{X^{2m+2}}$. 
\end{lemma}

\begin{lemma}[{\cite[Lemma 4.7]{IguchiTakayama2024}}]\label{lem:algebraX}
It holds that  $\|uv\|_{X^m} \lesssim \|u\|_{X^{m\vee2}}\|v\|_{X^m}$ for $m=0,1,2,\ldots$. 
\end{lemma}

\begin{lemma}[{\cite[Lemma 9.1]{IguchiTakayama2024}}]\label{lem:EstAu}
If $\tau|_{s=0}$, then we have 
\[
\|(\tau u')'\|_{X^m} \lesssim
\begin{cases}
 \min\{ \|\tau'\|_{L^\infty} \|u\|_{X^2}, \|\tau'\|_{X^1} \|u\|_{X^3} \} &\mbox{for}\quad m=0, \\
 \min\{ \|\tau'\|_{X^{m \vee 2}} \|u\|_{X^{m+2}}, \|\tau'\|_{X^m} \|u\|_{X^{m+2 \vee 4}} \} &\mbox{for}\quad m=0,1,2,\ldots.
\end{cases}
\]
\end{lemma}

\begin{lemma}[{\cite[Lemma 3.9]{IguchiTakayama2023-2}}]\label{lem:EstAu2}
Let $m$ and $j$ be integers such that $1\leq j\leq m$. 
If $\tau|_{s=0}$, then we have 
\[
\opnorm{ (\tau u')'(t) }_{m,j} \lesssim
\begin{cases}
 \min\{ \opnorm{\tau'(t)}_1 \opnorm{u(t)}_{4,1}, \opnorm{\tau'(t)}_{2,1} \opnorm{u(t)}_{3,1}  \} &\mbox{for}\quad m=1, \\
 \opnorm{\tau'(t)}_{m,j} \opnorm{u(t)}_{m+2,j} &\mbox{for}\quad m=2,3,\ldots.
\end{cases}
\]
\end{lemma}

\begin{lemma}[{\cite[Lemma 4.10]{IguchiTakayama2024}}]\label{lem:CalIneq1}
It holds that 
\[
\|u'v'\|_{Y^m} \lesssim
\begin{cases}
 \|u\|_{X^{m+2}} \|v\|_{X^{m+2}} &\mbox{for}\quad m=0,1, \\
 \|u\|_{X^{m+1\vee4}} \|v\|_{X^{m+1}} &\mbox{for}\quad m=0,1,2,\ldots.
\end{cases}
\]
\end{lemma}

\begin{lemma}[{\cite[Lemma 3.13]{IguchiTakayama2023-2}}]\label{lem:CalIneqY1}
For a positive integer $m$, we have 
\[
\opnorm{(u'v')(t)}_m^\dag \lesssim 
\begin{cases}
 \opnorm{u(t)}_{3,1} \opnorm{v(t)}_{3,1} &\mbox{for}\quad m=1, \\
 \opnorm{u(t)}_{m+1\vee4,m} \opnorm{v(t)}_{m+1,m} &\mbox{for}\quad m\geq1.
\end{cases}
\]
\end{lemma}

\begin{lemma}[{\cite[Lemma 4.11]{IguchiTakayama2024}}]\label{lem:CalIneq2}
If $\tau|_{s=0}=0$, then we have 
\[
\|\tau u''v''\|_{Y^m} \lesssim
\begin{cases}
 \|\tau'\|_{L^2} \|u\|_{X^4} \|v\|_{X^3} &\mbox{for}\quad m=0, \\
 \|\tau'\|_{L^\infty} \min\{ \|u\|_{X^4} \|v\|_{X^2}, \|u\|_{X^3} \|v\|_{X^3} \} &\mbox{for}\quad m=0, \\
 \min\{ \|\tau'\|_{L^2} \|u\|_{X^4},\|\tau'\|_{L^\infty} \|u\|_{X^3}\} \|v\|_{X^4} &\mbox{for}\quad m=1, \\
 \|\tau'\|_{L^\infty \cap X^{m-1}} \|u\|_{X^{m+2}} \|v\|_{X^{m+2}} &\mbox{for}\quad m\geq2.
\end{cases}
\]
\end{lemma}

\begin{lemma}[{\cite[Lemma 3.15]{IguchiTakayama2023-2}}]\label{lem:CalIneqY2}
Let $m$ be a positive integer. 
If $\tau|_{s=0}=0$, then we have 
\[
\opnorm{\tau u''v''}_m ^\dag \lesssim
\begin{cases}
 (\|\tau'\|_{L^\infty} + \|\dt\tau'\|_{L^2}) \opnorm{u}_{4,1} \opnorm{v}_{3,1} &\mbox{for}\quad m=1, \\
 \|(\tau',\dt\tau')\|_{L^2} \opnorm{u}_{4,1} \opnorm{v}_{4,1} &\mbox{for}\quad m=1, \\
 (\|\dt^{m-2}\tau'\|_{L^\infty} + \opnorm{\tau'}_{m-1} + \|\dt^m\tau'\|_{L^2}) \opnorm{u}_{m+2,m} \opnorm{v}_{m+2,m} &\mbox{for}\quad m\geq2.
\end{cases}
\]
\end{lemma}

\begin{lemma}[{\cite[Lemma 3.16]{IguchiTakayama2023-2}}]\label{lem:CalIneqY3}
If $\tau|_{s=0}=0$, then we have 
\[
\|s^\frac12\tau u''v''\|_{L^1} 
\lesssim \min\{ \|\tau'\|_{L^\infty}\|u\|_{X^2}\|v\|_{X^3}, \|\tau'\|_{L^2}\|u\|_{X^2}\|v\|_{X^4}, \|\tau'\|_{L^2}\|u\|_{X^3}\|v\|_{X^3} \}.
\]
\end{lemma}

%----------------------------------------------------------------------------
\subsection{Further properties of the weighted Sobolev spaces}
By the definition of the norms $\|\cdot\|_{X^m}$ and $\|\cdot\|_{Y^m}$, we see easily the following lemma.

\begin{lemma}\label{lem:NormRelation}
For a non-negative integer $m$, it holds that $\|u\|_{Y^m} \leq \|u\|_{X^m} \leq \|u\|_{Y^{m+1}}$ and $\|su'\|_{X^m} \lesssim \|u\|_{Y^{m+1}}$. 
\end{lemma}

In view of Lemma \ref{lem:NormEq} together with the standard tame estimates in the Sobolev space $H^m(D)$, we have the following tame estimates 
in the weighted Sobolev space $X^m$.

\begin{lemma}\label{lem:tame1}
It holds that $\|uv\|_{X^m} \lesssim \|u\|_{L^\infty}\|v\|_{X^m}+\|v\|_{L^\infty}\|u\|_{X^m}$ for $m=0,1,2,\ldots$. 
\end{lemma}

\begin{lemma}\label{lem:tame2}
Let $m$ be a non-negative integer, $\Omega$ an open set in $\mathbb{R}^N$, $K$ a compact set in $\Omega$, and $F\in C^m(\Omega)$. 
There exists a positive constant $C=C(m,N,K)$ such that if $u\in X^m$ takes its value in $K$, then we have 
$\|F(u)\|_{X^m} \leq C\|F\|_{C^m(K)}(1+\|u\|_m)$. 
\end{lemma}

\begin{lemma}\label{lem:NormEquiv}
For each $\epsilon\in(0,\frac12)$, the equivalence \eqref{NormEquiv} holds. 
\end{lemma}

\begin{proof}
For $\alpha>0$, we see that 
\begin{align*}
s^\alpha|u(s)|^2 &= \int_0^s(\sigma^\alpha|u(\sigma)|^2)'\mathrm{d}\sigma \\
&\lesssim \int_0^1 (\sigma^{\alpha-1}|u(\sigma)|^2 + \sigma^\alpha|u(\sigma)||u'(\sigma)|)\mathrm{d}\sigma \\
&\lesssim \|\sigma^\frac{\alpha-1}{2}u\|_{L^2}^2 + \|\sigma^\frac{\alpha+1}{2}u'\|_{L^2}^2.
\end{align*}
Therefore, putting $\alpha=2\beta+1$ with $\beta>-\frac12$ we have 
\[
\|s^{\beta+\frac12}u\|_{L^\infty} \lesssim \|s^\beta u\|_{L^2} + \|s^{\beta+1}u'\|_{L^2}.
\]

\noindent
Plugging $\beta=-\frac12+\epsilon$ in the above inequality, we obtain 
$\|s^\epsilon u\|_{L^\infty} \lesssim \|s^{-\frac12+\epsilon} u\|_{L^2} + \|s^{\frac12+\epsilon}u'\|_{L^2}$. 
Here, by integration by parts we see that 
\begin{align*}
\|s^{-\frac12+\epsilon} u\|_{L^2}^2
&= \int_0^1 \left(\frac{1}{2\epsilon}s^{2\epsilon}\right)'|u(s)|^2\mathrm{d}s \\
&\leq \frac{1}{2\epsilon}|u(1)|^2 + \frac{1}{\epsilon}\int_0^1s^{2\epsilon}|u(s)||u'(s)|\mathrm{d}s \\
&\lesssim |u(1)|^2+\|s^{-\frac12+\epsilon} u\|_{L^2}\|s^{\frac12+\epsilon} u'\|_{L^2}.
\end{align*}
These estimates together with the Sobolev embedding $|u(1)|\lesssim\|su\|_{H^1}$ imply 
$\|s^\epsilon u\|_{L^\infty} \lesssim \|u\|_{L^2}+\|s^{\frac12+\epsilon} u'\|_{L^2}$. 
Conversely, by integration by parts we see that 
\begin{align*}
\|u\|_{L^2}^2
&=|u(1)|^2 - 2\int_0^1su(s)u'(s)\mathrm{d}s \\
&\leq \|s^\epsilon u\|_{L^\infty}^2 + 2\|u\|_{L^2}\|su'\|_{L^2},
\end{align*}
which implies $\|u\|_{L^2} \lesssim \|s^\epsilon u\|_{L^\infty}+\|s^{\frac12+\epsilon}u'\|_{L^2}$. 
Therefore, we obtain the equivalence \eqref{NormEquiv} in the case $k=1$. 
The other cases can be proved in the same way so we omit it. 
\end{proof}

By Lemma \ref{lem:algebraX}, the weighted Sobolev space $X^m$ is an algebra if $m\geq2$. 
A similar algebraic property of the weighted Sobolev space $Y^m$ can follow from Lemma \ref{lem:CalIneq1}. 
To show this, we first prepare the following lemma.

\begin{lemma}\label{lem:NormEqui}
For $x(s)=-\int_s^1u(\sigma)\mathrm{d}\sigma$, we have $\|x\|_{L^2}\leq2\|su\|_{L^2}$. 
Particularly, for any non-negative integer $m$, we have $\|x\|_{X^{m+1}}\simeq\|u\|_{Y^m}$. 
\end{lemma}

\begin{proof}
Since $x(1)=0$ and $x'(s)=u(s)$, by integration by parts, we see that 
\begin{align*}
\int_0^1|x(s)|^2\mathrm{d}s
&= -2\int_0^1sx(s)x'(s)\mathrm{d}s \\
&\leq 2\|x\|_{L^2}\|su\|_{L^2},
\end{align*}
which yields $\|x\|_{L^2}\leq2\|su\|_{L^2}$. 
Particularly, we have $\|x\|_{L^2}\leq2\|u\|_{Y^0}$. 
This together with the identity $\|x\|_{X^{m+1}}^2=\|x\|_{L^2}^2+\|u\|_{Y^m}^2$ implies the equivalence $\|x\|_{X^{m+1}}\simeq\|u\|_{Y^m}$. 
\end{proof}

\begin{lemma}\label{lem:algebraY}
Let $m$ be a non-negative integer. 
Then, we have 
\[
\|uv\|_{Y^m} \lesssim
\begin{cases}
 \|u\|_{Y^{m+1}}\|v\|_{Y^{m+1}} &\mbox{for}\quad m=0,1, \\
 \|u\|_{Y^{m\vee3}}\|v\|_{Y^m} &\mbox{for}\quad m=0,1,2,\ldots.
\end{cases}
\]
\end{lemma}

\begin{proof}
Put $x(s)=-\int_s^1u(\sigma)\mathrm{d}\sigma$ and $y(s)=-\int_s^1v(\sigma)\mathrm{d}\sigma$. 
Then, by Lemma \ref{lem:CalIneq1} we see that 
\begin{align*}
\|uv\|_{Y^m}
&= \|x'y'\|_{Y^m} \\
&\lesssim 
 \begin{cases}
  \|x\|_{X^{m+2}}\|y\|_{X^{m+2}} &\mbox{for}\quad m=0,1, \\
  \|x\|_{X^{m+1\vee4}}\|y\|_{X^{m+1}} &\mbox{for}\quad m=0,1,2,\ldots,
 \end{cases}
\end{align*}
which together with Lemma \ref{lem:NormEqui} implies the desired estimate. 
\end{proof}

By Lemmas \ref{lem:CalIneq2} and \ref{lem:NormEqui}, we can also obtain the following lemma. 
Since the proof is almost the same as that of Lemma \ref{lem:algebraY}, we omit it.

\begin{lemma}\label{lem:CalIneq2a}
Let $m$ be a non-negative integer. 
If $\tau|_{s=0}=0$, then, we have 
\[
\|\tau u'v'\|_{Y^m} \lesssim
 \begin{cases}
  \|\tau'\|_{L^2}\|u\|_{Y^3}\|v\|_{Y^2} &\mbox{for}\quad m=0, \\
  \|\tau'\|_{L^\infty}\|u\|_{Y^2}\|v\|_{Y^2} &\mbox{for}\quad m=0, \\
  \|\tau'\|_{L^\infty}\|u\|_{Y^3}\|v\|_{Y^2} &\mbox{for}\quad m=1, \\
  \|\tau'\|_{L^\infty\cap X^{m-1}}\|u\|_{Y^{m+1}}\|v\|_{Y^{m+1}} &\mbox{for}\quad m\geq2.
 \end{cases}
\]
\end{lemma}

We will use an averaging operator $\mathscr{M}$ introduced in \cite{IguchiTakayama2024}, which is defined by 
\begin{equation}\label{AvOp}
(\mathscr{M}u)(s) = \frac1s\int_0^su(\sigma)\mathrm{d}\sigma.
\end{equation}

\begin{lemma}[{\cite[Corollary 4.13]{IguchiTakayama2024}}]\label{lem:AvOp}
Let $j$ be a non-negative integer, $1\leq p\leq \infty$, and $\beta<j+1-\frac1p$. 
Then, we have $\|s^\beta\ds^j(\mathscr{M}u)\|_{L^p} \leq (j+1-\beta-\frac1p)^{-1}\|s^\beta\ds^j u\|_{L^p}$. 
Particularly, $\|\mathscr{M}u\|_{X^m} \leq 2\|u\|_{X^m}$ holds for $m=0,1,2,\ldots$. 
\end{lemma}

\begin{lemma}\label{lem:CalIneq3}
Let $m$ be a non-negative integer and $\epsilon>0$. 
If $\tau|_{s=0}=0$, then we have 
\[
\begin{cases}
\|(\tau u)''\|_{Y^m} \lesssim \|\tau'\|_{X^{m+1}}\|u\|_{Y^{m+2\vee3}}, \\
\|s^\epsilon(\tau u)''\|_{Y^0} \lesssim (\|\tau'\|_{L^\infty}+\|\tau''\|_{Y^0})\|u\|_{Y^2}.
\end{cases}
\]
\end{lemma}

\begin{proof}
We first note the identity $\tau(s)=s(\mathscr{M}\tau')(s)$. 
By lemmas \ref{lem:algebraX}, \ref{lem:NormRelation}, and \ref{lem:AvOp}, we see that 
\begin{align*}
\|(\tau u)''\|_{Y^m}
&\leq \|(\tau u)'\|_{X^{m+1}} \\
&\leq \|\tau'u\|_{X^{m+1}} + \|(\mathscr{M}\tau')su'\|_{X^{m+1}} \\
&\lesssim \|\tau'\|_{X^{m+1}}\|u\|_{X^{m+1\vee2}}+\|\mathscr{M}\tau'\|_{X^{m+1}}\|su'\|_{X^{m+1\vee2}} \\
&\lesssim \|\tau'\|_{X^{m+1}}\|u\|_{Y^{m+2\vee3}},
\end{align*}
which gives the first estimate of the lemma. 
We recall the embedding $\|s^\epsilon u\|_{L^\infty} \lesssim \|u\|_{X^1}$, which was given by \cite[Lemma 4.3]{IguchiTakayama2024}. 
Since $|\tau(s)| \leq s\|\tau'\|_{L^\infty}$, we see that 
\begin{align*}
\|s^\epsilon(\tau u)''\|_{Y^0}
&\leq \|s^{\frac12+\epsilon}\tau''u\|_{L^2} + 2\|s^\frac12\tau'u'\|_{L^2} + \|s^\frac12\tau u''\|_{L^2} \\
&\leq \|s^\frac12\tau''\|_{L^2}\|s^\epsilon u\|_{L^\infty} + 2\|\tau'\|_{L^\infty}( \|s^\frac12 u'\|_{L^2}+\|s^\frac32 u''\|_{L^2}) \\
&\lesssim \|\tau''\|_{Y^0}\|u\|_{X^1} + \|\tau'\|_{L^\infty}\|u\|_{Y^2},
\end{align*}
which together with $\|u\|_{X^1}\leq \|u\|_{Y^2}$ gives the second estimate of the lemma. 
\end{proof}

%----------------------------------------------------------------------------
\subsection{Two-point boundary value problem}
In view of \eqref{BVP} we consider the two-point boundary value problem 
\begin{equation}\label{TBVP}
\begin{cases}
 -\tau''+|\bm{v}|^2\tau = h \quad\mbox{in}\quad (0,1), \\
 \tau(0)=0, \quad \tau'(1)=a,
\end{cases}
\end{equation}
where $\bm{v}(s)$ and $h(s)$ are given functions and $a$ is a constant.

\begin{lemma}[{\cite[Lemmas 3.5 and 3.7]{IguchiTakayama2024}}]\label{lem:EstSolBVP3}
For any $M>0$ there exists a constant $C=C(M)>0$ such that if $\|\bm{v}\|_{Y^0} \leq M$, then the solution $\tau$ 
to the two-point boundary value problem \eqref{TBVP} satisfies 
\[
\begin{cases}
 \|\tau\|_{L^\infty} \leq C(|a|+\|sh\|_{L^1}), \\
 \|s^\alpha\tau'\|_{L^p} \leq C(|a|+\|s^{\alpha+\frac{1}{p}}h\|_{L^1})
\end{cases}
\]
for any $p\in[1,\infty]$ and any $\alpha\geq0$ satisfying $\alpha+\frac{1}{p}\leq1$. 
\end{lemma}

\begin{lemma}\label{lem:HEstBVP}
For any $M>0$ and any integer $m\geq1$ there exists a constant $C=C(M,m)>0$ such that if $\|\bm{u}\|_{Y^{m\vee3}}\leq M$, 
then the solution $\tau$ to the two-point boundary value problem \eqref{TBVP} with $\bm{v}=\bm{u}'$ satisfies 
\[
\|\tau'\|_{X^m} \leq C(|a|+\|h\|_{Y^{m-1}}).
\]
\end{lemma}

\begin{proof}
By the first equation in \eqref{TBVP} with $\bm{v}=\bm{u}'$, we see that 
\begin{align*}
\|\tau'\|_{X^m}
&\leq \|\tau'\|_{L^2} + \|\tau''\|_{Y^{m-1}} \\
&\leq \|\tau'\|_{L^2} + \||\bm{u}'|^2\tau\|_{Y^{m-1}} + \|h\|_{Y^{m-1}}.
\end{align*}
Here, by Lemma \ref{lem:CalIneq2a} we have 
\[
\||\bm{u}'|^2\tau\|_{Y^{m-1}} \lesssim
 \begin{cases}
  \|\tau'\|_{L^2} \|\bm{u}\|_{Y^3}^2 &\mbox{for}\quad m=1, \\
  \|\tau'\|_{L^\infty} \|\bm{u}\|_{Y^3}^2 &\mbox{for}\quad m=2, \\
  \|\tau'\|_{L^\infty\cap X^{m-2}} \|\bm{u}\|_{Y^m}^2 &\mbox{for}\quad m\geq3.
 \end{cases}
\]
Moreover, by Lemma \ref{lem:EstSolBVP3} together with $\|\bm{u}'\|_{Y^0} \leq \|\bm{u}\|_{Y^2}$, $\|s^\frac12 h\|_{L^1} \leq \|h\|_{Y^0}$, and 
$\|h\|_{L^1}\leq\|h\|_{Y^1}$ we get 
\[
\begin{cases}
 \|\tau'\|_{L^2} \leq C(\|\bm{u}\|_{Y^2})(|a|+\|h\|_{Y^0}), \\
 \|\tau'\|_{L^\infty} \leq C(\|\bm{u}\|_{Y^2})(|a|+\|h\|_{Y^1}).
\end{cases}
\]
Combining these estimates, we obtain the desired estimate inductively on $m$. 
\end{proof}

%----------------------------------------------------------------------------------------------------------------------
%----------------------------------------------------------------------------------------------------------------------
\section{Compatibility conditions}\label{sect:CC}
Let $(\bm{x},\tau)$ be a smooth solution of the problem \eqref{HP2}--\eqref{IC2} and put 
$\bm{x}_j^\mathrm{in}=(\dt^j\bm{x})|_{t=0}$ and $\tau_j^\mathrm{in}=(\dt^j\tau)|_{t=0}$ for $j=0,1,2,\ldots$. 
These initial values are determined, at least formally, from the initial data $(\bm{x}_0^\mathrm{in},\bm{x}_1^\mathrm{in})$ as follows. 
As was explained in Remark \ref{re:MainTh} (1), the initial tension $\tau_0^\mathrm{in}$ is determined as a unique solution 
of the two-point boundary value problem \eqref{BVP2} in the case $t=0$. 
Then, putting $t=0$ in the first equation in \eqref{HP2} we have $\bm{x}_2^\mathrm{in}=(\tau_0^\mathrm{in}\bm{x}_0^{\mathrm{in}\prime})'+\bm{g}$. 
Now, suppose that $\{\bm{x}_k^\mathrm{in}\}_{k=0}^{j+1}$ and $\{\tau_k^\mathrm{in}\}_{k=0}^{j-1}$ with an integer $j\geq1$. 
Differentiating \eqref{BVP2} $j$-times with respect to $t$ and then putting $t=0$, we obtain 
\begin{equation}\label{BVPini}
\begin{cases}
 -\tau_j^{\mathrm{in}\prime\prime}+|\bm{x}_0^{\mathrm{in}\prime\prime}|^2\tau_j^\mathrm{in} = h_j^\mathrm{in} \quad\mbox{in}\quad (0,1), \\
 \tau_j^\mathrm{in}(0)=0, \quad \tau_j^{\mathrm{in}\prime}(1)=-\bm{g}\cdot\bm{x}_j^{\mathrm{in}\prime}(1),
\end{cases}
\end{equation}

\noindent
where $h_j^\mathrm{in}$ has been already determined as 
\begin{align*}
h_j^\mathrm{in}
&= \sum_{j_1+j_2=j}\frac{j!}{j_1!j_2!}\bm{x}_{j_1+1}^{\mathrm{in}\prime}\cdot\bm{x}_{j_2+1}^{\mathrm{in}\prime}
 - \sum_{j_0+j_1+j_2=j,j_0\leq j-1}\frac{j!}{j_0!j_1!j_2!}(\bm{x}_{j_1}^{\mathrm{in}\prime\prime}\cdot\bm{x}_{j_2}^{\mathrm{in}\prime\prime})\tau_{j_0}^\mathrm{in}.
\end{align*}
By solving this two-point boundary value problem, $\tau_j^\mathrm{in}$ is uniquely determined. 
Then, differentiating the first equation in \eqref{HP2} $j$-times with respect to $t$ and then putting $t=0$, we obtain 
\begin{equation}\label{Xjini}
\bm{x}_{j+2}^\mathrm{in}=\sum_{j_0+j_1=j}\frac{j!}{j_0!j_1!}(\tau_{j_0}^\mathrm{in}\bm{x}_{j_1}^{\mathrm{in}\prime})',
\end{equation}

\noindent
so that $\bm{x}_{j+2}^\mathrm{in}$ is determined. 
This procedure can be continued up to a certain order $j$ depending on the regularity of the initial data $(\bm{x}_0^\mathrm{in},\bm{x}_1^\mathrm{in})$. 
More precisely, we have the following lemma.

\begin{lemma}\label{lem:EstIV}
Let $m\geq1$ be an integer and assume that $\bm{x}_0^\mathrm{in} \in X^m$ and $\bm{x}_1^\mathrm{in} \in X^{m-1}$. 
Then, the initial values $\{\bm{x}_j^\mathrm{in}\}_{j=0}^{m-1}$ are in fact determined by \eqref{BVPini} and \eqref{Xjini} such that 
$\bm{x}_j^\mathrm{in} \in X^{m-j}$ for $0\leq j\leq m-1$ except in the case $m=3$. 
In the case $m=3$, we have $\bm{x}_2^\mathrm{in} \in X_\epsilon^1$ for any $\epsilon>0$. 
\end{lemma}

\begin{proof}
The case $m=1,2$ is trivial. 
The case $m\geq4$ was proved in \cite[Sections 8 and 9]{IguchiTakayama2024}. 
Therefore, it is sufficient to show the case $m=3$. 
It follows from Lemma \ref{lem:EstSolBVP3} together with $\|\bm{x}_0^{\mathrm{in}\prime\prime}\|_{Y^0} \leq \|\bm{x}_0^\mathrm{in}\|_{X^3}$ and 
$|\bm{x}_0^{\mathrm{in}\prime}(1)| \lesssim \|\bm{x}_0^\mathrm{in}\|_{X^2}$ that 
\[
\|\tau_0^{\mathrm{in}\prime}\|_{L^\infty} \leq C(\|\bm{x}_0^\mathrm{in}\|_{X^3})( \|\bm{x}_0^\mathrm{in}\|_{X^2} + \|\bm{x}_1^\mathrm{in}\|_{X^2}^2 ),
\]
so that we have $\tau_0^{\mathrm{in}\prime}\in L^\infty$. 
By Lemmas \ref{lem:CalIneq1} and \ref{lem:CalIneq2}, we see that 
\begin{align*}
\|\tau_0^{\mathrm{in}\prime\prime}\|_{Y^0}
&\leq \| |\bm{x}_1^{\mathrm{in}\prime}|^2\|_{Y^0} + \||\bm{x}_0^{\mathrm{in}\prime\prime}|^2\tau_0^\mathrm{in}\|_{Y^0} \\
&\lesssim \|\bm{x}_1^\mathrm{in}\|_{X^2}^2 + \|\tau_0^{\mathrm{in}\prime}\|_{L^\infty} \|\bm{x}_0^\mathrm{in}\|_{X^3}^2,
\end{align*}
so that we have $\tau_0^{\mathrm{in}\prime\prime} \in Y^0$. 
By Lemma \ref{lem:EstAu}, we have 
\begin{align*}
\|\bm{x}_2^\mathrm{in}\|_{L^2}
&\leq \|(\tau_0^\mathrm{in}\bm{x}_0^{\mathrm{in}\prime})'\|_{L^2} + 1 \\
&\lesssim \|\tau_0^{\mathrm{in}\prime}\|_{L^\infty}\|\bm{x}_0^\mathrm{in}\|_{X^2}.
\end{align*}
Moreover, for any $\epsilon>0$ we see also that 
\begin{align*}
\|s^{\frac12+\epsilon}\bm{x}_2^{\mathrm{in}\prime}\|_{L^2}
&= \|s^{\frac12+\epsilon}(\tau_0^\mathrm{in}\bm{x}_0^{\mathrm{in}\prime})''\|_{L^2} \\
&\leq \|\tau_0^{\mathrm{in}\prime}\|_{L^\infty}( \|s^\frac32\bm{x}_0^{\mathrm{in}\prime\prime\prime}\|_{L^2}
  + 2\|s^\frac12\bm{x}_0^{\mathrm{in}\prime\prime}\|_{L^2} )
 + \|s^\frac12\tau_0^{\mathrm{in}\prime\prime}\|_{L^2}\|s^\epsilon\bm{x}_0^{\mathrm{in}\prime}\|_{L^\infty} \\
&\lesssim (\|\tau_0^{\mathrm{in}\prime}\|_{L^\infty}+\|\tau_0^{\mathrm{in}\prime\prime}\|_{Y^0}) \|\bm{x}_0^\mathrm{in}\|_{X^3},
\end{align*}
where we used $\|s^\epsilon u\|_{L^\infty} \leq C_\epsilon\|u\|_{X^1}$ which was shown in \cite[Lemma 4.3]{IguchiTakayama2024} and also follows from 
a sharper estimate given in the proof of Lemma \ref{lem:NormEquiv}. 
These estimates imply $\bm{x}_2^\mathrm{in} \in X_\epsilon^1$. 
\end{proof}

Differentiating the first boundary condition in \eqref{HP2} $j$-times with respect to $t$ and then putting $t=0$, we obtain 
\begin{equation}\label{CC}
\bm{x}_j^\mathrm{in}(1)=\bm{0}.
\end{equation}
This is a compatibility condition at order $j$ for the problem \eqref{HP2}--\eqref{IC2}. 
Under the assumptions in Lemma \ref{lem:EstIV}, the lemma and the Sobolev embedding theorem imply $\bm{x}_j^\mathrm{in} \in C^{m-1-j}((0,1])$ for 
$j=0,1,\ldots,m-1$, so that the traces $\{\bm{x}_j^\mathrm{in}(1)\}_{j=0}^{m-1}$ are defined. 
Therefore, the compatibility conditions make sense for $j=0,1,\ldots,m-1$.

\begin{definition}\label{def:CC}
Let $m\geq1$ be an integer. 
We say that the initial data $(\bm{x}_0^\mathrm{in},\bm{x}_1^\mathrm{in})$ for the initial boundary value problem \eqref{Eq}--\eqref{IC} and for the 
problem \eqref{HP2}--\eqref{IC2} satisfy the compatibility conditions up to order $m-1$ if \eqref{CC} holds for any $j=0,1,\ldots,m-1$. 
\end{definition}

The following propositions ensure that the initial data $(\bm{x}_0^\mathrm{in},\bm{x}_1^\mathrm{in})$ for the problem \eqref{HP2}--\eqref{IC2} 
and for the problem \eqref{Eq}--\eqref{IC} can be approximated by smooth initial data satisfying higher order compatibility conditions. 
These propositions are a key to show our main result Theorem \ref{th:main} in the case $m=4,5$.

\begin{proposition}\label{prop:AppID1}
Let $m$ and $p$ be positive integers and assume that the initial data $(\bm{x}_0^\mathrm{in},\bm{x}_1^\mathrm{in})\in X^m\times X^{m-1}$ 
for the problem \eqref{HP2}--\eqref{IC2} satisfies the compatibility conditions up to order $m-1$. 
Then, there exists a sequence of initial data $\{(\bm{x}_0^{\mathrm{in}(n)},\bm{x}_1^{\mathrm{in}(n)})\}_{n=1}^\infty \subset C^\infty([0,1])$, 
which converges to $(\bm{x}_0^\mathrm{in},\bm{x}_1^\mathrm{in})$ in $X^m\times X^{m-1}$ and satisfies the compatibility conditions up to order $m+p-1$. 
\end{proposition}

\begin{proposition}\label{prop:AppID2}
If, in addition to the assumptions of Proposition \ref{prop:AppID1}, $m\geq4$ and the initial data satisfies 
$|\bm{x}_0^{\mathrm{in}\prime}(s)|=1$ and $\bm{x}_0^{\mathrm{in}\prime}(s)\cdot\bm{x}_1^{\mathrm{in}\prime}(s)=0$ for $0<s<1$, 
then the approximated initial data can be constructed such that $|\bm{x}_0^{{\mathrm{in}(n)}\prime}(s)|=1$ and 
$\bm{x}_0^{\mathrm{in}(n)\prime}(s)\cdot\bm{x}_1^{\mathrm{in}(n)\prime}(s)=0$ holds for $0<s<1$.
\end{proposition}

Our proofs of these propositions are rather technical and long so we postpone them until Sections \ref{sect:AppID1} and \ref{sect:AppID2}.

%----------------------------------------------------------------------------------------------------------------------
%----------------------------------------------------------------------------------------------------------------------
\section{Quasilinear problem}\label{sect:QLS}
As was shown in Iguchi and Takayama \cite{IguchiTakayama2023-2}, a linearized problem for unknowns $(\bm{y},\nu)$ 
to the nonlinear problem \eqref{Eq}--\eqref{IC} around $(\bm{x},\tau)$ can be solved uniquely in a weighted Sobolev space. 
However, the map $(\bm{x},\tau) \mapsto (\bm{y},\nu)$ reveals loss of twice derivatives. 
A standard procedure to overcome this difficulty is to reduce the nonlinear problem into a quasilinear problem introducing new quantities.

%----------------------------------------------------------------------------
\subsection{Derivation of the quasilinear problem}
For our problem, it is convenient to introduce new unknowns $(\bm{y},\nu)$ by 
\begin{equation}\label{defynu}
\bm{y} = \ddot{\bm{x}}, \quad \nu = \ddot{\tau}.
\end{equation}
Differentiating \eqref{Eq}, \eqref{BC}, and \eqref{BVP} twice with respect to $t$, we obtain 
\begin{equation}\label{QLy}
\begin{cases}
 \ddot{\bm{y}}=(\tau\bm{y}')'+(\nu\bm{x}')'+\bm{f} &\mbox{in}\quad (0,1)\times(0,T), \\
 \bm{y}=\bm{0} &\mbox{on}\quad \{s=1\}\times(0,T), \\
 (\bm{y},\dot{\bm{y}})|_{t=0}=(\bm{y}_0^\mathrm{in},\bm{y}_1^\mathrm{in}) &\mbox{in}\quad (0,1)
\end{cases}
\end{equation}
and 
\begin{equation}\label{QLnu}
\begin{cases}
 -\nu''+|\bm{x}''|^2\nu = 2\dot{\bm{x}}'\cdot\dot{\bm{y}}' - 2(\bm{x}''\cdot\bm{y}'')\tau + h &\mbox{in}\quad (0,1)\times(0,T), \\
 \nu = 0 &\mbox{on}\quad \{s=0\}\times(0,T), \\
 \nu' = -\bm{g}\cdot\bm{y}' &\mbox{on}\quad \{s=1\}\times(0,T),
\end{cases}
\end{equation}
where 
\begin{equation}\label{QLfh}
\begin{cases}
 \bm{f} = 2(\dot{\tau}\dot{\bm{x}}')', \\
 h = 2( |\ddot{\bm{x}}'|^2 - |\dot{\bm{x}}''|^2\tau - 2(\bm{x}''\cdot\dot{\bm{x}}'')\dot{\tau}).
\end{cases}
\end{equation}
These equations are used to determine $(\bm{y},\nu)$. 
On the other hand, to determine $(\bm{x},\tau)$ we use \eqref{defynu}, that is, 
\begin{equation}\label{QLxtau}
\begin{cases}
 \ddot{\bm{x}}=\bm{y}, &(\bm{x},\dot{\bm{x}})|_{t=0}=(\bm{x}_0^\mathrm{in},\bm{x}_1^\mathrm{in}), \\
 \ddot{\tau}=\nu,      &(\tau,\dot{\tau})|_{t=0}=(\tau_0^\mathrm{in},\tau_1^\mathrm{in}).
\end{cases}
\end{equation}
In this section, we will show the well-posedness of the initial boundary value problem \eqref{QLy}--\eqref{QLxtau} for unknowns $(\bm{x},\bm{y},\tau,\nu)$ 
in a weighted Sobolev space.

%----------------------------------------------------------------------------
\subsection{Compatibility conditions for the quasilinear problem}
Before going to show the well-posedness, we explain compatibility conditions for this problem. 
Let $(\bm{x},\bm{y},\tau,\nu)$ be a smooth solution of \eqref{QLy}--\eqref{QLxtau} and put 
$(\bm{x}_j^\mathrm{in},\bm{y}_j^\mathrm{in},\tau_j^\mathrm{in},\nu_j^\mathrm{in})=\dt^j(\bm{x},\bm{y},\tau,\nu)|_{t=0}$ for $j=0,1,2,\ldots$. 
These initial values are determined, at least formally, from the initial data 
$(\bm{x}_0^\mathrm{in},\bm{x}_1^\mathrm{in},\bm{y}_0^\mathrm{in},\bm{y}_1^\mathrm{in},\tau_0^\mathrm{in},\tau_1^\mathrm{in})$. 
In fact, these initial values are determined inductively as follows: 
Suppose that $\{\bm{y}_k^\mathrm{in}\}_{k=0}^{j+1}$ and $\{\nu_k^\mathrm{in}\}_{k=0}^{j-1}$ are determined for a non-negative integer $j$. 
It follows from \eqref{QLxtau} that $\bm{x}_{k+2}^\mathrm{in}=\bm{y}_k^\mathrm{in}$ and $\tau_{k+2}^\mathrm{in}=\nu_k^\mathrm{in}$ so that 
$\{\bm{x}_k^\mathrm{in}\}_{k=0}^{j+3}$ and $\{\tau_k^\mathrm{in}\}_{k=0}^{j+1}$ are determined. 
Differentiating \eqref{QLnu} $j$-times with respect to $t$ and then putting $t=0$, we see that 
\[
\begin{cases}
 -\nu_j^{\mathrm{in}\prime\prime}+|\bm{x}_0^{\mathrm{in}\prime\prime}|^2\nu_j^\mathrm{in}
  = H_j \quad\mbox{in}\quad (0,1), \\
 \nu_j^\mathrm{in}(0)=0, \quad \nu_j^{\mathrm{in}\prime}(1)=-\bm{g}\cdot\bm{y}_j^{\mathrm{in}\prime}(1),
\end{cases}
\]
where $H_j$ is a polynomial of $\{\bm{x}_k^\mathrm{in}\}_{k=0}^{j+3}$, $\{\tau_k^\mathrm{in}\}_{k=0}^{j+1}$, and their derivatives, 
so that $H_j$ is already determined. 
This two-point boundary value problem for $\nu_j^\mathrm{in}$ is uniquely solved so that $\nu_j^\mathrm{in}$ and $\tau_{j+2}^\mathrm{in}$ are determined. 
Differentiating \eqref{QLy} $j$-times with respect to $t$ and then putting $t=0$, we see that 
$\bm{y}_{j+2}^\mathrm{in}$ is a polynomial $\{\bm{x}_k^\mathrm{in}\}_{k=0}^{j+2}$, $\{\tau_k^\mathrm{in}\}_{k=0}^{j+2}$, and their derivatives, 
so that $\bm{y}_{j+2}^\mathrm{in}$ is also determined. 
On the other hand, differentiating the boundary condition in \eqref{QLy} $j$-times with respect to $t$ and then putting $t=0$ we have 
\begin{equation}\label{QLCC}
\bm{y}_j^\mathrm{in}(1)=\bm{0}. 
\end{equation}

\noindent
This is the compatibility condition at order $j$ for the problem \eqref{QLy}--\eqref{QLxtau}. 
As will be shown in Lemma \ref{lem:EstQLIV}, if $(\bm{x}_0^\mathrm{in},\bm{x}_1^\mathrm{in},\bm{y}_0^\mathrm{in})\in X^m$ and 
$(\bm{y}_1^\mathrm{in},\tau_0^\mathrm{in},\tau_0^{\mathrm{in}\prime},\tau_1^\mathrm{in},\tau_1^{\mathrm{in}\prime}) \in X^{m-1}$ with $m\geq4$, 
then $\bm{y}_j^\mathrm{in}$ is in fact determined as a function in $X^{m-j}$ for $j=0,1,\ldots,m$. 
Therefore, the compatibility conditions make sense up to order $m-1$.

\begin{lemma}\label{lem:EstQLIV}
For any integer $m\geq4$ and any positive constants $M_0$, there exists a positive constant $C_0$ such that if the initial data satisfy 
\begin{equation}\label{CondID2}
\begin{cases}
 \|(\bm{x}_0^\mathrm{in},\bm{x}_1^\mathrm{in},\bm{y}_0^\mathrm{in})\|_{X^m}
  + \|(\bm{y}_1^\mathrm{in},\tau_0^{\mathrm{in}\prime},\tau_1^{\mathrm{in}\prime})\|_{X^{m-1}} \leq M_0, \\
 \tau_0^\mathrm{in}(0)=\tau_1^\mathrm{in}(0)=0, 
\end{cases}
\end{equation}
then the initial values defined above satisfy 
$\sum_{j=0}^m \|\bm{x}_{j+2}^\mathrm{in}\|_{X^{m-j}} + \sum_{j=0}^{m-2} \|\tau_{j+2}^{\mathrm{in}\prime}\|_{X^{m-1-j}} \leq C_0$. 
\end{lemma}

\begin{proof}
This lemma can be proved by a similar argument to that in \cite[Sections 8 and 9]{IguchiTakayama2024}. 
Here, for completeness we give a proof. 
For simplicity, we omit the superscript ``in'' from the notation so that we write $\bm{x}_j, \bm{y}_j, \tau_j$, and $\nu_j$ instead of
$\bm{x}_j^{\rm in}, \bm{y}_j^{\rm in}, \tau_j^{\rm in}$, and $\nu_j^{\rm in}$, respectively, throughout the proof. 
For an integer $j=1,2,\ldots,m$, we put 
\[
\begin{cases}
\mathcal{T}_{j} = \|(\tau_0',\tau_1')\|_{X^{m-1}} + \sum_{k=2}^j \|\tau_k'\|_{X^{m+1-k}}, \\
\mathcal{X}_{j+2} = \|(\bm{x}_0,\bm{x}_1)\|_{X^m} + \sum_{k=2}^{j+2} \|\bm{x}_k\|_{X^{m+2-k}}.
\end{cases}
\]
Then, it is sufficient to show that $\mathcal{X}_{m+2}+\mathcal{T}_{m} \leq C_0$. 
Since $\bm{y}_j=\bm{x}_{j+2}$ and $\nu_j=\tau_{j+2}$, it is not difficult to see that 
$\bm{x}_{j+2} = \sum_{j_0+j_1=j}\frac{j!}{j_0!j_1!}(\tau_{j_0}\bm{x}_{j_1}')'$ holds for $j\geq2$. 
Therefore, by Lemma \ref{lem:EstAu} we see that 
\begin{align*}
\|\bm{x}_{j+2}\|_{X^{m-j}}
&\lesssim \|\tau_j'\|_{X^{m-j}}\|\bm{x}_0\|_{X^{m+2-j\vee4}} + \sum_{j_0+j_1=j, j_0\leq j-1}\|\tau_{j_0}'\|_{X^{m-j\vee2}}\|\bm{x}_{j_1}\|_{X^{m+2-j}} \\
&\lesssim \mathcal{T}_j\mathcal{X}_j
\end{align*}
so that 
\begin{equation}\label{EstX}
\mathcal{X}_{j+2} \lesssim \mathcal{X}_{j+1} + \mathcal{T}_j\mathcal{X}_j
\end{equation}
for $j=2,3,\ldots,m$. 
On the other hand, for $j\geq2$ we have 
\[
\begin{cases}
 -\tau_j''+|\bm{x}_0''|^2\tau_j = h_j \quad\mbox{in}\quad (0,1), \\
 \tau_j(0)=0, \quad \tau_j'(1)=-\bm{g}\cdot\bm{x}_j'(1),
\end{cases}
\]
where $h_j=\sum_{j_1+j_2=j}\frac{j!}{j_1!j_2!}\bm{x}_{j_1+1}'\cdot\bm{x}_{j_2+1}'
 - \sum_{j_0+j_1+j_2=j,j_0\leq j-1}\frac{j!}{j_0!j_1!j_2!}\tau_{j_0}\bm{x}_{j_1}''\cdot\bm{x}_{j_2}''$. 
Therefore, by Lemma \ref{lem:HEstBVP} together with $\|\bm{x}_0'\|_{Y^{m+1-j\vee3}} \leq \|\bm{x}_0\|_{X^{m+2-j\vee4}} \leq \|\bm{x}_0\|_{X^m}$ and 
$|\bm{x}_j'(1)| \lesssim \|\bm{x}_j\|_{X^2} \leq \|\bm{x}_j\|_{X^{m+2-j}}$, we obtain 
\[
\|\tau_j\|_{X^{m+1-j}} \leq C(\|\bm{x}_0\|_{X^m})( \|\bm{x}_j\|_{X^{m+2-j}} + \|h_j\|_{Y^{m-j}} )
\]
for $j=2,3,\ldots,m$. 
To evaluate $\|h_j\|_{Y^{m-j}}$, note first that by Lemmas \ref{lem:CalIneq1} and \ref{lem:CalIneq2} we have 
\[
\begin{cases}
 \|u'v'\|_{Y^k} \lesssim \min\{ \|u\|_{X^{k+1\vee4}} \|v\|_{X^{k+1}}, \|u\|_{X^{k+2}}\|v\|_{X^{k+2}} \}, \\
 \|\tau u''v''\|_{Y^k} \lesssim \|\tau'\|_{X^{k-1\vee2}} \|u\|_{X^{k+2\vee4}} \|v\|_{X^{k+2}}
\end{cases}
\]
for $k=0,1,2,\ldots$, if $\tau(0)=0$. 
Therefore, we see that 
\begin{align*}
\|h_j\|_{Y^{m-j}}
&\lesssim \sum_{j_1+j_2=j}\|\bm{x}_{j_1+1}'\cdot\bm{x}_{j_2+1}'\|_{Y^{m-j}}
 + \sum_{j_0+j_1+j_2=j,j_0\leq j-1}\|\tau_{j_0}\bm{x}_{j_1}''\cdot\bm{x}_{j_2}''\|_{Y^{m-j}} \\
&\lesssim \|\bm{x}_1\|_{X^{m+1-j\vee4}} \|\bm{x}_{j+1}\|_{X^{m+1-j}}
 + \sum_{j_1,j_2\leq j-1} \|\bm{x}_{j_1+1}\|_{X^{m+2-j}} \|\bm{x}_{j_2+1}\|_{X^{m+2-j}} \\
&\quad\;
 + \sum_{j_0+j_1+j_2=j, j_0\leq j-1, j_1\leq j_2} \|\tau_{j_0}'\|_{X^{m-1-j \vee2}} \|\bm{x}_{j_1}\|_{X^{m+2-j\vee4}} \|\bm{x}_{j_2}\|_{X^{m+2-j}} \\
&\lesssim (1+\mathcal{T}_{j-1})\mathcal{X}_{j+1}^2,
\end{align*}
so that $\|\tau_j\|_{X^{m+1-j}} \leq C(\|\bm{x}_0\|_{X^m})( \mathcal{X}_j + (1+\mathcal{T}_{j-1})\mathcal{X}_{j+1}^2)$.
Particularly, we obtain 
\begin{equation}\label{EstT}
\mathcal{T}_j \leq \mathcal{T}_{j-1} + C(\|\bm{x}_0\|_{X^m})( \mathcal{X}_j + (1+\mathcal{T}_{j-1})\mathcal{X}_{j+1}^2)
\end{equation}
for $j=2,3,\ldots,m$. 
Now, using \eqref{EstX} and \eqref{EstT} inductively we obtain $\mathcal{T}_m,\mathcal{X}_{m+2}\lesssim1$. 
In fact, by the assumptions of the lemma we have $\mathcal{T}_1,\mathcal{X}_3\lesssim 1$. 
Suppose that $\mathcal{T}_{j-1},\mathcal{X}_{j+1} \lesssim 1$ holds for some $j=2,3,\ldots,m$. 
Then, it follows from \eqref{EstT} that $\mathcal{T}_j \lesssim1$ and then from \eqref{EstX} that $\mathcal{X}_{j+2} \lesssim 1$. 
Therefore, by induction we obtain the desired estimates. 
\end{proof}

%----------------------------------------------------------------------------
\subsection{Well-posedness of the quasilinear problem}
The following proposition shows the well-posedness of the quasilinear problem \eqref{QLy}--\eqref{QLxtau}.

\begin{proposition}\label{prop:QLWP}
For any integer $m\geq4$ and any positive constants $M_0$ and $c_0$, 
there exists a small positive time $T$ such that if the initial data satisfy \eqref{CondID2}, the stability condition 
\begin{equation}\label{CondID3}
\frac{\tau_0^\mathrm{in}(s)}{s} \geq 2c_0 \quad\mbox{for}\quad 0<s<1, 
\end{equation}
and the compatibility conditions \eqref{QLCC} up to order $m-1$, then the initial boundary value problem \eqref{QLy}--\eqref{QLxtau} has a unique solution 
$(\bm{x},\bm{y},\tau,\nu)$ on the time interval $[0,T]$ satisfying $\bm{x},\dot{\bm{x}},\bm{y}\in\mathscr{X}_T^m$, 
$\tau',\dot{\tau}',\nu'\in\mathscr{X}_T^{m-1,*}$, and the stability condition \eqref{SC}. 
\end{proposition}

\begin{proof}
By Lemma \ref{lem:EstQLIV}, we can determine the initial values $\bm{x}_j^\mathrm{in}$ for $0\leq j\leq m+2$ and $\tau_j^\mathrm{in}$ for $0\leq j\leq m$, 
which satisfy $\bm{x}_{j+2} \in X^{m-j}$ for $0\leq j\leq m$ and $\tau_{j+2}' \in X^{m-1-j}$ for $0\leq j\leq m-2$. 
For $T,M_1,M_2,M_3>0$, we define $S^m_T(M_1,M_2,M_3)$ the set of all $(\bm{x},\tau)$ satisfying 
\begin{equation}\label{SolSet}
\begin{cases}
 \bm{x},\dot{\bm{x}},\ddot{\bm{x}} \in \mathscr{X}_T^m, \quad
  \tau', \dot{\tau}', \ddot{\tau}' \in \mathscr{X}_T^{m-1,*}, \\
 (\dt^j\bm{x})|_{t=0}=\bm{x}_j^\mathrm{in} \ (0\leq j\leq m+2), \quad
  (\dt^j\tau)|_{t=0}=\tau_j^\mathrm{in} \ (0\leq j\leq m), \\
 c_0s \leq \tau(s,t) \leq M_1s \quad\mbox{for}\quad (s,t)\in[0,1]\times[0,T], \\
 \opnorm{(\bm{x},\dot{\bm{x}})(t)}_m + \opnorm{(\tau',\dot{\tau}')(t)}_{m-1,*} \leq M_1, \\
 \opnorm{\ddot{\bm{x}}(t)}_m + \opnorm{\ddot{\tau}'(t)}_{m-2} \leq M_2, \quad \opnorm{\ddot{\tau}'(t)}_{m-1,*} \leq M_3 \quad\mbox{for}\quad t\in[0,T].
\end{cases}
\end{equation}
In view of Lemma \ref{lem:NormEq}, it is standard to check that there exist a small $T_0>0$ and large $M_{01},M_{02},M_{03}>0$ such that 
$S^m_{T_0}(M_{01},M_{02},M_{03})\ne\emptyset$.

For given $(\bm{x},\tau)\in S^m_T(M_1,M_2,M_3)$, we consider 
the initial boundary value problem \eqref{QLy} and \eqref{QLnu} for unknowns $(\bm{y},\nu)$, 
where $\bm{f}$ and $h$ are defind by \eqref{QLfh}. 
By Lemma \ref{lem:EstAu2}, we have 
\[
\begin{cases}
 \opnorm{\bm{f}}_{m-2} \lesssim \opnorm{\dot{\tau}'}_{m-2} \opnorm{\dot{\bm{x}}}_m, \\
 \opnorm{\dt\bm{f}}_{m-2} \lesssim \opnorm{(\dot{\tau}',\ddot{\tau}')}_{m-2} \opnorm{(\dot{\bm{x}},\ddot{\bm{x}})}_m.
\end{cases}
\]
As for $h$, we decompose it as $h=h_1+h_2$ with $h_1=2|\ddot{\bm{x}}'|^2$ and $h_2= - 2(|\dot{\bm{x}}''|^2\tau + 2(\bm{x}''\cdot\dot{\bm{x}}'')\dot{\tau})$. 
By Lemma \ref{lem:CalIneqY2}, we have $\opnorm{h_2}_{m-2}^\dag \lesssim \opnorm{(\tau',\dot{\tau}')}_{m-2} \opnorm{(\bm{x},\dot{\bm{x}})}_m^2$. 
By Lemma \ref{lem:CalIneqY1}, we have also $\opnorm{h_1}_{m-2}^\dag \lesssim \opnorm{\ddot{\bm{x}}}_{m-1}^2 \leq \opnorm{\dot{\bm{x}}}_m^2$ in the case $m\geq5$; 
and $\opnorm{h_1}_{m-3}^\dag = \opnorm{h_1}_1^\dag \lesssim \opnorm{\ddot{\bm{x}}}_{3}^2 \leq \opnorm{\dot{\bm{x}}}_m^2$ and 
$\opnorm{h_1}_{m-2}^\dag = \opnorm{h_1}_2^\dag \lesssim \opnorm{\ddot{\bm{x}}}_{4}^2 = \opnorm{\ddot{\bm{x}}}_m^2$ in the case $m=4$. 
Moreover, in the case $m=4$, by Lemma \ref{lem:embedding2} we have 
$\|s^\frac12\dt^{m-2}h_1\|_{L^1} \lesssim \|\dt^2\bm{x}\|_{X^2} \|\dt^4\bm{x}\|_{X^1} + \|\dt^3\bm{x}\|_{X^2}\|\dt^3\bm{x}\|_{X^1} \lesssim \|\dot{\bm{x}}\|_m^2$. 
Furthermore, by Lemmas \ref{lem:CalIneqY1} and \ref{lem:CalIneqY2} we have 
$\|s^\frac12\dt^{m-1}h\|_{L^1} \lesssim \opnorm{\dt h}_{m-2}
 \lesssim ( 1 + \opnorm{ (\tau',\dot{\tau}',\ddot{\tau}') }_{m-2} ) \opnorm{ (\bm{x},\dot{\bm{x}},\ddot{\bm{x}}) }_m^2$. 
Therefore, 
\[
\begin{cases}
 \opnorm{h}_{m-3}^\dag + \|s^\frac12\dt^{m-2}h\|_{L^1} \lesssim ( 1 + \opnorm{ (\tau',\dot{\tau}') }_{m-2} ) \opnorm{ (\bm{x},\dot{\bm{x}}) }_m^2, \\
 \opnorm{h}_{m-2}^\dag \lesssim ( 1 + \opnorm{ (\tau',\dot{\tau}') }_{m-2} ) \opnorm{ (\bm{x},\dot{\bm{x}},\ddot{\bm{x}}) }_m^2, \\
 \|s^\frac12\dt^{m-1}h\|_{L^1} \lesssim ( 1 + \opnorm{ (\tau',\dot{\tau}',\ddot{\tau}') }_{m-2} ) \opnorm{ (\bm{x},\dot{\bm{x}},\ddot{\bm{x}}) }_m^2.
\end{cases}
\]
Particularly, we have $\bm{f}\in\mathscr{X}_T^{m-2}$, $\dt^{m-1}\bm{f}\in L^1(0,T;L^2)$, $h\in\mathscr{Y}_T^{m-2}$, 
and $s^\frac12\dt^{m-1}h\in L^1((0,1)\times(0,T))$. 
On the other hand, it is straightforward to see that the data $(\bm{y}_0^\mathrm{in},\bm{y}_1^\mathrm{in},\bm{f},h)$ satisfy the compatibility conditions 
up to order $m-1$. 
Therefore, we can apply the existence theorem \cite[Theorem 2.8]{IguchiTakayama2023-2}, which guarantees the unique existence of a solution $(\bm{y},\nu)$ to 
the initial boundary value problem \eqref{QLy} and \eqref{QLnu} in the class $\bm{y}\in\mathscr{X}_T^m$ and $\nu'\in\mathscr{X}_T^{m-1,*}$. 
Moreover, the solution satisfies 
\begin{align}\label{EE1}
\opnorm{\bm{y}(t)}_m + \opnorm{\nu'(t)}_{m-2}
&\leq C_1\mathrm{e}^{C_2t}\biggl\{ \|\bm{y}_0^\mathrm{in}\|_{X^m} + \|\bm{y}_1^\mathrm{in}\|_{X^{m-1}} \\
&\quad\;
 + \max_{0\leq t'\leq t}\bigl( \opnorm{\bm{f}(t')}_{m-2} + \opnorm{h(t')}_{m-3}^\dag + \|s^\frac12\dt^{m-2}h(t')\|_{L^1} \bigr) \nonumber \\
&\quad\;
 + C_2\int_0^t (\|\dt^{m-1}\bm{f}(t')\|_{L^2} + \|s^\frac12\dt^{m-1}h(t')\|_{L^1} )\mathrm{d}t' \biggr\} \nonumber \\
&\leq C_1\mathrm{e}^{C_2t}, \nonumber
\end{align}

\noindent
where we denote the constants $C_1=C(m,c_0,M_0,M_1)$ and $C_2=C(m,c_0,M_0,M_1,M_2)$. 
These constants may change from line to line. 
We have also 
\begin{align}\label{EE2}
\opnorm{\nu'(t)}_{m-1,*}
&\leq C_1( \opnorm{\bm{y}(t)}_m + \opnorm{\nu'(t)}_{m-2} + \opnorm{h(t)}_{m-2}^\dag ) \\
&\leq C_2\mathrm{e}^{C_2t}. \nonumber
\end{align}

\noindent
We define a map $\Phi$ by $\Phi(\bm{x},\tau):=(\bm{y},\nu)$.

By using this solution $(\bm{y},\nu)$, we define $(\tilde{\bm{x}},\tilde{\tau})$ as a unique solution to \eqref{QLxtau}. 
Obviously, $(\tilde{\bm{x}},\tilde{\tau})$ satisfy the conditions in the first two lines in \eqref{SolSet}. 
By \eqref{EE1} and \eqref{EE2}, we have also 
\[
\begin{cases}
 \opnorm{ (\tilde{\bm{x}},\dt\tilde{\bm{x}} )(t) }_m + \opnorm{ (\tilde{\tau}',\dt\tilde{\tau}')(t) }_{m-1,*} \leq C_0 + C_3t^2, \\
 \opnorm{ \dt^2\tilde{\bm{x}}(t) }_m + \opnorm{ \dt^2\tilde{\tau}'(t) }_{m-2} \leq C_1\mathrm{e}^{C_2t}, \\
 \opnorm{ \dt^2\tilde{\tau}'(t) }_{m-1,*} \leq C_2\mathrm{e}^{C_2t}, \\
 |\frac{1}{s}\tilde{\tau}(s,t) - \frac{1}{s}\tau_0^\mathrm{in}(s)| \leq C_2t \quad\mbox{for}\quad (s,t)\in[0,1]\times[0,T],
\end{cases}
\]
where we denote the constants $C_0=C(m,c_0,M_0)$ and $C_3=C(m,c_0,M_0,M_1,M_2,M_3)$ and used Lemma \ref{lem:EstQLIV} to evaluate initial values. 
Now, we define the constants $M_1,M_2,M_3$ as $M_1=2C_0$, $M_2=2C_1$, $M_3=2C_2$, and choose the time $T>0$ so small that 
$C_3T^2\leq C_0$, $\mathrm{e}^{C_2T}\leq2$, and $C_2T\leq c_0$. 
Then, we see that $(\tilde{x},\tilde{\tau})$ satisfy all the conditions in \eqref{SolSet}. 
We define a map $\Psi$ by $\Psi(\bm{x},\tau):=(\tilde{\bm{x}},\tilde{\tau})$. 
We have shown that there exist large constants $M_1,M_2,M_3>0$ and a small time $T>0$ such that the set $S^m_T(M_1,M_2,M_3)$ 
is not empty and that $\Psi$ maps $S^m_T(M_1,M_2,M_3)$ into itself. 
In the following we fix these constants.

We first take $(\bm{x}^{(0)},\tau^{(0)}) \in S^m_T(M_1,M_2,M_3)$ arbitrarily. 
Then, we define inductively 
\[
(\bm{y}^{(n)},\nu^{(n)}):=\Phi(\bm{x}^{(n)},\tau^{(n)}), \quad(\bm{x}^{(n+1)},\tau^{(n+1)}):=\Psi(\bm{x}^{(n)},\tau^{(n)})
\]
for $n=0,1,2,\ldots$. 
$\{(\bm{x}^{(n)},\bm{y}^{(n)},\tau^{(n)},\nu^{(n)})\}_{n=0}^\infty$ is a sequence of approximate solutions to the initial value problem 
for the quasilinear system \eqref{QLy}--\eqref{QLxtau}. 
We note that $\{(\bm{x}^{(n)},\tau^{(n)})\}_{n=0}^\infty$ satisfy the uniform bounds in \eqref{SolSet}. 
We proceed to show that these approximate solutions converge to a solution of the problem. 
To this end, we put $\check{\bm{x}}^{(n)}=\bm{x}^{(n+1)}-\bm{x}^{(n)}$, $\check{\bm{y}}^{(n)}=\bm{y}^{(n+1)}-\bm{y}^{(n)}$, 
$\check{\tau}^{(n)}=\tau^{(n+1)}-\tau^{(n)}$, and $\check{\nu}^{(n)}=\nu^{(n+1)}-\nu^{(n)}$. 
Then, we see that $(\check{\bm{y}}^{(n)},\check{\nu}^{(n)})$ satisfy 
\[
\begin{cases}
 \dt^2\check{\bm{y}}^{(n)}=(\tau^{(n)}\check{\bm{y}}^{(n)\prime})'+(\check{\nu}^{(n)}\bm{x}^{(n)\prime})'+\bm{F}^{(n)} &\mbox{in}\quad (0,1)\times(0,T), \\
 \check{\bm{y}}^{(n)}=\bm{0} &\mbox{on}\quad \{s=1\}\times(0,T), \\
 (\check{\bm{y}}^{(n)},\dt\check{\bm{y}}^{(n)})|_{t=0}=(\bm{0},\bm{0}) &\mbox{in}\quad (0,1)
\end{cases}
\]
and 
\[
\begin{cases}
 -\check{\nu}^{(n)\prime\prime}+|\bm{x}^{(n)\prime\prime}|^2\check{\nu}^{(n)} \\
 \phantom{\nu}
 = 2\dot{\bm{x}}^{(n)\prime}\cdot\dt\check{\bm{y}}^{(n)\prime} - 2(\bm{x}^{(n)\prime\prime}\cdot\check{\bm{y}}^{(n)\prime\prime})\tau^{(n)} + H^{(n)}
  &\mbox{in}\quad (0,1)\times(0,T), \\
 \check{\nu}^{(n)} = 0 &\mbox{on}\quad \{s=0\}\times(0,T), \\
 \check{\nu}^{(n)\prime} = -\bm{g}\cdot\check{\bm{y}}^{(n)\prime} &\mbox{on}\quad \{s=1\}\times(0,T),
\end{cases}
\]
where 
\[
\begin{cases}
 \bm{F}^{(n)} = \bm{f}^{(n+1)}-\bm{f}^{(n)} + (\check{\tau}^{(n)}\bm{y}^{(n+1)\prime} + \nu^{(n+1)}\check{\bm{x}}^{(n)\prime})', \\
 H^{(n)} = h^{(n+1)}-h^{(h)} - ((\bm{x}^{(n+1)}+\bm{x}^{(n)})''\cdot\check{\bm{x}}^{(n)\prime\prime})\nu^{(n+1)} \\
 \ \ 
 + 2\bigl\{ (\dt\check{\bm{x}}^{(n)})'\cdot\dot{\bm{y}}^{(n+1)\prime} - (\bm{x}^{(n+1)\prime\prime}\cdot\bm{y}^{(n+1)\prime\prime})\check{\tau}^{(n)}
 -(\check{\bm{x}}^{(n)\prime\prime}\cdot\bm{y}^{(n+1)\prime\prime})\tau^{(n)} \bigr\},
\end{cases}
\]
and $(\bm{f}^{(n)},h^{(n)})$ are defined by \eqref{QLfh} with $(\bm{x},\tau)$ replaced with $(\bm{x}^{(n)},\tau^{(n)})$. 
We see also that $(\check{\bm{x}}^{(n)},\check{\tau}^{(n)})$ satisfy 
\[
\begin{cases}
 \dt^2\check{\bm{x}}^{(n)}=\check{\bm{y}}^{(n-1)}, &(\check{\bm{x}}^{(n)},\dt\check{\bm{x}}^{(n)})|_{t=0}=(\bm{0},\bm{0}), \\
 \dt^2\check{\tau}^{(n)}=\check{\nu}^{(n-1)},      &(\check{\tau}^{(n)},\dt\check{\tau}^{(n)})|_{t=0}=(0,0).
\end{cases}
\]
In the same way as the previous evaluation for $(\bm{f},h)$, by Lemmas \ref{lem:EstAu2}--\ref{lem:CalIneqY3} we see that 
\[
\begin{cases}
 \opnorm{\bm{F}^{(n)}}_{m-3} + \opnorm{H^{(n)}}_{m-3}^\dag \lesssim \opnorm{(\check{\bm{x}}^{(n)},\dt\check{\bm{x}}^{(n)})}_{m-1}
  + \opnorm{(\check{\tau}^{(n)\prime},\dt\check{\tau}^{(n)\prime})}_{m-2,*}, \\
 \opnorm{\dt\bm{F}^{(n)}}_{m-3} + \|s^\frac12\dt^{m-2}H^{(n)}\|_{L^1} \\
 \quad
  \lesssim \opnorm{ (\check{\bm{x}}^{(n)},\dt\check{\bm{x}}^{(n)},\dt^2\check{\bm{x}}^{(n)}) }_{m-1}
  + \opnorm{ (\check{\tau}^{(n)\prime},\dt\check{\tau}^{(n)\prime},\dt^2\check{\tau}^{(n)\prime}) }_{m-2,*}. 
\end{cases}
\]
Moreover, since $(\dt^j\check{\bm{x}}^{(n)})|_{t=0}=\bm{0}$ for $0\leq j\leq m+2$ and $(\dt^j\check{\tau}^{(n)})|_{t=0}=0$ for $0\leq j\leq m$, 
we have 
\begin{align*}
&\opnorm{(\check{\bm{x}}^{(n)},\dt\check{\bm{x}}^{(n)})(t)}_{m-1}
  + \opnorm{(\check{\tau}^{(n)\prime},\dt\check{\tau}^{(n)\prime})(t)}_{m-2,*} \\
&\lesssim \int_0^t\bigl( \opnorm{\dt^2\check{\bm{x}}^{(n)}(t')}_{m-1}
  + \opnorm{\dt^2\check{\tau}^{(n)\prime}(t')}_{m-2,*} \bigr)\mathrm{d}t'.
\end{align*}
Now, we apply the energy estimate obtained in \cite[Theorem 2.8]{IguchiTakayama2023-2} to obtain 
\begin{align}\label{EE3}
\opnorm{\check{\bm{y}}^{(n)}(t)}_{m-1} + \opnorm{\check{\nu}^{(n)\prime}(t)}_{m-2,*}
&\lesssim \max_{0\leq t'\leq t}\bigl( \opnorm{ \bm{F}^{(n)}(t') }_{m-3} + \opnorm{ H^{(n)}(t') }_{m-3}^\dag \bigr) \\
&\quad\;
 + \int_0^t \bigl( \|\dt^{m-2}\bm{F}^{(n)}(t')\|_{L^2} + \|s^\frac12\dt^{m-2}H^{(n)}(t')\|_{L^1} \bigr)\mathrm{d}t' \nonumber \\
&\lesssim \int_0^t\bigl( \opnorm{\dt^2\check{\bm{x}}^{(n)}(t')}_{m-1}
  + \opnorm{\dt^2\check{\tau}^{(n)\prime}(t')}_{m-2,*} \bigr)\mathrm{d}t' \nonumber \\
&= \int_0^t\bigl( \opnorm{\check{\bm{y}}^{(n-1)}(t')}_{m-1}
  + \opnorm{\check{\nu}^{(n-1)\prime}(t')}_{m-2,*} \bigr)\mathrm{d}t'. \nonumber
\end{align}

\noindent
This shows that $\{(\bm{y}^{(n)},\nu^{(n)\prime})\}_{n=0}^\infty$ converge in $\mathscr{X}_T^{m-1}\times\mathscr{X}_T^{m-2,*}$. 
Particularly, $\{\nu^{(n)}\}_{n=0}^\infty$ also converges in $\mathscr{X}_T^{m-2,*}$. 
Then, we see that $\{(\bm{x}^{(n)},\dot{\bm{x}}^{(n)})\}_{n=0}^\infty$ and $\{(\tau^{(n)},\tau^{(n)\prime},\dot{\tau}^{(n)},\dot{\tau}^{(n)\prime})\}_{n=0}^\infty$ 
converge in $\mathscr{X}_T^{m-1}$ and $\mathscr{X}_T^{m-2,*}$, respectively. 
Let $(\bm{x},\bm{y},\tau,\nu)$ be the limit of $\{(\bm{x}^{(n)},\bm{y}^{(n)},\tau^{(n)},\nu^{(n)})\}_{n=0}^\infty$. 
Then, $(\bm{x},\bm{y},\tau,\nu)$ is a solution to \eqref{QLy}--\eqref{QLxtau} satisfying 
$\bm{x},\dot{\bm{x}},\bm{y}\in\mathscr{X}_T^{m-1}$ and $\tau',\dot{\tau}',\nu'\in\mathscr{X}_T^{m-2,*}$, and the stability condition \eqref{SC}.

We next show that the solution satisfies the regularity $\bm{x},\dot{\bm{x}},\bm{y}\in\mathscr{X}_T^{m}$ and $\tau',\dot{\tau}',\nu'\in\mathscr{X}_T^{m-1,*}$. 
Since the approximate solutions $(\bm{x}^{(n)},\tau^{(n)})$ satisfy the uniform bounds in \eqref{SolSet}, by standard compactness arguments we have 
$\opnorm{(\bm{x},\dot{\bm{x}},\ddot{\bm{x}})(t)}_m + \opnorm{ (\tau',\dot{\tau}',\ddot{\tau}')(t)}_{m-1,*} \leq C_2$ for $t\in[0,T]$, 
so that $(\bm{f},h)$ defined by \eqref{QLfh} satisfy $\bm{f}\in\mathscr{X}_T^{m-2}$, $\dt^{m-1}\bm{f}\in L^1(0,T;L^2)$, $h\in\mathscr{Y}_T^{m-3}$, 
$s^\frac12\dt^{m-2}h\in C^0([0,T];L^1)$, and $s^\frac12\dt^{m-1}h\in L^1((0,1)\times(0,T))$. 
Now, we regard \eqref{QLy} and \eqref{QLnu} as a linear problem for $(\bm{y},\tau)$ and apply the existence theorem \cite[Theorem 2.8]{IguchiTakayama2023-2} 
to obtain $\bm{y}\in\mathscr{X}_T^m$ and $\nu'\in\mathscr{X}_T^{m-2}$. 
Then, we see that $h\in\mathscr{Y}_T^{m-2}$. 
Therefore, by applying the existence theorem \cite[Theorem 2.8]{IguchiTakayama2023-2} again to obtain $\nu'\in\mathscr{X}_T^{m-1,*}$. 
Moreover, in view of \eqref{QLxtau} we easily see that $(\bm{x},\tau)$ also satisfy the desired regularity. 
Finally, the uniqueness of solutions follows from an estimate similar to \eqref{EE3}. 
\end{proof}

%----------------------------------------------------------------------------------------------------------------------
%----------------------------------------------------------------------------------------------------------------------
\section{Proof of Theorem \ref{th:main}}\label{sect:Proof}
In this section, we will prove our main result, that is, Theorem \ref{th:main}. 
The proof is divided into three cases: the case $m\geq6$, the case $m=5$, and the critical case $m=4$. 
In the following, we assume that the initial data $(\bm{x}_0^\mathrm{in},\bm{x}_1^\mathrm{in})$ satisfy \eqref{CondID}.

%------------------------------------------------------------------------------
\subsection{The case $m\geq6$}
In this subsection, we will prove Theorem \ref{th:main} in the case $m\geq6$. 
Let $\{\bm{x}_j^\mathrm{in}\}_{j=0}^m$ and $\{\tau_j^\mathrm{in}\}_{j=0}^{m-2}$ be the initial values determined from the initial data 
$(\bm{x}_0^\mathrm{in},\bm{x}_1^\mathrm{in})$ as in Section \ref{sect:CC}. 
Putting $(\bm{y}_0^\mathrm{in},\bm{y}_1^\mathrm{in})=(\bm{x}_2^\mathrm{in},\bm{x}_3^\mathrm{in})$, 
we consider the initial boundary value problem to the quasilinear system \eqref{QLy}--\eqref{QLxtau} for unknown $(\bm{x},\bm{y},\tau,\nu)$. 
Since the initial data $(\bm{x}_0^\mathrm{in},\bm{x}_1^\mathrm{in})\in X^m\times X^{m-1}$ are assumed to satisfy the compatibility conditions up to order $m-1$ 
and the stability condition \eqref{CondID3}, we see by Lemma \ref{lem:EstQLIV} that $\bm{y}_0^\mathrm{in}\in X^{m-2}$, 
$\bm{y}_1^\mathrm{in}, \tau_0^{\mathrm{in}\prime},\tau_1^{\mathrm{in}\prime}\in X^{m-3}$, and that the initial data 
$(\bm{x}_0^\mathrm{in},\bm{x}_1^\mathrm{in},\bm{y}_0^\mathrm{in},\bm{y}_1^\mathrm{in},\tau_0^\mathrm{in},\tau_1^\mathrm{in})$ for the problem 
\eqref{QLy}--\eqref{QLxtau} satisfy the compatibility conditions up to order $m-3$. 
Therefore, we can apply Proposition \ref{prop:QLWP} with $m$ replaced with $m-2$ to obtain a unique solution $(\bm{x},\bm{y},\tau,\nu)$ 
to the quasilinear system \eqref{QLy}--\eqref{QLxtau} on some time interval $[0,T]$ satisfying 
$\bm{x},\dot{\bm{x}},\ddot{\bm{x}} \in \mathscr{X}_T^{m-2}$, $\tau',\dot{\tau}',\ddot{\tau}' \in \mathscr{X}_T^{m-3,*}$, 
and the stability condition \eqref{SC}.

Moreover, we see that $(\bm{x},\tau)$ satisfy \eqref{HP2} and \eqref{BVP2}. 
In fact, putting $\bm{r}=\ddot{\bm{x}}-(\tau\bm{x}')' - \bm{g}$ we see from \eqref{QLy} and \eqref{QLxtau} that $\ddot{\bm{r}}=\bm{0}$. 
Since we have set $(\bm{y}_0^\mathrm{in},\bm{y}_1^\mathrm{in})=(\bm{x}_2^\mathrm{in},\bm{x}_3^\mathrm{in})$, 
it holds also that $(\bm{r},\dot{\bm{r}})|_{t=0}=(\bm{0},\bm{0})$. 
Therefore, we obtain $\bm{r}\equiv\bm{0}$, which implies \eqref{HP2}. 
Putting $r=-\tau''+|\bm{x}''|^2\tau - |\dot{\bm{x}}'|^2$ we see from \eqref{QLnu} and \eqref{QLxtau} that $\ddot{r}=0$ and $(r,\dot{r})|_{t=0}=(0,0)$, 
so that $r\equiv0$, which implies the first equation in \eqref{BVP2}. 
Putting $a=(\tau'+\bm{g}\cdot\bm{x}')|_{s=1}$, we see that $\ddot{a}=0$ and $(a,\dot{a})|_{t=0}=(0,0)$, so that $a\equiv0$, 
which implies the third equation in \eqref{BVP2}. 
Similarly, we have $\ddot{\tau}|_{s=0}=0$ and $(\tau,\dot{\tau})|_{s=1}=(0,0)$, so that $\tau|_{s=1}\equiv0$, which implies the second equation in \eqref{BVP2}.

We proceed to show the regularity of the solution, that is, $\bm{x}\in\mathscr{X}_T^m$ and $\tau'\in\mathscr{X}_T^{m-1,*}$. 
Since we already knew that $\ddot{\bm{x}}\in\mathscr{X}_T^{m-2}$ and $\ddot{\tau}\in\mathscr{X}_T^{m-3,*}$, it is sufficient to show that 
$\dt^j\bm{x}\in C^0([0,T];X^{m-j})$ and $\dt^j\tau'\in C^0([0,T];X^{m-1-j})$ for $j=0,1$. 
To this end, we consider the following preliminary problem 
\begin{equation}\label{EP}
(\tau u')'=f \quad\mbox{in}\quad (0,1)\times(0,T). 
\end{equation}

\begin{lemma}\label{lem:reg}
Let $j$ be a non-negative integer and $M_0>0$ and assume that $\tau$ satisfies $M_0^{-1}s \leq \tau(s,t)\leq M_0s$ and $\tau'\in C^0([0,T];X^{j\vee2})$ 
and that $f\in C^0([0,T];X^j)$. 
If $u\in C^0([0,T];L^2)$ solves \eqref{EP}, then we have $u\in C^0([0,T];X^{j+2})$. 
\end{lemma}

\begin{remark}\label{re:reg}
Let $D$ be the unit disc in $\mathbb{R}^2$ centered at the origin. 
For a function $u$ defined in the open interval $(0,1)$, we put $u^\sharp(x_1,x_2)=u(x_1^2+x_2^2)$ for $x=(x_1,x_2)\in D$. 
Then, \eqref{EP} is transformed into the boundary value problem 
\[
\begin{cases}
 \nabla\cdot(\mu^\sharp\nabla u^\sharp)=f^\sharp &\mbox{in}\quad D\times(0,T), \\
 N^\perp\cdot\nabla u^\sharp = 0 &\mbox{on}\quad \partial D\times(0,T),
\end{cases}
\]
where $\mu(s,t)=\frac{1}{s}\tau(s,t)$ and $N$ is the unit normal vector on $\partial D$, so that $N^\perp\cdot\nabla$ is a tangential derivative on $\partial D$. 
By Lemma \ref{lem:NormEq}, the condition $f\in C^0([0,T];X^j)$ is transformed into $f^\sharp \in C^0([0,T];H^j(D))$. 
Therefore, if we suppose a sufficient regularity on $\mu$, then the result of this lemma follows from a standard elliptic regularity theorem. 
Here, we will give a direct proof of this lemma. 
\end{remark}

\begin{proof}[Proof of Lemma \ref{lem:reg}]
In view of the identity $\|u\|_{X^{j+2}}^2 = \|u\|_{L^2}^2 + \|u'\|_{X^j}^2 + \|s^{\frac{j}{2}+1}\ds^{j+2}u\|_{L^2}^2$, 
it is sufficient to show that $u' \in C^0([0,T];X^j)$ and $s^{\frac{j}{2}+1}\ds^{j+2}u \in C^0([0,T];L^2)$. 
We introduce $\mu(s,t)=\frac{1}{s}\tau(s,t)=(\mathscr{M}(\tau'(\cdot,t)))(s)$, where $\mathscr{M}$ is an averaging operator defined by \eqref{AvOp}. 
By Lemma \ref{lem:AvOp}, we have $\mu\in C^0([0,T];X^{j\vee2})$. 
We note also that $\mu$ is strictly positive, that is, $\mu(s,t)\geq M_0^{-1}$. 
We first evaluate $u'$. 
Integrating \eqref{EP} over $[0,s]$, we have $u'=\frac{1}{\mu}\mathscr{M}f$, so that 
\[
\|u'\|_{X^j}
\lesssim \|\tfrac{1}{\mu}\|_{X^{j\vee2}} \|\mathscr{M}f\|_{X^j}
 \lesssim C(\|\mu\|_{X^{j\vee2}}) \|f\|_{X^j} 
 \lesssim C(\|\tau'\|_{X^{j\vee2}}) \|f\|_{X^j}.
\]
To evaluate $s^{\frac{j}{2}+1}\ds^{j+2}u$, we differentiate \eqref{EP} $j$-times with respect to $s$ to obtain 
$\tau\ds^{j+2}u = \ds^j f - [\ds^{j+1},\tau]u'$, so that 
\begin{align*}
\|s^{\frac{j}{2}+1}\ds^{j+2}u\|_{L^2}
&\lesssim \|s^\frac{j}{2}\ds^j f\|_{L^2} + \|s^\frac{j}{2}[\ds^{j+1},\tau]u'\|_{L^2} \\
&\lesssim \|f\|_{X^j} + \|\tau'\|_{X^{j\vee2}}\|u'\|_{X^j},
\end{align*}
where we used a commutator estimate given in \cite[Lemma 4.9]{IguchiTakayama2024}. 
Summarizing the above estimates, we obtain 
\[
\|u(t)\|_{X^{j+2}} \leq \|u(t)\|_{L^2} + C(\|\tau'(t)\|_{X^{j\vee2}}) \|f(t)\|_{X^j}.
\]
This shows that $u\in L^\infty(0,T;X^{j+2})$. 
To show the strong continuity in time, it is sufficient to evaluate $u(s,t_1)-u(s,t_2)$ as above. 
\end{proof}

We go back to show the desired regularity of the solution $(\bm{x},\tau)$ obtained above by bootstrap arguments. 
We already knew that 
\begin{equation}\label{reg1}
\begin{cases}
 (\tau\bm{x}')'=\ddot{\bm{x}}-\bm{g} \in C^0([0,T];X^{m-2}), \\
 \tau' \in C^0([0,T];X^{m-3}),
\end{cases}
\end{equation}
so that by Lemma \ref{lem:reg} we obtain $\bm{x} \in C^0([0,T];X^{m-1})$. 
Differentiating the first equation in \eqref{Eq} with respect to $t$ and using Lemma \ref{lem:EstAu}, we see that 
\[
\begin{cases}
 (\tau\dot{\bm{x}}')'=\dddot{\bm{x}}-(\dot{\tau}\bm{x}')' \in C^0([0,T];X^{m-3}), \\
 \tau' \in C^0([0,T];X^{m-3}),
\end{cases}
\]
so that by Lemma \ref{lem:reg} we obtain $\dot{\bm{x}} \in C^0([0,T];X^{m-1})$. 
Then, we evaluate $\tau'$. 
In view of the identity $\|\tau'\|_{X^{j+1}}^2=\|\tau'\|_{L^2}^2+\|\tau''\|_{Y^j}^2$, 
it is sufficient to evaluate $\tau''$ in the space $Y^j$. 
By the first equation in \eqref{BVP} and Lemmas \ref{lem:CalIneq1} and \ref{lem:CalIneq2}, 
we see that $\tau''=\tau\bm{x}''\cdot\bm{x}''-\dot{\bm{x}}'\cdot\dot{\bm{x}}' \in C^0([0,T];Y^{m-3})$, so that $\tau' \in C^0([0,T];X^{m-2})$. 
This together with \eqref{reg1} and Lemma \ref{lem:reg} implies $\bm{x} \in C^0([0,T];X^m)$. 
Then, by Lemmas \ref{lem:CalIneq1} and \ref{lem:CalIneq2} we see that 
$\tau''=\tau\bm{x}''\cdot\bm{x}''-\dot{\bm{x}}'\cdot\dot{\bm{x}}' \in C^0([0,T];Y^{m-2})$, so that $\tau' \in C^0([0,T];X^{m-1})$. 
Similarly, by Lemmas \ref{lem:CalIneq1} and \ref{lem:CalIneq2} we see that 
$\dot{\tau}''=\dot{\tau}\bm{x}''\cdot\bm{x}'' + 2\tau\dot{\bm{x}}''\cdot\bm{x}''-2\ddot{\bm{x}}'\cdot\dot{\bm{x}}' \in C^0([0,T];Y^{m-3})$, 
so that $\dot{\tau}' \in C^0([0,T];X^{m-2})$.

We have shown the existence of a solution $(\bm{x},\tau)$ to the problem \eqref{HP2}--\eqref{IC2} 
satisfying $\bm{x}\in\mathscr{X}_T^m$, $\tau'\in\mathscr{X}_T^{m-1,*}$, and the stability condition \eqref{SC}. 
This solution is not necessarily a solution to the original problem \eqref{Eq}--\eqref{IC} 
if the initial data $\bm{x}_0^\mathrm{in}$ does not satisfy the constraints $|\bm{x}_0^{\mathrm{in}\prime}(s)|\equiv1$ 
and $\bm{x}_0^{\mathrm{in}\prime}(s)\cdot\bm{x}_1^{\mathrm{in}\prime}(s)\equiv0$. 
However, we have assumed in Theorem \ref{th:main} that the initial data $(\bm{x}_0^\mathrm{in},\bm{x}_1^\mathrm{in})$ satisfy these constraints, 
so that we can apply \cite[Theorem 2.5]{IguchiTakayama2024} to ensure that the solution satisfies $|\bm{x}'(s,t)|\equiv1$. 
Therefore, $(\bm{x},\tau)$ is in fact a solution to the original problem \eqref{Eq}--\eqref{IC}. 
The uniqueness of solutions has been shown in \cite[Theorem 2.3]{IguchiTakayama2024}. 
Now, the proof is complete. 
\hfill$\Box$

%------------------------------------------------------------------------------
\subsection{The case $m=5$}
In this subsection, we will prove Theorem \ref{th:main} in the case $m=5$. 
By Proposition \ref{prop:AppID2}, we can approximate the initial data $(\bm{x}_0^\mathrm{in},\bm{x}_1^\mathrm{in})$ by 
$\{(\bm{x}_0^{\mathrm{in}(n)},\bm{x}_1^{\mathrm{in}(n)})\}_{n=1}^\infty \subset X^6\times X^5$, 
which converge to $(\bm{x}_0^\mathrm{in},\bm{x}_1^\mathrm{in})$ in $X^5\times X^4$ and satisfy 
$|\bm{x}_0^{\mathrm{in}(n)\prime}(s)|=1$ and $\bm{x}_0^{\mathrm{in}(n)\prime}(s)\cdot\bm{x}_1^{\mathrm{in}(n)\prime}(s)=0$ for $0<s<1$, 
and the compatibility conditions up to order $5$. 
Then, as in \cite[Section 8]{IguchiTakayama2024} we see that the corresponding initial tensions satisfy 
$\tau_0^{\mathrm{in}(n)\prime} \to \tau_0^{\mathrm{in}\prime}$ 
in $X^4$, so that without loss of generality we can assume that $\frac{1}{s}\tau_0^{\mathrm{in}(n)}(s) \geq \frac32c_0$ for $0<s<1$. 
Therefore, by the result obtained in the previous subsection, for each $n\in\mathbb{N}$ the initial boundary value problem \eqref{Eq}--\eqref{IC} 
with this initial data $(\bm{x}_0^{\mathrm{in}(n)},\bm{x}_1^{\mathrm{in}(n)})$ has a unique solution $(\bm{x}^{(n)},\tau^{(n)})$ 
on some time interval $[0,T_n]$ satisfying $\bm{x}^{(n)}\in\mathscr{X}_{T_n}^6$, $\tau^{(n)\prime}\in\mathscr{X}_{T_n}^{5,*}$, 
and the stability condition \eqref{SC}. 
Now, we apply the following a priori estimate, which was essentially given in \cite[Theorem 2.1]{IguchiTakayama2024}.

\begin{proposition}\label{prop:APE}
For any integer $m\geq 4$ and any positive constants $M_0$, $M_*$, and $c_0$, 
there exist a positive time $T=T(M_*,c_0)$ 
and a large constant $C=C(M_0,M_*,c_0,m)$ such that if the initial data satisfy 
\begin{equation}\label{CondID_APE}
\begin{cases}
 \|\bm{x}_0^\mathrm{in}\|_{X^m}+\|\bm{x}_1^\mathrm{in}\|_{X^{m-1}} \leq M_0, \qquad
 \|\bm{x}_0^\mathrm{in}\|_{X^4}+\|\bm{x}_1^\mathrm{in}\|_{X^3} \leq M_*, 
 \\
 \frac{\tau_0^\mathrm{in}(s)}{s} \geq 2c_0 \quad\mbox{for}\quad 0<s<1, 
\end{cases}
\end{equation}
where $\tau_0^\mathrm{in}(s)=\tau(s,0)$ is the initial tension, 
then any regular solution $(\bm{x},\tau)$ to the initial boundary value problem \eqref{Eq}--\eqref{IC} satisfies the stability condition \eqref{SC}, 
$\opnorm{\bm{x}(t)}_m \leq C$, and 
\[
\begin{cases}
 C^{-1}s \leq \tau(s,t) \leq Cs, \quad \sum_{j=1}^{m-3}|\dt^j\tau(s,t)|\leq Cs, \quad |\dt^{m-2}\tau(s,t)|\leq Cs^\frac12; \\
 \opnorm{\tau'(t)}_{m-1,*} \leq C \qquad\ \mbox{in the case $m\geq5$}, \\
 \opnorm{\tau'(t)}_{3,*,\epsilon} \leq C(\epsilon) \qquad\mbox{in the case $m=4$}
\end{cases}
\]
for $0\leq t\leq T$ and $\epsilon>0$, 
where the constant $C(\epsilon)$ depends also on $\epsilon$. 
\end{proposition}

\begin{remark}
This proposition gives a slightly better estimate than \cite[Theorem 2.1]{IguchiTakayama2024} on the existence time $T$, 
which does not depend on the higher order norm of the initial data. 
For the sake of completeness, we will sketch a proof of this proposition in Section \ref{sect:APE}. 
\end{remark}

By this proposition, there exist a positive time $T$ and a large constant $C$ independent of $n$ 
such that we can extend the solution $(\bm{x}^{(n)},\tau^{(n)})$ up to $T$ 
and the solution satisfies the uniform bounds 
$\opnorm{\bm{x}^{(n)}(t)}_5 + \opnorm{\tau^{(n)\prime}(t)}_{4,*} \leq C$ for $t\in[0,T]$. 
Here, by Lemma \ref{lem:NormEq} we see that the embedding $X^{j+1}\hookrightarrow X^j$ is compact so that by the Aubin--Lions lemma the embeddings 
$\mathscr{X}_T^5\hookrightarrow\mathscr{X}_T^4$ and $\mathscr{X}_T^{4,*}\hookrightarrow\mathscr{X}_T^{3,*}$ are also compact. 
Therefore, $\{(\bm{x}^{(n)},\tau^{(n)})\}_{n=1}^\infty$ has a subsequence $\{(\bm{x}^{(n_j)},\tau^{(n_j)})\}_{j=1}^\infty$ such that 
$\{\bm{x}^{(n_j)}\}_{j=1}^\infty$ converges to a $\bm{x}$ in $\mathscr{X}_T^4$ and that $(\tau^{(n_j)},\tau^{(n_j)\prime})\}_{j=1}^\infty$ to 
$(\tau,\tau')$ in $X^{3,*}$. 
Obviously, $(\bm{x},\tau)$ is a unique solution to \eqref{Eq}--\eqref{IC}. 
As a result, without taking a subsequence, $\{(\bm{x}^{(n)},\tau^{(n)\prime})\}_{n=1}^\infty$ converge to $(\bm{x},\tau')$ in 
$\mathscr{X}_T^4\times\mathscr{X}_T^{3,*}$. 
Moreover, by standard compactness arguments we have also 
\begin{equation}\label{APE}
\opnorm{\bm{x}(t)}_5 + \opnorm{\tau'(t)}_{4,*} \leq C \quad\mbox{for}\quad t\in[0,T].
\end{equation}

It remains to show that the solution has the regularity $\bm{x}\in\mathscr{X}_T^5$ and $\tau'\in\mathscr{X}_T^{4,*}$. 
We put $(\bm{y},\nu)=(\ddot{\bm{x}},\ddot{\tau})$, which satisfy \eqref{QLy} and \eqref{QLnu}, where 
$(\bm{y}_0^\mathrm{in},\bm{y}_1^\mathrm{in})=(\bm{x}_2^\mathrm{in},\bm{x}_3^\mathrm{in})$ and $(\bm{f},h)$ are defined by \eqref{QLfh}. 
By Lemma \ref{lem:EstIV}, we have $(\bm{y}_0^\mathrm{in},\bm{y}_1^\mathrm{in})\in X^3\times X^2$. 
Since we have already knew that $(\bm{x},\tau') \in \mathscr{X}_T^4 \times \mathscr{X}_T^{3,*}$ and that they satisfy \eqref{APE}, 
by Lemmas \ref{lem:EstAu}--\ref{lem:CalIneq1}, \ref{lem:CalIneq2}, and \ref{lem:CalIneqY3}, we have 
\[
\begin{cases}
 \opnorm{\bm{f}}_1 \lesssim \opnorm{\tau'}_{3,*} \opnorm{\bm{x}}_4, \\
 \opnorm{\bm{f}}_2 \lesssim \opnorm{\tau'}_{4,*} \opnorm{\bm{x}}_5, \\
 \|h\|_{Y^0} + \|s^\frac12\dt h\|_{L^1} \lesssim (1+\opnorm{\tau'}_{3,*})\opnorm{\bm{x}}_4^2, \\
 \|s^\frac12\dt^2 h\|_{L^1} \lesssim (1+\opnorm{\tau'}_{4,*})\opnorm{\bm{x}}_5^2.
\end{cases}
\]
Particularly, we have $\bm{f}\in\mathscr{X}_T^1$, $\dt^2\bm{f}\in L^1(0,T;L^2)$, $h\in\mathscr{Y}_T^0$, $s^\frac12\dt h\in C^0([0,T];L^1)$, and 
$s^\frac12\dt^2 h\in L^1((0,1)\times(0,T))$. 
Therefore, we can apply the existence and uniqueness theorem \cite[Theorem 2.8]{IguchiTakayama2023-2} to obtain 
$\ddot{\bm{x}}\in\mathscr{X}_T^3$ and $\ddot{\tau}'\in\mathscr{X}_T^{1,*}$. 
Then, by Lemmas \ref{lem:CalIneqY1} and \ref{lem:CalIneqY2}, we see that $h\in\mathscr{Y}_T^1$. 
Hence, we can apply \cite[Theorem 2.8]{IguchiTakayama2023-2} 
again to obtain $\ddot{\tau}'\in\mathscr{X}_T^{2,*}$. 
We have already knew that 
\[
\begin{cases}
 (\tau\bm{x}')'=\ddot{\bm{x}}-\bm{g} \in C^0([0,T];X^3), \\
 \tau' \in C^0([0,T];X^3),
\end{cases}
\]
so that by Lemma \ref{lem:reg} we obtain $\bm{x} \in C^0([0,T];X^5)$. 
Similarly, we have $(\tau\dot{\bm{x}}')'=\dddot{\bm{x}}-(\dot{\tau}\bm{x}')' \in C^0([0,T];X^2)$, 
so that by Lemma \ref{lem:reg} again we obtain $\dot{\bm{x}} \in C^0([0,T];X^4)$. 
Therefore, we get $\bm{x}\in\mathscr{X}_T^5$. 
As before, to evaluate $\tau'$ in $X^{j+1}$ it is sufficient to evaluate $\tau''$ in $Y^j$. 
By Lemmas \ref{lem:CalIneq1} and \ref{lem:CalIneq2}, we see that $\tau''=\tau\bm{x}''\cdot\bm{x}''-\dot{\bm{x}}'\cdot\dot{\bm{x}}' \in C^0([0,T];Y^3)$, 
so that $\tau' \in C^0([0,T];X^4)$. 
We see also that $\dot{\tau}''=\dot{\tau}\bm{x}''\cdot\bm{x}'' + 2\tau\dot{\bm{x}}''\cdot\bm{x}''-2\ddot{\bm{x}}'\cdot\dot{\bm{x}}' \in C^0([0,T];Y^2)$, 
so that $\dot{\tau}' \in C^0([0,T];X^3)$. 
Therefore, we get $\tau'\in\mathscr{X}_T^{4,*}$. 
Now, the proof is complete. 
\hfill$\Box$

%------------------------------------------------------------------------------
\subsection{The case $m=4$}
In this last subsection, we will prove Theorem \ref{th:main} in the critical case $m=4$. 
By Proposition \ref{prop:AppID2}, we can approximate the initial data $(\bm{x}_0^\mathrm{in},\bm{x}_1^\mathrm{in})$ by 
$\{(\bm{x}_0^{\mathrm{in}(n)},\bm{x}_1^{\mathrm{in}(n)})\}_{n=1}^\infty \subset X^5\times X^4$, 
which converges to $(\bm{x}_0^\mathrm{in},\bm{x}_1^\mathrm{in})$ in $X^4\times X^3$ and satisfies 
$|\bm{x}_0^{\mathrm{in}(n)\prime}(s)|=1$ and $\bm{x}_0^{\mathrm{in}(n)\prime}(s)\cdot\bm{x}_1^{\mathrm{in}(n)\prime}(s)=0$ for $0<s<1$, 
and the compatibility conditions up to order $4$. 
Then, as in \cite[Section 8]{IguchiTakayama2024} we see that the corresponding initial tensions satisfy 
$\tau_0^{\mathrm{in}(n)\prime} \to \tau_0^{\mathrm{in}\prime}$ in $L^\infty$, 
so that without loss of generality we can assume that $\frac{1}{s}\tau_0^{\mathrm{in}(n)}(s) \geq \frac32c_0$ for $0<s<1$. 
Therefore, by the result obtained in the previous subsection, for each $n\in\mathbb{N}$ the initial value problem \eqref{Eq}--\eqref{IC} 
with this initial data $(\bm{x}_0^{\mathrm{in}(n)},\bm{x}_1^{\mathrm{in}(n)})$ has a unique solution $(\bm{x}^{(n)},\tau^{(n)})$ on some time interval 
$[0,T_n]$ satisfying $\bm{x}^{(n)}\in\mathscr{X}_{T_n}^5$, $\tau^{(n)\prime}\in\mathscr{X}_{T_n}^{4,*}$, and the stability condition \eqref{SC}. 
Then, by Proposition \ref{prop:APE}, there exist a positive time $T$ and a large constant $C$ independent of $n$ 
such that we can extend the solution $(\bm{x}^{(n)},\tau^{(n)})$ up to $T$ and the solution satisfies the uniform bounds 
\begin{equation}\label{APE2}
\opnorm{\bm{x}^{(n)}(t)}_4 + \opnorm{\tau^{(n)\prime}(t)}_{3,*,\epsilon} \leq C(\epsilon) \quad\mbox{for}\quad t\in[0,T]
\end{equation}
for any $\epsilon>0$, where $C(\epsilon)$ is a positive constant depending on $\epsilon$, but not on $n$ nor $t$. 
We know that the embedding $\mathscr{X}_T^4\hookrightarrow\mathscr{X}_T^3$ is compact. 
On the other hand, for any $p\in[1,2)$ we have $\|u\|_{L^p} \leq C(p,\epsilon)\|s^\epsilon u\|_{L^2}$ for any $\epsilon\in(0,\frac1p-\frac12)$. 
Therefore, we have 
\begin{align*}
& \|u(t)\|_{L^p}+\|u'(t)\|_{L^p}+\|s^\frac12 u''(t)\|_{L^p}+\|s^\frac 32u'''(t)\|_{L^p} \\
& + \|\dot{u}(t)\|_{L^p}+\|\dot{u}'(t)\|_{L^p}+\|s \dot{u}''(t)\|_{L^p}
  + \|\ddot{u}(t)\|_{L^p}+\|s^\frac12 \ddot{u}'(t)\|_{L^p} \leq C(p,\epsilon) \opnorm{u(t)}_{3,*,\epsilon}
\end{align*}
for any $\epsilon\in(0,\frac1p-\frac12)$. 
Particularly, as in the proof of \cite[Proposition 3.2]{Takayama2018} we obtain 
\begin{gather*}
\|u^\sharp(t)\|_{W^{3,p}(D)} + \|\dot{u}^\sharp(t)\|_{W^{2,p}(D)} + \|\ddot{u}^\sharp(t)\|_{W^{1,p}(D)} \leq C(p,\epsilon) \opnorm{u(t)}_{3,*,\epsilon}
\end{gather*}
for any $\epsilon\in(0,\frac1p-\frac12)$, where $u^\sharp(x,t)=u(x_1^2+x_2^2,t)$. 
Since the embedding $W^{j+1,p}(D) \hookrightarrow W^{j,2}(D)$ is compact if $p>1$, 
this estimate implies that the embedding $\mathscr{X}_T^{3,*,\epsilon} \hookrightarrow \mathscr{X}_T^{2,*}$ is also compact. 
Therefore, $\{(\bm{x}^{(n)},\tau^{(n)})\}_{n=1}^\infty$ has a subsequence $\{(\bm{x}^{(n_j)},\tau^{(n_j)})\}_{j=1}^\infty$ such that 
$\{\bm{x}^{(n_j)}\}_{j=1}^\infty$ converges to a $\bm{x}$ in $\mathscr{X}_T^3$ and that $(\tau^{(n_j)},\tau^{(n_j)\prime})\}_{j=1}^\infty$ to 
$(\tau,\tau')$ in $\mathscr{X}_T^{2,*}$. 
Particularly, $(\bm{x},\tau)$ is a solution of \eqref{Eq}--\eqref{IC} satisfying the stability condition \eqref{SC}. 
Moreover, by standard compactness arguments, without loss of generality we see also that $\{\dt^i\bm{x}^{(n_j)}\}_{j=1}^\infty$ converges to $\dt^i\bm{x}$ 
in $L^\infty(0,T;X^{4-i})$ in the weak-$*$ topology for $i=0,1,\ldots,4$. 
Therefore, we can apply the uniqueness theorem in \cite[Theorem 2.3]{IguchiTakayama2024} to verify that $(\bm{x},\tau)$ is a unique solution of 
\eqref{Eq}--\eqref{IC}. 
As a result, without taking a subsequence, $\{(\bm{x}^{(n)},\tau^{(n)\prime})\}_{n=1}^\infty$ converge to $(\bm{x},\tau')$ in 
$\mathscr{X}_T^3\times\mathscr{X}_T^{2,*}$, which satisfies 
$\opnorm{\bm{x}(t)}_4 + \opnorm{\tau'(t)}_{3,*,\epsilon} \leq C(\epsilon)$ for $t\in[0,T]$ and for any $\epsilon>0$. 
We see also that $\dt^j\bm{x} \in C_\mathrm{w}([0,T];X^{4-j})$ for $j=0,1,2,3$.

It remains to show the strong continuity in time, that is, 
the solution has the regularity $\bm{x}\in\mathscr{X}_T^4$ and $\tau'\in\mathscr{X}_T^{3,*,\epsilon}$ for any $\epsilon>0$. 
In this critical case $m=4$, we do not have enough regularity on $\bm{x}$ to apply directly the existence and uniqueness theorem 
\cite[Theorem 2.8]{IguchiTakayama2023-2} to obtain the regularity $\ddot{\bm{x}}\in\mathscr{X}_T^2$. 
Instead, we will apply the energy estimate obtained in \cite[Section 6]{IguchiTakayama2024} by utilizing fully the constraint $|\bm{x}'|=1$. 
We put $(\bm{y}^{(n)},\nu^{(n)})=(\ddot{\bm{x}}^{(n)},\ddot{\tau}^{(n)})$, which solves \eqref{QLy} and \eqref{QLnu} with $(\bm{f},h)$ replaced with 
$(\bm{f}^{(n)},h^{(n)})$, which are defined by \eqref{QLfh} with $(\bm{x},\tau)$ replaced with $(\bm{x}^{(n)},\tau^{(n)})$. 
Moreover, we have the following almost orthogonality 
\[
\bm{x}^{(n)\prime}\cdot\bm{y}^{(n)\prime} = f^{(n)}
\quad\mbox{in}\quad (0,1)\times(0,T),
\]
where $f^{(n)}=-|\dot{\bm{x}}^{(n)\prime}|^2$. 
Following \cite[Section 6]{IguchiTakayama2024}, we define an energy functional $\mathscr{E}^{(n)}(t)$ by 
\begin{align*}
\mathscr{E}^{(n)}(t)
&= (\tau^{(n)}\dot{\bm{y}}^{(n)\prime},\dot{\bm{y}}^{(n)\prime})_{L^2}^2 + \|(\tau^{(n)}\bm{y}^{(n)\prime})'\|_{L^2}^2
 + 2((\tau^{(n)}\bm{y}^{(n)\prime})',\bm{f}^{(n)})_{L^2} - \biggl( \frac{\varphi^{(n)\prime}}{\varphi^{(n)}}(\nu^{(n)})^2 \biggr)\biggr|_{s=1},
\end{align*}
where $\varphi^{(n)}(\cdot,t)$ is a unique solution to the initial value problem 
\[
\begin{cases}
 -\varphi^{(n)\prime\prime}(\cdot,t)+|\bm{x}^{(n)\prime\prime}(\cdot,t)|^2\varphi^{(n)}(\cdot,t)=0 \quad\mbox{in}\quad (0,1), \\
 \varphi^{(n)}(0,t)=0, \quad \varphi^{(n)\prime}(0,t)=1.
\end{cases}
\]
Since we have the uniform bound \eqref{APE2} of the solution $(\bm{x}^{(n)},\tau^{(n)})$, as in the proof of \cite[Proposition 6.2]{IguchiTakayama2024} 
with $\epsilon=\frac18$, we obtain 
\begin{align*}
\biggl| \frac{\mathrm{d}}{\mathrm{d}t}\mathscr{E}^{(n)}(t) \biggr|
&\lesssim \opnorm{\bm{y}^{(n)}}_{2,*}^2 + \|(\bm{f}^{(n)},\dot{\bm{f}}^{(n)})\|_{L^2}^2 + \|s^\frac38\dot{f}^{(n)}\|_{L^2}^2 + |\dot{f}^{(n)}|_{s=1}|^2 \\
&\quad\;
 + \|s\dot{h}^{(n)}\|_{L^1}^2 + \|s^\frac58h^{(n)}\|_{L^2}^2.
\end{align*}
Since we have the uniform bound \eqref{APE2}, we easily obtain $\opnorm{\bm{y}^{(n)}}_{2,*}^2 + \|\bm{f}^{(n)}\|_{L^2}^2 + |\dot{f}^{(n)}|_{s=1}|^2 \lesssim 1$. 
We evaluate the other terms as follows: 
\begin{align*}
\|\dot{\bm{f}}^{(n)}\|_{L^2}^2
&\lesssim \|\ddot{\tau}^{(n)\prime}\dot{\bm{x}}^{(n)\prime}\|_{L^2} + \|\ddot{\tau}^{(n)}\dot{\bm{x}}^{(n)\prime\prime}\|_{L^2}
 + \|(\dot{\tau}^{(n)}\ddot{\bm{x}}^{(n)\prime})'\|_{L^2} \\
&\lesssim \|s^\frac18\ddot{\tau}^{(n)\prime}\|_{L^\infty}\|s^{-\frac18}\|_{L^4}\|\dot{\bm{x}}^{(n)\prime}\|_{L^4}
 + \|s^\frac12\ddot{\tau}^{(n)\prime}\|_{L^\infty}\|s^\frac12\dot{\bm{x}}^{(n)\prime\prime}\|_{L^2}
 + \|\dot{\tau}^{(n)\prime}\|_{L^\infty} \|\ddot{\bm{x}}^{(n)}\|_{X^2} \\
&\lesssim \|s^\frac18\ddot{\tau}^{(n)\prime}\|_{L^\infty}\|\dot{\bm{x}}^{(n)}\|_{X^3}
 + \|\dot{\tau}^{(n)\prime}\|_{L^\infty} \|\ddot{\bm{x}}^{(n)}\|_{X^2},
\end{align*}
\begin{align*}
\|s^\frac38\dot{f}^{(n)}\|_{L^2}
&\lesssim \|s^\frac18\dot{\bm{x}}^{(n)\prime}\|_{L^4} \|s^\frac14\ddot{\bm{x}}^{(n)\prime}\|_{L^4} \\
&\lesssim \|\dot{\bm{x}}^{(n)}\|_{X^3} \|\ddot{\bm{x}}^{(n)}\|_{X^2},
\end{align*}
\begin{align*}
\|s\dot{h}^{(n)}\|_{L^1}
&\lesssim \|s^\frac12\ddot{\bm{x}}^{(n)\prime}\|_{L^2} \|s^\frac12\dddot{\bm{x}}^{(n)\prime}\|_{L^2}
 + \|\tau^{(n)\prime}\|_{L^\infty}\|s\dot{\bm{x}}^{(n)\prime\prime}\|_{L^2} \|s\ddot{\bm{x}}^{(n)\prime\prime}\|_{L^2} \\
&\quad\;
 + \|\dot{\tau}^{(n)\prime}\|_{L^\infty} \bigl( \|s\dot{\bm{x}}^{(n)\prime\prime}\|_{L^2}^2
  + \|s\bm{x}^{(n)\prime\prime}\|_{L^2}\|s\ddot{\bm{x}}^{(n)\prime\prime}\|_{L^2} \bigr) \\
&\quad\;
 + \|s^\frac12\ddot{\tau}^{(n)\prime}\|_{L^\infty} \|\bm{x}^{(n)\prime\prime}\|_{L^2} \|s^\frac12\dot{\bm{x}}^{(n)\prime\prime}\|_{L^2} \\
&\lesssim \|\ddot{\bm{x}}^{(n)}\|_{X^1} \|\dddot{\bm{x}}^{(n)}\|_{X^1}
 + \|\tau^{(n)\prime}\|_{L^\infty} \|\dot{\bm{x}}^{(n)}\|_{X^2} \|\ddot{\bm{x}}^{(n)}\|_{X^2} \\
&\quad\;
 + \|\dot{\tau}^{(n)\prime}\|_{L^\infty} \bigl( \|\dot{\bm{x}}^{(n)}\|_{X^2}^2 + \|\bm{x}^{(n)}\|_{X^2}\|\ddot{\bm{x}}^{(n)}\|_{X^2} \bigr)
 + \|s^\frac12\ddot{\tau}^{(n)\prime}\|_{L^\infty} \|\bm{x}^{(n)}\|_{X^4} \|\dot{\bm{x}}^{(n)}\|_{X^3}, 
\end{align*}
and 
\begin{align*}
\|s^\frac12h^{(n)}\|_{L^2}^2
&\lesssim \|s^\frac14\ddot{\bm{x}}^{(n)\prime}\|_{L^4}^2 + \|\tau^{(n)\prime}\|_{L^\infty}\|s^\frac34\dot{\bm{x}}^{(n)\prime\prime}\|_{L^4}^2
 + \|\dot{\tau}^{(n)\prime}\|_{L^\infty}\|s^\frac34\bm{x}^{(n)\prime\prime}\|_{L^4}\|s^\frac34\dot{\bm{x}}^{(n)\prime\prime}\|_{L^4} \\
&\lesssim \|\ddot{\bm{x}}^{(n)}\|_{X^2}^2 + \|\tau^{(n)\prime}\|_{L^\infty}\|\dot{\bm{x}}^{(n)}\|_{X^3}^2
 + \|\dot{\tau}^{(n)\prime}\|_{L^\infty} \|\bm{x}^{(n)}\|_{X^3} \|\dot{\bm{x}}^{(n)}\|_{X^3}.
\end{align*}
Therefore, we obtain $\bigl|\frac{\mathrm{d}}{\mathrm{d}t}\mathscr{E}^{(n)}(t)\bigr| \lesssim 1$, so that 
$| \mathscr{E}^{(n)}(t)-\mathscr{E}^{(n)}(t_0) | \lesssim |t-t_0|$. 
Using this estimate together with the technique used in \cite{Majda1984, MajdaBertozzi2002}, we see that the corresponding energy functional 
$\mathscr{E}(t)$ to the solution $(\bm{y},\nu)=(\ddot{\bm{x}},\ddot{\tau})$ 
satisfies $| \mathscr{E}(t)-\mathscr{E}(t_0) | \lesssim |t-t_0|$. 
Since $((\tau(t)\bm{y}'(t))',\bm{f}(t))_{L^2}$ and $\frac{\varphi'(1,t)}{\varphi(1,t)}(\nu(1,t))^2$ are both continuous in $t$, 
we see that 
$(\tau(t)\dot{\bm{y}}'(t),\dot{\bm{y}}'(t))_{L^2}+\|(\tau(t)\bm{y}'(t))'\|_{L^2}^2$ is also continuous in $t$. 
This fact together with the weak continuity $\dt^j\bm{x} \in C_\mathrm{w}([0,T];X^{4-j})$ for $j=2,3$ implies the strong continuity, that is, 
$\dt^j\bm{x} \in C^0([0,T];X^{4-j})$ for $j=2,3$. 
Therefore, we have 
\[
\begin{cases}
 (\tau\bm{x}')'=\ddot{\bm{x}}-\bm{g} \in C^0([0,T];X^2), \\
 \tau' \in C^0([0,T];X^2),
\end{cases}
\]
so that by Lemma \ref{lem:reg} we obtain $\bm{x} \in C^0([0,T];X^4)$. 
Then, we have $(\tau\dot{\bm{x}}')'=\dddot{\bm{x}}-(\dot{\tau}\bm{x}')' \in C^0([0,T];X^1)$, 
so that by Lemma \ref{lem:reg} again we obtain $\dot{\bm{x}} \in C^0([0,T];X^3)$. 
Therefore, we get $\bm{x}\in \mathscr{X}_T^{4,*}$. 
Once we obtain this regularity of $\bm{x}$, The regularity of $\tau$, that is, $\tau'\in\mathscr{X}_T^{3,*,\epsilon}$ 
can be obtained in the same way as the proof of \cite[Lemma 8.2]{IguchiTakayama2024}, so we omit it. 
Finally, from the first equation in \eqref{Eq}, this regularity of $(\bm{x},\tau')$ implies 
$\dt^4\bm{x}\in C([0,T];L^2)$, that is, $\bm{x}\in \mathscr{X}_T^4$. 
Now, the proof is complete. 
\hfill$\Box$

%----------------------------------------------------------------------------------------------------------------------
%----------------------------------------------------------------------------------------------------------------------
\section{Approximation of the initial data I}\label{sect:AppID1}
The objective in this section is to give a proof of Proposition \ref{prop:AppID1}, 
which ensure that the initial data $(\bm{x}_0^\mathrm{in},\bm{x}_1^\mathrm{in})$ for the problem \eqref{HP2}--\eqref{IC2} can be 
approximated by smooth initial data satisfying higher order compatibility conditions. 
The proof is divided into two steps. 
In the first step, we approximate the initial data by smooth data keeping the compatibility conditions up to the same order.

\begin{lemma}\label{lem:AppID1}
Let $m\geq1$ be an integer and assume that the initial data $(\bm{x}_0^\mathrm{in},\bm{x}_1^\mathrm{in})\in X^m\times X^{m-1}$ 
for the problem \eqref{HP2}--\eqref{IC2} satisfies the compatibility conditions up to order $m-1$. 
In the case $m\geq3$, we also assume $\tau_0^\mathrm{in}(1)>0$. 
Then, for any $\varepsilon>0$ there exist approximate initial data $(\bm{x}_0^{\mathrm{in},\varepsilon},\bm{x}_1^{\mathrm{in},\varepsilon}) \in C^\infty([0,1])$ 
satisfying the compatibility conditions up to order $m-1$ and 
$\|\bm{x}_0^{\mathrm{in},\varepsilon}-\bm{x}_0^{\mathrm{in}}\|_{X^m} + \|\bm{x}_1^{\mathrm{in},\varepsilon}-\bm{x}_1^{\mathrm{in}}\|_{X^{m-1}} < \varepsilon$. 
\end{lemma}

The second step is to compensate the smooth approximated initial data in order that they satisfy higher order compatibility conditions.

\begin{lemma}\label{lem:AppID2}
Let $m\geq1$ be an integer and assume that the initial data $(\bm{x}_0^\mathrm{in},\bm{x}_1^\mathrm{in})\in C^\infty([0,1])$ 
for the problem \eqref{HP2}--\eqref{IC2} satisfies the compatibility conditions up to order $m-1$. 
In the case $m\geq2$, we also assume that $\tau_0^\mathrm{in}(1)>0$. 
Then, for any $\varepsilon>0$ there exist approximate initial data $(\bm{x}_0^{\mathrm{in},\varepsilon},\bm{x}_1^{\mathrm{in},\varepsilon}) \in C^\infty([0,1])$ 
satisfying the compatibility conditions up to order $m$ and 
$\|\bm{x}_0^{\mathrm{in},\varepsilon}-\bm{x}_0^{\mathrm{in}}\|_{X^m} + \|\bm{x}_1^{\mathrm{in},\varepsilon}-\bm{x}_1^{\mathrm{in}}\|_{X^{m-1}} < \varepsilon$. 
\end{lemma}

Once we obtain these two lemmas, we can easily show Proposition \ref{prop:AppID1} by a standard technique. 
In the rest of this section, we prove these two lemmas.

%------------------------------------------------------------------------------
\subsection{Structure of the compatibility conditions}
In order to prove Lemmas \ref{lem:AppID1} and \ref{lem:AppID2}, we need to know the structure of the compatibility conditions in detail and 
to prepare several notations. 
Given initial data $(\bm{x}_0,\bm{x}_1)$, we let $(\bm{x},\tau)$ be a smooth solution of the problem \eqref{HP2} and \eqref{BVP2} under the initial conditions 
$(\bm{x},\dot{\bm{x}})|_{t=0}=(\bm{x}_0,\bm{x}_1)$. 
Then, we put $\bm{x}_j=(\dt^j\bm{x})|_{t=0}$ and $\tau_j=(\dt^j\tau)|_{t=0}$ for $j=0,1,2,\ldots$. 
As before, these initial values $\bm{x}_j$ and $\tau_j$ are determined uniquely up to certain order $j$ 
depending on the regularity of initial data $(\bm{x}_0,\bm{x}_1)$. 
We put $\bm{u}_j=\bm{x}_j'$ for $j=0,1,2,\ldots$. 
We note that the initial data $(\bm{x}_0,\bm{x}_1)$ can be recovered from $(\bm{u}_0,\bm{u}_1)$ by 
$\bm{x}_j(s)=-\int_s^1\bm{u}_j(\sigma)\mathrm{d}\sigma$ for $j=0,1$ due to the compatibility conditions at order $0$ and $1$. 
Therefore, $\bm{x}_j$ and $\tau_j$ are determined from $\bm{u}_0$ and $\bm{u}_1$. 
In fact, they are determined inductively from the following equations 
\begin{equation}\label{TPBVPj}
\begin{cases}
 -\tau_j''+|\bm{u}_0'|^2\tau_j = h_j \quad\mbox{in}\quad (0,1), \\
 \tau_j(0)=0, \quad \tau_j'(1) = -\bm{g}\cdot\bm{u}_j(1),
\end{cases}
\end{equation}
where 
\[
h_j = \sum_{j_1+j_2=j}\frac{j!}{j_1!j_2!}\bm{u}_{j_1+1}\cdot\bm{u}_{j_2+1}
 - \sum_{j_0+j_1+j_2=j,j_0\leq j-1}\frac{j!}{j_0!j_1!j_2!}(\bm{u}_{j_1}'\cdot\bm{u}_{j_2}')\tau_{j_0},
\]
and 
\begin{equation}\label{EqUj}
\bm{x}_{j+2} = \sum_{j_0+j_1=j}\frac{j!}{j_0!j_1!}(\tau_{j_0}\bm{u}_{j_1})' + \delta_{j0}\,\bm{g}, \qquad \bm{u}_{j+2} = \bm{x}_{j+2}'
\end{equation}
for $j=0,1,2,\ldots$, where $\delta_{j0}$ is the Kronecker delta. 
We proceed to express $\bm{x}_{j+2}(1)$ explicitly in terms of $\bm{u}_0$ and $\bm{u}_1$.

\begin{lemma}\label{lem:ExpX}
For any non-negative integer $j$, on $s=1$ we have 
\[
\begin{cases}
 \bm{x}_{2j+2} = \tau_0^{j+1}\ds^{2j+1}\bm{u}_0 + \mathcal{P}_{2j+2}(\ds^{\leq 2j}\bm{u}_0,\ds^{\leq 2j-1}\bm{u}_1,\tau_0,\ldots,\tau_{2j}), \\
 \bm{x}_{2j+3} = \tau_0^{j+1}\ds^{2j+1}\bm{u}_1 + \mathcal{P}_{2j+3}(\ds^{\leq 2j+1}\bm{u}_0,\ds^{\leq 2j}\bm{u}_1,\tau_0,\ldots,\tau_{2j+1}),
\end{cases}
\]
where $\mathcal{P}_m$ is a polynomial in the indicated variables and $\ds^{\leq m}u=\{\ds^ju | j=0,1,\ldots,m\}$. 

\end{lemma}

\begin{proof}
Throughout this proof, we denote polynomials by the same symbol $\mathcal{P}$, which may change from line to line. 
It follows from the first equation in \eqref{TPBVPj} that 
$\tau_j''=\mathcal{P}(\ds^{\leq1}\bm{u}_0,\ldots,\ds^{\leq1}\bm{u}_j,\bm{u}_{j+1},$ $\tau_0,\ldots,\tau_j)$, so that 
\[
\ds^{k+2}\tau_j = \mathcal{P}(\ds^{\leq k+1}\bm{u}_0,\ldots,\ds^{\leq k+1}\bm{u}_j,\ds^{\leq k}\bm{u}_{j+1},\ds^{\leq k}\tau_0,\ldots,\ds^{\leq k}\tau_j).
\]
Similarly, it follows from \eqref{EqUj} that 
\[
\ds^k\bm{u}_{j+2} = \tau_0\ds^{k+2}\bm{u}_j + \mathcal{P}(\ds^{\leq k+2}\bm{u}_0,\ldots,\ds^{\leq k+2}\bm{u}_{j-1},\ds^{\leq k+1}\bm{u}_{j},
 \ds^{\leq k+2}\tau_0,\ldots,\ds^{\leq k+2}\tau_j).
\]
Using these equations inductively on $j$ and the second boundary condition in \eqref{TPBVPj}, on $s=1$ we obtain 
\[
\begin{cases}
 \ds^{k+2}\tau_{2j} = \mathcal{P}(\ds^{\leq 2j+k+1}\bm{u}_0,\ds^{\leq 2j+k}\bm{u}_1, \tau_0,\ldots,\tau_{2j}), \\
 \ds^{k}\bm{u}_{2j+2} = \tau_0^{j+1}\ds^{2j+k+2}\bm{u}_0 + \mathcal{P}(\ds^{\leq 2j+k+1}\bm{u}_0,\ds^{\leq 2j+k}\bm{u}_1, \tau_0,\ldots,\tau_{2j}),
\end{cases}
\]
and 
\[
\begin{cases}
 \ds^{k+2}\tau_{2j+1} = \mathcal{P}(\ds^{\leq 2j+k+2}\bm{u}_0,\ds^{\leq 2j+k+1}\bm{u}_1, \tau_0,\ldots,\tau_{2j+1}), \\
 \ds^{k}\bm{u}_{2j+3} = \tau_0^{j+1}\ds^{2j+k+2}\bm{u}_1 + \mathcal{P}(\ds^{\leq 2j+k+2}\bm{u}_0,\ds^{\leq 2j+k+1}\bm{u}_1, \tau_0,\ldots,\tau_{2j+1}).
\end{cases}
\]
Since $\bm{x}_{j+2} = \tau_0\bm{u}_j'+\mathcal{P}(\ds^{\leq 1}\bm{u}_0,\ldots,\ds^{\leq 1}\bm{u}_{j-1},\bm{u}_j,\tau_0,\ldots,\tau_j)$, 
these equations imply the desired ones. 
\end{proof}

%------------------------------------------------------------------------------
\subsection{Analysis of the functional $\Theta_j(\bm{u}_0,\bm{u}_1)$}
We recall that $\tau_j(1)$ is determined uniquely from $\bm{u}_0$ and $\bm{u}_1$ through solving the two-point boundary value problem \eqref{TPBVPj}, 
so that $\tau_j(1)$ depends non-locally on $(\bm{u}_0,\bm{u}_1)$. 
In view of this, we define functionals $\Theta_j=\Theta_j(\bm{u}_0,\bm{u}_1)$ by 
\begin{equation}\label{defTheta}
\Theta_j(\bm{u}_0,\bm{u}_1)=\tau_j(1)
\end{equation}
for $j=0,1,2,\ldots$. 
In order to show Proposition \ref{prop:AppID1}, we use the property of these functionals.

\begin{lemma}\label{lem:Theta}
It holds that $\Theta_0 \in C^1(X^1\times X^0;\mathbb{R})$ and 
$\Theta_j \in C^1(Y^{j+1}\times Y^{j};\mathbb{R})$ for $j=1,2,3,\ldots$. 
\end{lemma}

\begin{remark}\label{re:Theta}
In fact, we have $\Theta_0 \in C^\infty(X^1\times X^0;\mathbb{R})$ and $\Theta_j \in C^\infty(Y^{j+1}\times Y^{j};\mathbb{R})$ for $j=1,2,3,\ldots$. 
However, for our purpose, it is sufficient to show that they are of class $C^1$. 
\end{remark}

In the rest of this subsection, we prove Lemma \ref{lem:Theta}. 
Since $\tau_0$ satisfies \eqref{TPBVPj} with $j=0$, it follows from Lemma \ref{lem:EstSolBVP3} together with $\|\bm{u}_0'\|_{Y^0} \leq \|\bm{u}_0\|_{X^1}$ 
and $|\bm{u}_0(1)| \lesssim \|\bm{u}_0\|_{X^1}$ that 
\begin{equation}\label{EstT0}
\|\tau_0'\|_{L^\infty}\leq C(\|\bm{u}_0\|_{X^1})( \|\bm{u}_0\|_{X^1} + \|\bm{u}_1\|_{L^2}^2). 
\end{equation}
This shows that the functional $\Theta_0$ is in fact defined on $X^1\times X^0$. 
To study the map $\Theta_j$ for $j\geq1$, we evaluate $\bm{u}_j$ and $\tau_j$.

\begin{lemma}\label{lem:EstUjTj1}
For any integers $j\geq0$ and $k\geq1$, we have 
\[
\begin{cases}
 \|\tau_j'\|_{L^\infty} + \|\tau_j''\|_{Y^0} \leq C( \|\bm{u}_0\|_{Y^{j+2}}, \|\bm{u}_1\|_{Y^{j+1}}), \\
 \|\tau_j''\|_{Y^k}+\|\bm{u}_{j+2}\|_{Y^k} \leq C( \|\bm{u}_0\|_{Y^{k+j+2}}, \|\bm{u}_1\|_{Y^{k+j+1}}).
\end{cases}
\]
\end{lemma}

\begin{proof}
Since $\tau_j$ satisfies \eqref{TPBVPj}, it follows from Lemma \ref{lem:EstSolBVP3} together with $|\tau_{j_0}(s)|\leq s\|\tau_{j_0}'\|_{L^\infty}$ that 
$\|\tau_j'\|_{L^\infty} \leq C(\|(\bm{u}_0,\ldots,\bm{u}_j)\|_{X^1},\|\bm{u}_{j+1}\|_{L^2},\|(\tau_0',\ldots,\tau_{j-1}')\|_{L^\infty})$. 
Using this inductively and Lemma \ref{lem:NormRelation}, we obtain 
\begin{equation}\label{EstTj}
\|\tau_j'\|_{L^\infty} \leq C(\|(\bm{u}_0,\ldots,\bm{u}_j)\|_{Y^2},\|\bm{u}_{j+1}\|_{Y^1}).
\end{equation}

By Lemmas \ref{lem:algebraY} and \ref{lem:CalIneq2a}, we see that 
\begin{align*}
\|\tau_j''\|_{Y^l}
&\lesssim \sum_{j_1+j_2=j}\|\bm{u}_{j_1+1}\cdot\bm{u}_{j_2+1}\|_{Y^l} + \sum_{j_0+j_1+j_2=j}\|\tau_{j_0}\bm{u}_{j_1}'\cdot\bm{u}_{j_2}'\|_{Y^l} \\
&\lesssim \sum_{j_1+j_2=j}\|\bm{u}_{j_1+1}\|_{Y^{l+1}} \|\bm{u}_{j_2+1}\|_{Y^{l+1}} \\
&\quad\;
 + \sum_{j_0+j_1+j_2=j}
 \begin{cases}
  \|\tau_{j_0}'\|_{L^\infty}\|\bm{u}_{j_1}\|_{Y^{l+2}}\|\bm{u}_{j_2}\|_{Y^{l+2}} &(l=0,1) \\
  (\|\tau_{j_0}'\|_{L^\infty}+\|\tau_{j_0}''\|_{Y^{l-2}})\|\bm{u}_{j_1}\|_{Y^{l+1}}\|\bm{u}_{j_2}\|_{Y^{l+1}} &(l\geq2)
 \end{cases} \\
&\lesssim \|(\bm{u}_1,\ldots,\bm{u}_{j+1})\|_{Y^{l+1}}^2 \\
&\quad\;
 + 
 \begin{cases}
  \|(\tau_0',\ldots,\tau_j')\|_{L^\infty}\|(\bm{u}_0,\ldots,\bm{u}_j)\|_{Y^{l+2}}^2 &(l=0,1) \\
  ( \|(\tau_0',\ldots,\tau_j')\|_{L^\infty} + \|(\tau_0'',\ldots,\tau_j'')\|_{Y^{l-2}} ) \|(\bm{u}_0,\ldots,\bm{u}_j)\|_{Y^{l+1}}^2 &(l\geq2).
 \end{cases}
\end{align*}
Using this inductively on $j$ and \eqref{EstTj}, we obtain 
\begin{equation}\label{EstTj''}
\|\tau_j''\|_{Y^l} \leq C( \|(\bm{u}_0,\ldots,\bm{u}_j)\|_{Y^{l+2}}, \|\bm{u}_{j+1}\|_{Y^{l+1}} )
\end{equation}
for $l=0,1,2,\ldots$.

On the other hand, it follows from \eqref{EqUj} together with Lemma \ref{lem:CalIneq3}, \eqref{EstTj}, and \eqref{EstTj''} that 
\begin{align*}
\|\bm{u}_{j+2}\|_{Y^k} 
&\lesssim \sum_{j_0+j_1=j}\|(\tau_{j_0}\bm{u}_{j_1})''\|_{Y^k} \\
&\lesssim (\|(\tau_0',\ldots,\tau_j')\|_{L^\infty} + \|(\tau_0'',\ldots,\tau_j'')\|_{Y^k} ) \|(\bm{u}_0,\ldots,\bm{u}_j)\|_{Y^{k+2}} \\
&\leq C( \|(\bm{u}_0,\ldots,\bm{u}_j)\|_{Y^{k+2}}, \|\bm{u}_{j+1}\|_{Y^{k+1}} ).
\end{align*}
Using this inductively on $j$, we obtain the desired estimate for $\bm{u}_{j+2}$. 
Then, plugging the estimate into \eqref{EstTj} and \eqref{EstTj''}, we get the other estimates of the lemma. 
\end{proof}

\begin{lemma}\label{lem:EstUjTj2}
For any integer $j\geq1$, we have 
\[
\|s\bm{u}_{j+1}\|_{L^2} + \|(\tau_j,s\tau_j')\|_{L^\infty} \leq C( \|\bm{u}_0\|_{Y^{j+1}}, \|\bm{u}_1\|_{Y^{j}}).
\]
\end{lemma}

\begin{proof}
By Lemma \ref{lem:CalIneq3} with $\epsilon=\frac12$, we see that 
\begin{align*}
\|s\bm{u}_{j+1}\|_{L^2}
&\lesssim \sum_{j_0+j_1=j-1} \|s^\frac12(\tau_{j_0}\bm{u}_{j_1})''\|_{Y^0} \\
&\lesssim \sum_{j_0+j_1=j-1} ( \|\tau_{j_0}'\|_{L^\infty} + \|\tau_{j_0}''\|_{Y^0} ) \|\bm{u}_{j_1}\|_{Y^2},
\end{align*}
which together with Lemma \ref{lem:EstUjTj1} yields the desired estimate for $\bm{u}_{j+1}$.

Since $\tau_j$ satisfies \eqref{TPBVPj}, it follows from Lemma \ref{lem:EstSolBVP3} together with the embedding $|u(1)|\lesssim\|u\|_{Y^1}$ that 
$\|(\tau_j,s\tau_j')\|_{L^\infty} \leq C(\|\bm{u}_0\|_{Y^2})( \|\bm{u}_j\|_{Y^1}+\|sh_j\|_{L^1})$. 
Here, we see that 
\begin{align*}
\|sh_j\|_{L^1}
&\lesssim \sum_{j_1+j_2=j} \|s\bm{u}_{j_1+1}\cdot\bm{u}_{j_2+1}\|_{L^1}
 + \sum_{j_0+j_1+j_2=j,j_0\leq j-1} \|s\tau_{j_0}\bm{u}_{j_1}'\cdot\bm{u}_{j_2}'\|_{L^1} \\
&\lesssim \|\bm{u}_1\|_{L^2} \|s\bm{u}_{j+1}\|_{L^2} + \|s^\frac12(\bm{u}_2,\ldots,\bm{u}_j)\|_{L^2}^2 
 + \|(\tau_0',\ldots,\tau_{j-1}')\|_{L^\infty} \|s(\bm{u}_0',\ldots,\bm{u}_j')\|_{L^2}^2 \\
&\lesssim \|\bm{u}_1\|_{Y^1} \|s\bm{u}_{j+1}\|_{L^2} + \|(\bm{u}_2,\ldots,\bm{u}_j)\|_{Y^0}^2 
 + \|(\tau_0',\ldots,\tau_{j-1}')\|_{L^\infty} \|(\bm{u}_0,\ldots,\bm{u}_j)\|_{Y^1}^2.
\end{align*}
These estimates together with Lemma \ref{lem:EstUjTj1} imply the desired estimate for $\tau_j$.
\end{proof}

Lemma \ref{lem:EstUjTj2} implies that the functional $\Theta_j$ is in fact defined on $Y^{j+1}\times Y^j$ for $j=1,2,3,\ldots$.

We proceed to show the differentiability of the map $\Theta_j$ for $j=0,1,2,\ldots$. 
To this end, we need to study the differentiability of the map $(\bm{u}_0,\bm{u}_1) \mapsto (\bm{u}_j, \tau_j)$. 
With a slight abuse of notation, we write $\bm{u}_j=\bm{u}_j(\bm{u}_0,\bm{u}_1)$ and $\tau_j=\tau_j(\bm{u}_0,\bm{u}_1)$. 
This would cause no confusion. 
Assuming that these maps are Fr\'echet differentiable in appropriate spaces, 
we write $\hat{\bm{u}}_j=D\bm{u}_j(\bm{u}_0,\bm{u}_1)[\hat{\bm{u}}_0,\hat{\bm{u}}_1]$ and 
$\hat{\tau}_j=D\tau_j(\bm{u}_0,\bm{u}_1)[\hat{\bm{u}}_0,\hat{\bm{u}}_1]$. 
Then, we see that $\hat{\bm{u}}_j$ and $\hat{\tau}_j$ satisfy 
\begin{equation}\label{DTPBVPj}
\begin{cases}
 -\hat{\tau}_j''+|\bm{u}_0'|^2\hat{\tau}_j = \hat{h}_j \quad\mbox{in}\quad (0,1), \\
 \hat{\tau}_j(0)=0, \quad \hat{\tau}_j'(1) = -\bm{g}\cdot\hat{\bm{u}}_j(1),
\end{cases}
\end{equation}
where 
\begin{align*}
\hat{h}_j 
&= 2\sum_{j_1+j_2=j}\frac{j!}{j_1!j_2!}\bm{u}_{j_1+1}\cdot\hat{\bm{u}}_{j_2+1} 
 - 2\sum_{j_0+j_1+j_2=j}\frac{j!}{j_0!j_1!j_2!}(\bm{u}_{j_1}'\cdot\hat{\bm{u}}_{j_2}')\tau_{j_0} \\
&\quad\;
 - \sum_{j_0+j_1+j_2=j,j_0\leq j-1}\frac{j!}{j_0!j_1!j_2!}(\bm{u}_{j_1}'\cdot\bm{u}_{j_2}')\hat{\tau}_{j_0},
\end{align*}
and 
\begin{equation}\label{DEqUj}
\hat{\bm{u}}_{j+2} = \sum_{j_0+j_1=j}\frac{j!}{j_0!j_1!}(\tau_{j_0}\hat{\bm{u}}_{j_1}+\hat{\tau}_{j_0}\bm{u}_{j_1})''
\end{equation}
for $j=0,1,2,\ldots$.

Since $\hat{\tau}_0$ is defined as a unique solution to the two-point boundary value problem \eqref{DTPBVPj} with $j=0$, 
by Lemma \ref{lem:EstSolBVP3} together with the estimates $|\hat{\bm{u}}_0(1)| \lesssim \|\hat{\bm{u}}_0\|_{X^1}$, 
$\|\hat{h}_0\|_{L^1} \leq 2\|\bm{u}_1\|_{X^0}\|\hat{\bm{u}}_1\|_{X^0} + 2\|\tau_0'\|_{L^\infty}\|\bm{u}_0\|_{X^1}\|\hat{\bm{u}}_0\|_{X^1}$, 
and \eqref{EstT0}, we obtain 
\begin{equation}\label{EstDT0}
\|\hat{\tau}_0'\|_{L^\infty} \leq C(\|\bm{u}_0\|_{X^1},\|\bm{u}_1\|_{X^0})( \|\hat{\bm{u}}_0\|_{X^1} + \|\hat{\bm{u}}_1\|_{X^0} ).
\end{equation}
Therefore, we see that the map $(\hat{\bm{u}}_0,\hat{\bm{u}}_1) \mapsto \hat{\tau}_0$ defines a linear bounded operator from $X^1\times X^0$ to $W^{1,\infty}$. 
Moreover, by considering a similar two-point boundary value problem to \eqref{DTPBVPj} for 
$\tau_0(\bm{u}_0+\hat{\bm{u}}_0,\bm{u}_1+\hat{\bm{u}}_1)-\tau_0(\bm{u}_0,\bm{u}_1)-\hat{\tau}_0$, we also obtain 
\begin{align*}
&\|(\tau_0(\bm{u}_0+\hat{\bm{u}}_0,\bm{u}_1+\hat{\bm{u}}_1)-\tau_0(\bm{u}_0,\bm{u}_1)-\hat{\tau}_0)'\|_{L^\infty} \\
&\leq  C(\|(\bm{u}_0,\hat{\bm{u}}_0)\|_{X^1},\|(\bm{u}_1,\hat{\bm{u}}_1)\|_{X^0})( \|\hat{\bm{u}}_0\|_{X^1} + \|\hat{\bm{u}}_1\|_{X^0} )^2.
\end{align*}
This shows that the map $\tau_0 \colon X^1\times X^0\to W^{1,\infty}$ is Fr\'echet differentiable and the derivative applied to 
$[\hat{\bm{u}}_0,\hat{\bm{u}}_1]$ is given by $\hat{\tau}_0$. 
Particularly, we can easily deduce that $\Theta_0 \colon X^1\times X^0\to\mathbb{R}$ is Fr\'echet differentiable and 
$D\Theta_0(\bm{u}_0,\bm{u}_1)[\hat{\bm{u}}_0,\hat{\bm{u}}_1]$ is in fact given by $\hat{\tau}_0(1)$. 
Similarly, a Lipschitz continuity of the linear bounded operator $D\Theta_0(\bm{u}_0,\bm{u}_1)$ with respect to $(\bm{u}_0,\bm{u}_1) \in X^1\times X^0$ 
can be proved. 
Therefore, we have $\Theta_0 \in C^1(X^1\times X^0;\mathbb{R})$.

To show the differentiability of the map $\Theta_j$ for $j\geq1$, we proceed to evaluate $\hat{\bm{u}}_j$ and $\hat{\tau}_j$ defined by 
\eqref{DTPBVPj} and \eqref{DEqUj}.

\begin{lemma}\label{lem:DEstUjTj1}
For any integers $j\geq0$ and $k\geq1$, we have 
\[
\begin{cases}
 \|\hat{\tau}_j'\|_{L^\infty} + \|\hat{\tau}_j''\|_{Y^0}
  \leq C( \|\bm{u}_0\|_{Y^{j+2}}, \|\bm{u}_1\|_{Y^{j+1}})( \|\hat{\bm{u}}_0\|_{Y^{j+2}} + \|\hat{\bm{u}}_1\|_{Y^{j+1}}), \\
 \|\hat{\tau}_j''\|_{Y^k}+\|\hat{\bm{u}}_{j+2}\|_{Y^k}
  \leq C( \|\bm{u}_0\|_{Y^{k+j+2}}, \|\bm{u}_1\|_{Y^{k+j+1}})( \|\hat{\bm{u}}_0\|_{Y^{k+j+2}} + \|\hat{\bm{u}}_1\|_{Y^{k+j+1}}).
\end{cases}
\]
\end{lemma}

\begin{proof}
The proof is almost the same as that of Lemma \ref{lem:EstUjTj1}. 
We recall that $\hat{\tau}_j$ is defined as a unique solution to the two-point boundary value problem \eqref{DTPBVPj}. 
By Lemma \ref{lem:EstSolBVP3}, we have $\|\hat{\tau}_j'\|_{L^\infty} \leq C(\|\bm{u}_0\|_{Y^2})(\|\hat{\bm{u}}_j\|_{Y^1}+\|\hat{h}_j\|_{L^1})$. 
Here, it follows from Lemmas \ref{lem:NormRelation} and \ref{lem:EstUjTj1} that 
\begin{align*}
\|\hat{h}_j\|_{L^1}
&\lesssim \|(\bm{u}_1,\ldots,\bm{u}_{j+1})\|_{L^2} \|(\hat{\bm{u}}_1,\ldots,\hat{\bm{u}}_{j+1})\|_{L^2}
 + \|(\hat{\tau}_0',\ldots,\hat{\tau}_{j-1}')\|_{L^\infty} \|(\bm{u}_0,\ldots,\bm{u}_j)\|_{X^1}^2 \\
&\quad\;
 + \|(\tau_0',\ldots,\tau_j')\|_{L^\infty} \|(\bm{u}_0,\ldots,\bm{u}_j)\|_{X^1} \|(\hat{\bm{u}}_0,\ldots,\hat{\bm{u}}_j)\|_{X^1} \\
&\leq C( \|\bm{u}_0\|_{Y^{j+2}},\|\bm{u}_1\|_{Y^{j+1}} )
 ( \|(\hat{\bm{u}}_0,\ldots,\hat{\bm{u}}_j)\|_{Y^2}+\|\hat{\bm{u}}_{j+1}\|_{Y^1} + \|(\hat{\tau}_0',\ldots,\hat{\tau}_{j-1}')\|_{L^\infty} ).
\end{align*}
Therefore, by induction on $j$ together with \eqref{EstDT0} we obtain 
\begin{equation}\label{EstDTj}
\|\hat{\tau}_j'\|_{L^\infty} \leq C( \|\bm{u}_0\|_{Y^{j+2}},\|\bm{u}_1\|_{Y^{j+1}} )
 ( \|(\hat{\bm{u}}_0,\ldots,\hat{\bm{u}}_j)\|_{Y^2}+\|\hat{\bm{u}}_{j+1}\|_{Y^1} ).
\end{equation}

By Lemmas \ref{lem:algebraY} and \ref{lem:CalIneq2a}, we see that 
\begin{align*}
\|\hat{\tau}_j''\|_{Y^l}
&\lesssim \sum_{j_1+j_2=j} \|\bm{u}_{j_1+1}\|_{Y^{l+1}} \|\hat{\bm{u}}_{j_2+1}\|_{Y^{l+1}} \\
&\quad\;
 + \sum_{j_0+j_1+j_2=j} 
 \begin{cases}
  \|\tau_{j_0}'\|_{L^\infty} \|\bm{u}_{j_1}\|_{Y^{l+2}} \|\hat{\bm{u}}_{j_2}\|_{Y^{l+2}} \\
  \quad
    + \|\hat{\tau}_{j_0}'\|_{L^\infty} \|\bm{u}_{j_1}\|_{Y^{l+2}} \|\bm{u}_{j_2}\|_{Y^{l+2}} &(l=0,1) \\
  (\|\tau_{j_0}'\|_{L^\infty}+\|\tau_{j_0}''\|_{Y^{l-2}}) \|\bm{u}_{j_1}\|_{Y^{l+1}} \|\hat{\bm{u}}_{j_2}\|_{Y^{l+1}} \\
  \quad
   + (\|\hat{\tau}_{j_0}'\|_{L^\infty}+\|\hat{\tau}_{j_0}''\|_{Y^{l-2}}) \|\bm{u}_{j_1}\|_{Y^{l+1}} \|\bm{u}_{j_2}\|_{Y^{l+1}} &(l\geq2).
 \end{cases}
\end{align*}
Using this inductively on $j$ together with Lemma \ref{lem:EstUjTj1} and \eqref{EstDTj}, we obtain 
\begin{equation}\label{EstDTj''}
\|\hat{\tau}_j''\|_{Y^l} \leq C(\|\bm{u}_0\|_{Y^{l+j+2}}, \|\bm{u}_1\|_{Y^{l+j+1}})
 ( \|(\hat{\bm{u}}_0,\ldots,\hat{\bm{u}}_j)\|_{Y^{l+2}}+\|\hat{\bm{u}}_{j+1}\|_{Y^{l+1}} )
\end{equation}
for $l=0,1,2,\ldots$.

On the other hand, it follows from \eqref{DEqUj} together with Lemma \ref{lem:CalIneq3}, \eqref{EstDTj}, and \eqref{EstDTj''} that 
\begin{align*}
\|\hat{\bm{u}}_{j+2}\|_{Y^k}
&\lesssim ( \|(\tau_0',\ldots,\tau_j')\|_{L^\infty} + \|(\tau_0'',\ldots,\tau_j'')\|_{Y^k} ) \|(\hat{\bm{u}}_0,\ldots,\hat{\bm{u}}_j)\|_{Y^{k+2}} \\
&\quad\;
 + ( \|(\hat{\tau}_0',\ldots,\hat{\tau}_j')\|_{L^\infty} + \|(\hat{\tau}_0'',\ldots,\hat{\tau}_j'')\|_{Y^k} ) \|(\bm{u}_0,\ldots,\bm{u}_j)\|_{Y^{k+2}} \\
&\leq C(\|\bm{u}_0\|_{Y^{k+j+2}},\|\bm{u}_1\|_{Y^{k+j+1}}) ( \|(\hat{\bm{u}}_0,\ldots,\hat{\bm{u}}_j)\|_{Y^{k+2}} + \|\hat{\bm{u}}_{j+1}\|_{Y^{k+1}} ).
\end{align*}
Using this inductively on $j$, we obtain the desired estimate for $\hat{\bm{u}}_{j+2}$. 
Then, plugging the estimate into \eqref{EstDTj} and \eqref{EstDTj''}, we get the other estimates of the lemma. 
\end{proof}

\begin{lemma}\label{lem:DEstUjTj2}
For any integer $j\geq1$, we have 
\[
\|s\hat{\bm{u}}_{j+1}\|_{L^2} + \|(\hat{\tau}_j,s\hat{\tau}_j')\|_{L^\infty}
 \leq C( \|\bm{u}_0\|_{Y^{j+1}}, \|\bm{u}_1\|_{Y^{j}})( \|\hat{\bm{u}}_0\|_{Y^{j+1}} + \|\hat{\bm{u}}_1\|_{Y^{j}}).
\]
\end{lemma}

\begin{proof}
The proof is almost the same as that of Lemma \ref{lem:EstUjTj2}. 
By Lemma \ref{lem:CalIneq3} with $\epsilon=\frac12$, we see that 
\begin{align*}
\|s\hat{\bm{u}}_{j+1}\|_{L^2}
&\lesssim \sum_{j_0+j_1=j-1} \{ ( \|\tau_{j_0}'\|_{L^\infty} + \|\tau_{j_0}''\|_{Y^0} ) \|\hat{\bm{u}}_{j_1}\|_{Y^2}
 + ( \|\hat{\tau}_{j_0}'\|_{L^\infty} + \|\hat{\tau}_{j_0}''\|_{Y^0} ) \|\bm{u}_{j_1}\|_{Y^2}\},
\end{align*}
which together with Lemmas \ref{lem:EstUjTj1} and \eqref{lem:DEstUjTj1} yields the desired estimate for $\hat{\bm{u}}_{j+1}$.

Since $\tau_j$ satisfies \eqref{DTPBVPj}, it follows from Lemma \ref{lem:EstSolBVP3} that 
$\|(\hat{\tau}_j,s\hat{\tau}_j')\|_{L^\infty} \leq C(\|\bm{u}_0\|_{Y^2})( \|\hat{\bm{u}}_j\|_{Y^1}+\|s\hat{h}_j\|_{L^1})$. 
Here, we see that 
\begin{align*}
\|s\hat{h}_j\|_{L^1}
&\lesssim \|\bm{u}_1\|_{Y^1} \|s\hat{\bm{u}}_{j+1}\|_{L^2} + \|s\bm{u}_{j+1}\|_{L^2}\|\hat{\bm{u}}_1\|_{Y^1}
 + \|(\bm{u}_2,\ldots,\bm{u}_j)\|_{Y^0}\|(\hat{\bm{u}}_2,\ldots,\hat{\bm{u}}_j)\|_{Y^0} \\
&\quad\;
 + \|\tau_j\|_{L^\infty}\|\bm{u}_0\|_{Y^2}\|\hat{\bm{u}}_0\|_{Y^2}
 + \|(\tau_0',\ldots,\tau_{j-1}')\|_{L^\infty} \|(\bm{u}_0,\ldots,\bm{u}_j)\|_{Y^1}\|(\hat{\bm{u}}_0,\ldots,\hat{\bm{u}}_j)\|_{Y^1} \\
&\quad\;
 + \|(\hat{\tau}_0',\ldots,\hat{\tau}_{j-1}')\|_{L^\infty} \|(\bm{u}_0,\ldots,\bm{u}_j)\|_{Y^1}^2.
\end{align*}
These estimates together with Lemmas \ref{lem:EstUjTj1}--\ref{lem:DEstUjTj1} imply the desired estimate for $\hat{\tau}_j$. 
\end{proof}

By this Lemma \ref{lem:DEstUjTj2}, we see that the map $(\hat{\bm{u}}_0,\hat{\bm{u}}_1) \mapsto (\hat{\tau}_j,s\hat{\tau}_j')$ defines a linear operator 
from $Y^{j+1}\times Y^j$ to $L^\infty$ for $j=1,2,3,\ldots$. 
Moreover, by a similar argument to the case $j=0$ we can also prove that the map $\Theta_j \colon Y^{j+1}\times Y^j \to \mathbb{R}$ is continuously differentiable 
and $D\Theta_j(\bm{u}_0,\bm{u}_1)[\hat{\bm{u}}_0,\hat{\bm{u}}_1]$ is in fact given by $\hat{\tau}_j(1)$ for $j=1,2,3,\ldots$. 
Since the proof is standard and straightforward, we omit it. 
The proof of Lemma \ref{lem:Theta} is complete. 
\hfill$\qed$

%------------------------------------------------------------------------------
\subsection{Proof of Lemma \ref{lem:AppID1}}
{\bf The case $m=1,2$.} 
This case is trivial. 
In fact, the case $m=1$ can be proved as follows. 
As can be shown similarly to the proof of \cite[Lemma 4.6]{Takayama2018}, $C^\infty([0,1])$ is dense in $X^k$ for any non-negative integer $k$. 
Therefore, for any $\varepsilon>0$ there exist $(\tilde{\bm{x}}_0^{\mathrm{in},\varepsilon},\tilde{\bm{x}}_1^{\mathrm{in},\varepsilon}) \in C^\infty([0,1])$ 
such that $\|\tilde{\bm{x}}_0^{\mathrm{in},\varepsilon}-\bm{x}_0^{\mathrm{in}}\|_{X^1}
 + \|\tilde{\bm{x}}_1^{\mathrm{in},\varepsilon}-\bm{x}_1^{\mathrm{in}}\|_{X^0} < \varepsilon$. 
Then, we define the approximate data by 
$\bm{x}_0^{\mathrm{in},\varepsilon}(s)=\tilde{\bm{x}}_0^{\mathrm{in},\varepsilon}(s)-\tilde{\bm{x}}_0^{\mathrm{in},\varepsilon}(1)$ and 
$\bm{x}_1^{\mathrm{in},\varepsilon}(s)=\tilde{\bm{x}}_1^{\mathrm{in},\varepsilon}(s)$. 
Obviously, the data $(\bm{x}_0^{\mathrm{in},\varepsilon},\bm{x}_1^{\mathrm{in},\varepsilon})$ satisfy the compatibility condition at order $0$, that is, 
$\bm{x}_0^{\mathrm{in},\varepsilon}(1)=\bm{0}$. 
Since $\bm{x}_0^\mathrm{in}(1)=\bm{0}$, we have $(\bm{x}_0^{\mathrm{in},\varepsilon}-\bm{x}_0^\mathrm{in})(s)
 = -\int_s^1(\tilde{\bm{x}}_0^{\mathrm{in},\varepsilon}-\bm{x}_0^\mathrm{in})'(\sigma)\mathrm{d}\sigma$. 
Moreover, by Lemma \ref{lem:NormEqui} we see that 
\begin{align*}
\|\bm{x}_0^{\mathrm{in},\varepsilon}-\bm{x}_0^{\mathrm{in}}\|_{X^1}
&\leq \|\bm{x}_0^{\mathrm{in},\varepsilon}-\bm{x}_0^{\mathrm{in}}\|_{L^2} + \|s^\frac12(\bm{x}_0^{\mathrm{in},\varepsilon}-\bm{x}_0^{\mathrm{in}})'\|_{L^2} \\
&\leq 2\|s(\tilde{\bm{x}}_0^{\mathrm{in},\varepsilon}-\bm{x}_0^{\mathrm{in}})'\|_{L^2}
 + \|s^\frac12(\tilde{\bm{x}}_0^{\mathrm{in},\varepsilon}-\bm{x}_0^{\mathrm{in}})'\|_{L^2} \\
&\leq 3\|\tilde{\bm{x}}_0^{\mathrm{in},\varepsilon}-\bm{x}_0^{\mathrm{in}}\|_{X^1}.
\end{align*}
Therefore, we obtain 
$\|\bm{x}_0^{\mathrm{in},\varepsilon}-\bm{x}_0^{\mathrm{in}}\|_{X^1} + \|\bm{x}_1^{\mathrm{in},\varepsilon}-\bm{x}_1^{\mathrm{in}}\|_{X^0} < 3\varepsilon$. 
This shows the lemma in the case $m=1$. 
Similarly, we can prove the case $m=2$.

\medskip
\noindent
{\bf The case $m=2k+1$ with $k\geq1$.}
We recall that $\bm{x}_j$ for $j=0,1,\ldots,m-1$ are determined uniquely from the initial data $(\bm{u}_0,\bm{u}_1)\in Y^{m-1}\times Y^{m-2}$; see Lemma \ref{lem:EstIV}. 
Taking this into account, we introduce a functional $\bm{X}_j$ as 
\begin{equation}\label{defXj}
\bm{X}_j(\bm{u}_0,\bm{u}_1)=\bm{x}_j(1)
\end{equation}
for $j=0,1,\ldots,m-1$. 
Then, we see that the initial data $(\bm{x}_0,\bm{x}_1)$ satisfy the compatibility condition at order $j$ if and only if 
$\bm{X}_j(\bm{x}_0',\bm{x}_1')=\bm{0}$ holds.

Let $\psi\in C_0^\infty(\mathbb{R})$ be a cut-off function satisfying $\psi(s)=1$ for $|s|\leq1$ and put 
\begin{equation}\label{defCut-off}
\psi_j^\delta(s) = \frac{(s-1)^j}{j!}\psi\left(\frac{s-1}{\delta}\right)
\end{equation}
for $\delta>0$ and $j=0,1,2,\ldots$. 
Obviously, we have 
\begin{equation}\label{EstCut-off}
 \begin{cases}
  (\ds^k\psi_j^\delta)(1) = \delta_{kj} &\mbox{for}\quad j,k=0,1,2,\ldots, \\
  \|\psi_j^\delta\|_{H^k} \leq C_{j,k}\delta^{j-k+\frac12} &\mbox{for}\quad 0<\delta\leq1, \ j,k=1,2,3,\ldots,
 \end{cases}
\end{equation}
where $\delta_{kj}$ is the Kronecker delta and $C_{j,k}$ is a positive constant depending only on $j$ and $k$. 
For $\bm{a}=(\bm{a}_1,\bm{a}_2,\ldots,\bm{a}_{m-2}) \in (\mathbb{R}^3)^{m-2}$, we are going to construct the approximate initial data 
$(\bm{U}_0^\delta,\bm{U}_1^\delta)$ in the form 
\begin{equation}\label{AppID1}
 \begin{cases}
  \bm{U}_0^\delta(s) = \bm{u}_0(s) + \sum_{l=1}^k \delta^{-2(l-1)}\psi_{2l-1}^\delta(s)\bm{a}_{2l-1}, \\
  \bm{U}_1^\delta(s) = \bm{u}_1(s) + \sum_{l=1}^{k-1} \delta^{-(2l-1)}\psi_{2l-1}^\delta(s)\bm{a}_{2l}. 
 \end{cases}
\end{equation}
With a slight abuse of notation, we write $\bm{U}_0^\delta=\bm{U}_0^\delta(\bm{u}_0,\bm{a})$ and $\bm{U}_1^\delta=\bm{U}_1^\delta(\bm{u}_1,\bm{a})$. 
This would cause no confusion. 
Here, we note that $(\bm{u}_0,\bm{u}_1)$ will play as smooth approximation of the initial data and $\bm{a}$ will be chosen so that 
$(\bm{U}_0^\delta,\bm{U}_1^\delta)$ satisfy the compatibility conditions up to order $m-1$. 
Then, we define a map $\bm{\Xi}^\delta = (\bm{\Xi}_2^\delta, \bm{\Xi}_3^\delta, \ldots,\bm{\Xi}_{m-1}^\delta)$ by 
\begin{equation}\label{defXi}
\bm{\Xi}_{j+2}^\delta(\bm{u}_0,\bm{u}_1,\bm{a}) = \delta^j\bm{X}_{j+2}(\bm{U}_0^\delta(\bm{u}_0,\bm{a}),\bm{U}_1^\delta(\bm{u}_1,\bm{a}))
\end{equation}
for $j=0,1,\ldots,m-3$. 
We proceed to solve the equation $\bm{\Xi}^\delta(\bm{u}_0,\bm{u}_1,\bm{a})=\bm{0}$ in terms of $\bm{a}$. 
To this end, we need to investigate the map $\bm{\Xi}^\delta$ in detail.

It follows from \eqref{EstCut-off} and \eqref{AppID1} that 
\[
 \begin{cases}
  (\ds^{2j}\bm{U}_0^\delta)(1) = (\ds^{2j}\bm{u}_0)(1), \\
  (\ds^{2j+1}\bm{U}_0^\delta)(1) = (\ds^{2j+1}\bm{u}_0)(1) + \delta^{-2j}\bm{a}_{2j+1}
 \end{cases}
\]
for $j=0,1,\ldots,k-1$ and that 
\[
 \begin{cases}
  (\ds^{2j}\bm{U}_1^\delta)(1) = (\ds^{2j}\bm{u}_1)(1), \\
  (\ds^{2j+1}\bm{U}_1^\delta)(1) = (\ds^{2j+1}\bm{u}_1)(1) + \delta^{-(2j+1)}\bm{a}_{2j+2}
 \end{cases}
\]
for $j=0,1,\ldots,k-2$. 
Therefore, by Lemma \ref{lem:ExpX} together with \eqref{defTheta} and \eqref{defXj} we obtain 
\begin{align}\label{ExpXi1}
& \bm{\Xi}_{2j+2}^\delta(\bm{u}_0,\bm{u}_1,\bm{a}) \\
&= \Theta_0(\bm{U}_0^\delta,\bm{U}_1^\delta)^{j+1}\bm{a}_{2j+1} + \delta^{2j} \{ \Theta_0(\bm{U}_0^\delta,\bm{U}_1^\delta)^{j+1}(\ds^{2j+1}\bm{u}_0)(1) 
 \nonumber \\
&\qquad
 + \mathcal{P}_{2j+2}( (\ds^{\leq 2j}\bm{U}_0^\delta)(1),(\ds^{\leq 2j-1}\bm{U}_1^\delta)(1),
  \Theta_0(\bm{U}_0^\delta,\bm{U}_1^\delta),\ldots,\Theta_{2j}(\bm{U}_0^\delta,\bm{U}_1^\delta)) \} \nonumber
\end{align}
for $j=0,1,\ldots,k-1$ and 
\begin{align}\label{ExpXi2}
& \bm{\Xi}_{2j+3}^\delta(\bm{u}_0,\bm{u}_1,\bm{a}) \\
&= \Theta_0(\bm{U}_0^\delta,\bm{U}_1^\delta)^{j+1}\bm{a}_{2j+2} + \delta^{2j+1} \{ \Theta_0(\bm{U}_0^\delta,\bm{U}_1^\delta)^{j+1}(\ds^{2j+1}\bm{u}_1)(1) 
 \nonumber \\
&\qquad
 + \mathcal{P}_{2j+3}( (\ds^{\leq 2j+1}\bm{U}_0^\delta)(1),(\ds^{\leq 2j}\bm{U}_1^\delta)(1),
  \Theta_0(\bm{U}_0^\delta,\bm{U}_1^\delta),\ldots,\Theta_{2j+1}(\bm{U}_0^\delta,\bm{U}_1^\delta)) \} \nonumber
\end{align}
for $j=0,1,\ldots,k-2$. 
These expressions together with Lemma \ref{lem:Theta} imply that 
\begin{equation}\label{C^1}
\bm{\Xi}^\delta \in C^1(Y^{m-1}\times Y^{m-2}\times(\mathbb{R}^3)^{m-2};(\mathbb{R}^3)^{m-2}). 
\end{equation}

We proceed to evaluate the derivative of the map $\bm{\Xi}^\delta(\bm{u}_0,\bm{u}_1,\bm{a})$ with respect to $\bm{a}$ at $\bm{a}=\bm{0}$, that is, 
$D_{\bm{a}}\bm{\Xi}^\delta(\bm{u}_0,\bm{u}_1,\bm{0})[\hat{\bm{a}}]$. 
By \eqref{EstCut-off} and \eqref{AppID1}, we see easily that 
\[
\|D_{\bm{a}}\bm{U}_0^\delta(\bm{u}_0,\bm{u}_1,\bm{a})[\hat{\bm{a}}]\|_{H^{j+1}} + \|D_{\bm{a}}\bm{U}_1^\delta(\bm{u}_0,\bm{u}_1,\bm{a})[\hat{\bm{a}}]\|_{H^j}
\lesssim \delta^{\frac12-j}|\hat{\bm{a}}|
\]
for $0<\delta\leq1$ and for $j=0,1,2,\ldots$, so that by Lemma \ref{lem:Theta} we obtain 
\[
|D_{\bm{a}}\Theta_j(\bm{U}_0^\delta,\bm{U}_1^\delta)|_{\bm{a}=\bm{0}}[\hat{\bm{a}}]|
\leq 
\begin{cases}
 C(\|\bm{u}_0\|_{X^1},\|\bm{u}_1\|_{X^0}) \delta^\frac12|\hat{\bm{a}}| &\mbox{for}\quad j=0, \\
 C(\|\bm{u}_0\|_{Y^{j+1}},\|\bm{u}_1\|_{Y^j}) \delta^{\frac12-j}|\hat{\bm{a}}| &\mbox{for}\quad j=1,2,3,\ldots.
\end{cases}
\]
Therefore, we get 
\[
\begin{cases}
 D_{\bm{a}}\bm{\Xi}_{2j+2}^\delta(\bm{u}_0,\bm{u}_1,\bm{0})[\hat{\bm{a}}] = \tau_0(1)^{j+1}\hat{\bm{a}}_{2j+1}
  + \delta^\frac12\bm{R}_{2j+2}^\delta(\bm{u}_0,\bm{u}_1)[\hat{\bm{a}}] &\mbox{for}\quad 0\leq j\leq k-1, \\
 D_{\bm{a}}\bm{\Xi}_{2j+3}^\delta(\bm{u}_0,\bm{u}_1,\bm{0})[\hat{\bm{a}}] = \tau_0(1)^{j+1}\hat{\bm{a}}_{2j+2}
  + \delta^\frac12\bm{R}_{2j+3}^\delta(\bm{u}_0,\bm{u}_1)[\hat{\bm{a}}] &\mbox{for}\quad 0\leq j\leq k-2,
\end{cases}
\]
where $\bm{R}_{j+2}^\delta$ satisfies 
\[
|\bm{R}_{j+2}^\delta(\bm{u}_0,\bm{u}_1)[\hat{\bm{a}}]| \leq C(\|\bm{u}_0\|_{Y^{j+2}},\|\bm{u}_1\|_{Y^{j+1}})|\hat{\bm{a}}|
\]
for $0<\delta\leq1$ and for $j=0,1,\ldots, m-3$. 
This shows that for any positive constants $M$ and $c$ there exists a positive constant $\delta_0$ such that 
if $\tau_0(1)\geq c$ and $\|\bm{u}_0\|_{Y^{m-1}}+\|\bm{u}_1\|_{Y^{m-2}} \leq M$, then the map 
$D_{\bm{a}}\bm{\Xi}^\delta(\bm{u}_0,\bm{u}_1,\bm{0}) \colon (\mathbb{R}^3)^{m-2} \to (\mathbb{R}^3)^{m-2}$ is invertible for any $\delta\in(0,\delta_0]$.

Let $(\bm{x}_0^\mathrm{in},\bm{x}_1^\mathrm{in})$ be the initial data stated in the lemma and put 
$(\bm{u}_0^\mathrm{in},\bm{u}_1^\mathrm{in})=(\bm{x}_0^{\mathrm{in}\prime},\bm{x}_1^{\mathrm{in}\prime})$. 
Then, we have $(\bm{u}_0^\mathrm{in},\bm{u}_1^\mathrm{in}) \in Y^{m-1}\times Y^{m-2}$. 
Since the initial data satisfy the compatibility conditions up to order $m-1$, 
we have $\bm{X}_j(\bm{u}_0^\mathrm{in},\bm{u}_1^\mathrm{in})=\bm{0}$ for $j=0,1,\ldots,m-1$. 
Therefore, in view of \eqref{defXi}, we see that 
\begin{equation}\label{CondIFT}
\begin{cases}
 \bm{\Xi}^\delta(\bm{u}_0^\mathrm{in},\bm{u}_1^\mathrm{in},\bm{0})=\bm{0}, \\
 D_{\bm{a}}\bm{\Xi}^\delta(\bm{u}_0^\mathrm{in},\bm{u}_1^\mathrm{in},\bm{0}) \colon (\mathbb{R}^3)^{m-2} \to (\mathbb{R}^3)^{m-2}
  \mbox{ is invertible}
\end{cases}
\end{equation}
for any sufficiently small $\delta>0$. 
Now, we fix the parameter $\delta$ so that \eqref{CondIFT} holds. 
By \eqref{C^1} and \eqref{CondIFT} we can apply the implicit function theorem to see that the equation $\bm{\Xi}^\delta(\bm{u}_0,\bm{u}_1,\bm{a})=\bm{0}$ 
for $\bm{a}$ can be uniquely solved in a neighborhood of $(\bm{u}_0,\bm{u}_1,\bm{a})=(\bm{u}_0^\mathrm{in},\bm{u}_1^\mathrm{in},\bm{0})$ in 
$Y^{j-1}\times Y^{j-2}\times(\mathbb{R}^3)^{m-2}$ and that the solution satisfies 
$|\bm{a}| \lesssim \|\bm{u}_0-\bm{u}_0^\mathrm{in}\|_{Y^{m-1}}+\|\bm{u}_1-\bm{u}_1^\mathrm{in}\|_{Y^{m-2}}$.

We are ready to construct the approximate initial data $(\bm{x}_0^{\mathrm{in},\varepsilon},\bm{x}_1^{\mathrm{in},\varepsilon})$. 
By the density of the space $C^\infty([0,1])$, for any $\varepsilon>0$ there exist 
$(\tilde{\bm{x}}_0^{\mathrm{in},\varepsilon},\tilde{\bm{x}}_1^{\mathrm{in},\varepsilon})\in C^\infty([0,1])$ such that 
$\|\tilde{\bm{x}}_0^{\mathrm{in},\varepsilon}-\bm{x}_0^{\mathrm{in}}\|_{X^m}
 + \|\tilde{\bm{x}}_1^{\mathrm{in},\varepsilon}-\bm{x}_1^{\mathrm{in}}\|_{X^{m-1}} < \varepsilon$. 
Put $\tilde{\bm{u}}_0^{\mathrm{in},\varepsilon}=\tilde{\bm{x}}_0^{\mathrm{in},\varepsilon \prime}$ and 
$\tilde{\bm{u}}_1^{\mathrm{in},\varepsilon}=\tilde{\bm{x}}_1^{\mathrm{in},\varepsilon \prime}$. 
Then, we have $\|\tilde{\bm{u}}_0^{\mathrm{in},\varepsilon}-\bm{u}_0^{\mathrm{in}}\|_{Y^{m-1}}
 + \|\tilde{\bm{u}}_1^{\mathrm{in},\varepsilon}-\bm{u}_1^{\mathrm{in}}\|_{Y^{m-2}} < \varepsilon$. 
Therefore, for any small $\varepsilon>0$ there exists a unique $\bm{a}^\varepsilon\in(\mathbb{R}^3)^{m-2}$ such that 
$\bm{\Xi}^\delta(\tilde{\bm{u}}_0^{\mathrm{in},\varepsilon},\tilde{\bm{u}}_1^{\mathrm{in},\varepsilon},\bm{a}^\varepsilon)=\bm{0}$ and 
$|\bm{a}^\varepsilon| \lesssim \varepsilon$. 
Now, we define the approximate initial data by 
$\bm{u}_j^{\mathrm{in},\varepsilon}=\bm{U}_j^\delta(\tilde{\bm{u}}_j^{\mathrm{in},\varepsilon},\bm{a}^\varepsilon)$ and 
$\bm{x}_j^{\mathrm{in},\varepsilon}(s)=-\int_s^1\bm{u}_j^{\mathrm{in},\varepsilon}(\sigma)\mathrm{d}\sigma$ for $j=0,1$. 
Then, by \eqref{AppID1} we have 
\begin{align*}
\|\bm{u}_0^{\mathrm{in},\varepsilon}-\bm{u}_0^\mathrm{in}\|_{Y^{m-1}} + \|\bm{u}_1^{\mathrm{in},\varepsilon}-\bm{u}_1^\mathrm{in}\|_{Y^{m-2}}
&\leq \|\tilde{\bm{u}}_0^{\mathrm{in},\varepsilon}-\bm{u}_0^\mathrm{in}\|_{Y^{m-1}}
 + \|\tilde{\bm{u}}_1^{\mathrm{in},\varepsilon}-\bm{u}_1^\mathrm{in}\|_{Y^{m-2}} + C_\delta |\bm{a}^\varepsilon| \\
&\lesssim \varepsilon.
\end{align*}
Therefore, we see that $(\bm{x}_0^{\mathrm{in},\varepsilon},\bm{x}_1^{\mathrm{in},\varepsilon})$ satisfy the desired properties.

\medskip
\noindent
{\bf The case $m=2k+2$ with $k\geq1$.} 
This case can be proved in the same way as in the previous case by replacing the definition \eqref{AppID1} of $\bm{U}_0^\delta$ and $\bm{U}_1^\delta$ with 
\begin{equation}\label{AppID2}
 \begin{cases}
  \bm{U}_0^\delta(s) = \bm{u}_0(s) + \sum_{l=1}^k \delta^{-2(l-1)}\psi_{2l-1}^\delta(s)\bm{a}_{2l-1}, \\
  \bm{U}_1^\delta(s) = \bm{u}_1(s) + \sum_{l=1}^k \delta^{-(2l-1)}\psi_{2l-1}^\delta(s)\bm{a}_{2l}. 
 \end{cases}
\end{equation}
The proof of Lemma \ref{lem:AppID1} is complete. 
\hfill$\qed$

%------------------------------------------------------------------------------
\subsection{Proof of Lemma \ref{lem:AppID2}}
The proof is carried out in a similar way to the proof of Lemma \ref{lem:AppID1}. 
We recall that $\psi_j^\delta$ is defined by \eqref{defCut-off} for $\delta>0$ and $j=0,1,2,\ldots$. 
We extend these cut-off functions as $\psi_j^\delta(s)\equiv0$ for $\delta\leq0$. 
Then, we see that for each $j=0,1,2,\ldots$, the map $\mathbb{R}\ni\delta\mapsto\psi_j^\delta\in H^j$ is continuous. 
We note however that the map $\mathbb{R}\ni\delta\mapsto\psi_j^\delta\in H^{j+1}$ is not continuous at $\delta=0$.

\medskip
\noindent{\bf The case $m=1$.} 
This case is trivial. 
In fact, for $\varepsilon>0$ we define the approximate data by $\bm{x}_0^{\mathrm{in},\varepsilon}(s)=\bm{x}_0^{\mathrm{in}}(s)$ and 
$\bm{x}_1^{\mathrm{in},\varepsilon}(s)=\bm{x}_1^{\mathrm{in}}(s)-\psi_0^\varepsilon(s)\bm{x}_1^{\mathrm{in}}(1)$. 
Obviously, the data $(\bm{x}_0^{\mathrm{in},\varepsilon},\bm{x}_1^{\mathrm{in},\varepsilon})$ satisfy the compatibility conditions up to order 1, that is, 
$\bm{x}_0^{\mathrm{in},\varepsilon}(1)=\bm{x}_1^{\mathrm{in},\varepsilon}(1)=\bm{0}$. 
Moreover, by \eqref{EstCut-off} we have 
$\|\bm{x}_1^{\mathrm{in},\varepsilon}-\bm{x}_1^{\mathrm{in}}\|_{X^0} \leq \|\psi_0^\varepsilon\|_{L^2}|\bm{x}_1^{\mathrm{in}}(1)| \lesssim \varepsilon^\frac12$. 
This shows the lemma in the case $m=1$.

\medskip
\noindent{\bf The case $m=2k$ with $k\geq1$.} 
We still use the functional $\bm{X}_j$ defined by \eqref{defXj}. 
For $\varepsilon\in\mathbb{R}$ and $\bm{a}=(\bm{a}_1,\bm{a}_2,\ldots,\bm{a}_{m-1}) \in (\mathbb{R}^3)^{m-1}$, 
we are going to construct the approximate initial data $(\bm{U}_0^\delta,\bm{U}_1^\delta)$ in the form 
\begin{equation}\label{AppID3}
 \begin{cases}
  \bm{U}_0^\delta(s) = \bm{u}_0^\mathrm{in}(s) + \sum_{l=1}^{k-1} \delta^{-2(l-1)}\psi_{2l-1}^\delta(s)\bm{a}_{2l-1}
   + \delta^{-2(k-1)}\psi_{2k-1}^\varepsilon(s)\bm{a}_{2k-1}, \\
  \bm{U}_1^\delta(s) = \bm{u}_1^\mathrm{in}(s) + \sum_{l=1}^{k-1} \delta^{-(2l-1)}\psi_{2l-1}^\delta(s)\bm{a}_{2l},
 \end{cases}
\end{equation}
instead of \eqref{AppID1}, where $\delta$ is a positive parameter and $\bm{u}_j^\mathrm{in}=\bm{x}_j^{\mathrm{in}\prime}$ for $j=0,1$ as before. 
Again, with a slight abuse of notation, we write $\bm{U}_j^\delta=\bm{U}_j^\delta(\varepsilon,\bm{a})$ for $j=0,1$. 
We will choose $\bm{a}$ so that $(\bm{U}_0^\delta,\bm{U}_1^\delta)$ satisfy the compatibility conditions up to order $m$. 
We define a map $\bm{\mathfrak{X}}^\delta=(\bm{\mathfrak{X}}_2^\delta,\bm{\mathfrak{X}}_3^\delta,\ldots,\bm{\mathfrak{X}}_m^\delta)$ by 
\begin{equation}\label{deffX1}
\bm{\mathfrak{X}}_{j+2}^\delta(\varepsilon,\bm{a}) = \delta^j\bm{X}_{j+2}(\bm{U}_0^\delta(\varepsilon,\bm{a}),\bm{U}_1^\delta(\varepsilon,\bm{a}))
\end{equation}
for $j=0,1,\ldots,m-3$ and 
\begin{align}\label{deffX2}
\bm{\mathfrak{X}}_{m}^\delta(\varepsilon,\bm{a}) 
& = \Theta_0(\bm{U}_0^\delta,\bm{U}_1^\delta)^{k}\bm{a}_{m-1} + \delta^{m-2} \{ \Theta_0(\bm{U}_0^\delta,\bm{U}_1^\delta)^{k}(\ds^{m-1}\bm{u}_0)(1) 
\\
&\qquad
 + \mathcal{P}_{m}( (\ds^{\leq m-2}\bm{U}_0^\delta)(1),(\ds^{\leq m-3}\bm{U}_1^\delta)(1),
  \Theta_0(\bm{U}_0^\delta,\bm{U}_1^\delta),\ldots,\Theta_{m-2}(\bm{U}_0^\delta,\bm{U}_1^\delta)) \}. \nonumber
\end{align}
\begin{remark}
By Lemma \ref{lem:ExpX} together with \eqref{defTheta} and \eqref{defXj} we have 
\begin{equation}\label{deffX1bis}
\bm{\mathfrak{X}}_m^\delta(\varepsilon,\bm{a}) = \delta^{m-2}\bm{X}_m(\bm{U}_0^\delta(\varepsilon,\bm{a}),\bm{U}_1^\delta(\varepsilon,\bm{a}))
\end{equation}
for $\varepsilon>0$; see also \eqref{ExpXi1}. 
However, this identity does not hold for $\varepsilon\leq0$. 
In fact, the left-hand side is continous in $(\varepsilon,\bm{a})$ whereas the right-hand side has a discontinuity at $\varepsilon=0$ 
due to the discontinuity of the map $\varepsilon\mapsto(\ds^{m-1}\psi_{m-1}^\varepsilon)(1)$ at $\varepsilon=0$. 
We will use this discontinuity to construct the approximate initial data stated in Lemma \ref{lem:AppID2}. 
\end{remark}
We see also that similar expressions to \eqref{ExpXi1} and \eqref{ExpXi2} holds 
by replacing $\bm{\Xi}_{j+2}^\delta(\bm{u}_0,\bm{u}_1,\bm{a})$ with 
$\bm{\mathfrak{X}}_{j+2}^\delta(\varepsilon,\bm{a})$ for $j=0,1,\ldots,m-3$, 
so that by Lemma \ref{lem:Theta} 
the map $\bm{\mathfrak{X}}^\delta \colon \mathbb{R}\times(\mathbb{R}^3)^{m-1} \to (\mathbb{R}^3)^{m-1}$ is continuous 
and partially differentiable 
with respect to $\bm{a}$. 
Moreover, the partial derivative $D_{\bm{a}}\bm{\mathfrak{X}}^\delta(\varepsilon,\bm{a})$ is also continuous and 
\[
\begin{cases}
 D_{\bm{a}}\bm{\mathfrak{X}}_{2j+2}^\delta(0,\bm{0},\ldots,\bm{0},\bm{a}_{m-1})[\hat{\bm{a}}] = \tau_0^\mathrm{in}(1)^{j+1}\hat{\bm{a}}_{2j+1}
  + \delta^\frac12\bm{R}_{2j+2}^\delta[\hat{\bm{a}}] &\mbox{for}\quad 0\leq j\leq k-1, \\
 D_{\bm{a}}\bm{\mathfrak{X}}_{2j+3}^\delta(0,\bm{0},\ldots,\bm{0},\bm{a}_{m-1})[\hat{\bm{a}}] = \tau_0^\mathrm{in}(1)^{j+1}\hat{\bm{a}}_{2j+2}
  + \delta^\frac12\bm{R}_{2j+3}^\delta[\hat{\bm{a}}] &\mbox{for}\quad 0\leq j\leq k-2,
\end{cases}
\]
where $\bm{R}_{j+2}^\delta$ satisfies $|\bm{R}_{j+2}^\delta[\hat{\bm{a}}]| \lesssim |\hat{\bm{a}}|$ for $0<\delta\leq1$ and $j=0,1,\ldots,m-2$. 
Therefore, there exists a positive constant $\delta_0$ such that for any $\delta\in(0,\delta_0]$ the map 
$D_{\bm{a}}\bm{\mathfrak{X}}(0,\bm{0},\ldots,\bm{0},\bm{a}_{m-1}) \colon (\mathbb{R}^3)^{m-1} \to (\mathbb{R}^3)^{m-1}$ is invertible. 
Now, we fix $\delta=\delta_0$.

Since the initial data $(\bm{x}_0^\mathrm{in},\bm{x}_1^\mathrm{in})$ satisfy the compatibility conditions up to $m-1$, we have 
$\bm{\mathfrak{X}}_{j+2}^\delta(0,\bm{0},\ldots,\bm{0},\bm{a}_{m-1}) = \delta^j\bm{X}_{j+2}(\bm{u}_0^\mathrm{in},\bm{u}_1^\mathrm{in})
 = \delta^j\bm{x}_{j+2}^{\mathrm{in}}(1)=\bm{0}$ for $j=0,1,\ldots,m-3$. 
We have also 
\[
\bm{\mathfrak{X}}_m^\delta(0,\bm{0},\ldots,\bm{0},\bm{a}_{m-1}) = \tau_0^\mathrm{in}(1)^k \bm{a}_{m-1} + \delta^{m-2}\bm{x}_m^\mathrm{in}(1).
\]
Taking this into account, we put $\bm{a}_{m-1}^*=-\tau_0^\mathrm{in}(1)^{-k}\delta^{m-2}\bm{x}_m^\mathrm{in}(1)$ and 
$\bm{a}^*=(\bm{0},\ldots,\bm{0},\bm{a}_{m-1}^*) \in (\mathbb{R}^3)^{m-1}$. 
Then, we have 
\[
\begin{cases}
 \bm{\mathfrak{X}}^\delta(0,\bm{a}^*)=\bm{0}, \\
 D_{\bm{a}}\bm{\mathfrak{X}}(0,\bm{a}^*) \colon (\mathbb{R}^3)^{m-1} \to (\mathbb{R}^3)^{m-1} \mbox{ is invertible}.
\end{cases}
\]
Therefore, we can apply the implicit function theorem to see that the equation $\bm{\mathfrak{X}}^\delta(\epsilon,\bm{a})=\bm{0}$ for $\bm{a}$ 
can be uniquely solved in a neighborhood of $(\varepsilon,\bm{a})=(0,\bm{a}^*)$. 
We denote the solution by $\bm{a}(\varepsilon)$, which is continuous and satisfies $\bm{a}(0)=\bm{a}^*$.

We define the approximate initial data by $\bm{u}_j^{\mathrm{in},\varepsilon}=\bm{U}_j^\delta(\varepsilon,\bm{a}(\varepsilon))$ and 
$\bm{x}_j^{\mathrm{in},\varepsilon}(s)=-\int_s^1\bm{u}_j^{\mathrm{in},\varepsilon}(\sigma)\mathrm{d}\sigma$ for $0<\varepsilon\ll1$ and $j=0,1$. 
Then, by \eqref{AppID3} and \eqref{EstCut-off} we have 
\begin{align*}
\|\bm{u}_0^{\mathrm{in},\varepsilon}-\bm{u}_0^\mathrm{in}\|_{Y^{m-1}} + \|\bm{u}_1^{\mathrm{in},\varepsilon}-\bm{u}_1^\mathrm{in}\|_{Y^{m-2}}
&\leq C_\delta( |\bm{a}(\varepsilon)-\bm{a}^*| + \varepsilon^\frac12|\bm{a}(\varepsilon)|),
\end{align*}
which converges to 0 as $\varepsilon\to+0$. 
Therefore, we see that $(\bm{x}_0^{\mathrm{in},\varepsilon},\bm{x}_1^{\mathrm{in},\varepsilon})$ satisfy the desired properties.

\medskip
\noindent
{\bf The case $m=2k+1$ with $k\geq1$.} 
This case can be proved in the same way as in the previous case by replacing the definition \eqref{AppID3} of $\bm{U}_0^\delta$ and $\bm{U}_1^\delta$ with 
\begin{equation}\label{AppID4}
 \begin{cases}
  \bm{U}_0^\delta(s) = \bm{u}_0^\mathrm{in}(s) + \sum_{l=1}^k \delta^{-2(l-1)}\psi_{2l-1}^\delta(s)\bm{a}_{2l-1}, \\
  \bm{U}_1^\delta(s) = \bm{u}_1^\mathrm{in}(s) + \sum_{l=1}^{k-1} \delta^{-(2l-1)}\psi_{2l-1}^\delta(s)\bm{a}_{2l}
   + \delta^{-(2k-1)}\psi_{2k-1}^\varepsilon(s)\bm{a}_{2k}.
 \end{cases}
\end{equation}
The proof of Lemma \ref{lem:AppID2} is complete. 
\hfill$\qed$

%----------------------------------------------------------------------------------------------------------------------
%----------------------------------------------------------------------------------------------------------------------
\section{Approximation of the initial data II}\label{sect:AppID2}
The objective in this section is to give a proof of Proposition \ref{prop:AppID2}, 
which ensure that the initial data $(\bm{x}_0^\mathrm{in},\bm{x}_1^\mathrm{in})$ for the problem \eqref{Eq}--\eqref{IC} can be approximated 
by smooth initial data satisfying higher order compatibility conditions and the constraints 
$|\bm{x}_0^{\mathrm{in}\prime}(s)|=1$ and $\bm{x}_0^{\mathrm{in}\prime}(s)\cdot\bm{x}_1^{\mathrm{in}\prime}(s)=0$.

We recall that the approximate initial data for the problem \eqref{HP2}--\eqref{IC2} were constructed in the previous section in the form 
\eqref{AppID1}, \eqref{AppID2}, \eqref{AppID3}, and \eqref{AppID4}. 
However, these approximated data do not satisfy the constraints $|\bm{u}_0^\mathrm{in}(s)|=1$ and $\bm{u}_0^\mathrm{in}(s)\cdot\bm{u}_1^\mathrm{in}(s)=0$, 
even if the original data satisfy them. 
Therefore, we need to modify these approximations. 
To this end, we introduce a map $\bm{F} : \mathbb{R}^3\setminus\{\bm{0}\} \to \mathbb{S}^2$ by 
\[
\bm{F}(\bm{v}) = \frac{\bm{v}}{|\bm{v}|}.
\]
Then, we have $D\bm{F}(\bm{v})[\hat{\bm{v}}] = \frac{1}{|\bm{v}|}P(\bm{v})\hat{\bm{v}}$, where $P(\bm{v})$ is the projection onto the orthogonal complement 
of the one-dimensional space spanned by $\bm{v}$ and defined by 
\[
P(\bm{v})=\mathrm{Id} - \bm{F}(\bm{v})\otimes\bm{F}(\bm{v}).
\]
Therefore, $D\bm{F}(\bm{v})[\hat{\bm{v}}]\cdot\bm{v}=0$ holds for any $\bm{v}\in\mathbb{R}^3\setminus\{\bm{0}\}$ and $\hat{\bm{v}}\in\mathbb{R}^3$. 
Our strategy of the modification is that we are going to construct the approximate initial data in the form 
\begin{equation}\label{defmodu}
\begin{cases}
 \bm{u}_0 = \bm{F}(\bm{v}_0), \\
 \bm{u}_1 = D\bm{F}(\bm{v}_0)[\bm{v}_1],
\end{cases}
\end{equation}
where $\bm{v}_0$ and $\bm{v}_1$ will be constructed in a similar form to \eqref{AppID1}, \eqref{AppID2}, \eqref{AppID3}, and \eqref{AppID4}. 
Then, the constraints $|\bm{u}_0(s)|=1$ and $\bm{u}_0(s)\cdot\bm{u}_1(s)=0$ are automatically satisfied. 
Moreover, if $\bm{v}_0$ and $\bm{v}_1$ themselves satisfy these constraints, then $(\bm{u}_0,\bm{u}_1)$ are coincided with $(\bm{v}_0,\bm{v}_1)$.

A drawback of this strategy is caused by the degeneracy of the map $\bm{F}$. 
In fact, if we use the maps $\bm{\Xi}^\delta$ and $\bm{\mathfrak{X}}^\delta$ defined by \eqref{defXi} and \eqref{deffX1}--\eqref{deffX2}, respectively, 
then the derivatives $D_{\bm{a}}\bm{\Xi}^\delta(\bm{u}_0^\mathrm{in},\bm{u}_1^\mathrm{in},\bm{0})$ and $D_{\bm{a}}\bm{\mathfrak{X}}^\delta(0,\bm{a}^*)$ 
are not invertible due to the degeneracy, so that we cannot apply the implicit function theorem for the maps. 
Therefore, we need to modify these maps in order that the implicit function theorem can be applied.

To illustrate the idea of the modification, we first consider the compatibility condition at order $2$, that is, $\bm{x}_2(1)=\bm{0}$. 
In the following, we use the same notations as those in the previous section. 
Suppose that the initial data $(\bm{u}_0,\bm{u}_1)$ are given in the form \eqref{defmodu}. 
Then, we see that 
\begin{align*}
\bm{x}_2(1)
&= (\tau_0\bm{u}_0)'(1)+\bm{g} \\
&= \frac{\tau_0(1)}{|\bm{v}_0(1)|}P(\bm{v}_0(1))\bm{v}_0'(1) - (\bm{u}_0(1)\cdot\bm{g})\bm{u}_0(1) + \bm{g} \\
&= P(\bm{v}_0(1))\left( \frac{\tau_0(1)}{|\bm{v}_0(1)|}\bm{v}_0'(1) +  P(\bm{v}_0(1))\bm{g} \right).
\end{align*}
Therefore, it is sufficient to define $\bm{X}_2(\bm{v}_0,\bm{v}_1)$ by 
\[
\bm{X}_2(\bm{v}_0,\bm{v}_1) := \frac{\tau_0(1)}{|\bm{v}_0(1)|}\bm{v}_0'(1) +  P(\bm{v}_0(1))\bm{g},
\]
rather than $\bm{X}_2(\bm{v}_0,\bm{v}_1)=(\tau_0\bm{u}_0)'(1)+\bm{g}$. 
We note that if $\bm{v}_0$ satisfies the constraint $|\bm{v}_0(s)|=1$, then these two definitions coincide.

We proceed to consider the general case, that is, the compatibility condition at order $j+2$ with $j\geq1$, that is, $\bm{x}_{j+2}(1)=\bm{0}$. 
As in the case $j=0$, we are going to find an appropriate definition of $\bm{X}_{j+2}(\bm{v}_0,\bm{v}_1)$, 
which would coincide with the definition in the previous section when $\bm{v}_0$ and $\bm{v}_1$ satisfy the constraints.

\begin{lemma}\label{lem:Str1}
If the initial data $(\bm{u}_0,\bm{u}_1)$ satisfy the constraints $|\bm{u}_0(s)|=1$ and $\bm{u}_0(s)\cdot\bm{u}_1(s)=0$ for $0<s<1$, 
then for any non-negative integer $j$ we have 
\[
\sum_{j_1+j_2=j}\frac{j!}{j_1!j_2!}\bm{u}_{j_1}(s)\cdot\bm{u}_{j_2}(s) = \delta_{j0}
\]
for $0<s<1$ as long as $\{\bm{u}_k\}_{k=0}^j$ can be defined by \eqref{TPBVPj} and \eqref{EqUj}, where $\delta_{j0}$ is the Kronecker delta. 
\end{lemma}

\begin{remark}\label{re:Str1}
If $\{\bm{u}_k\}_{k=0}^j$ are defined by $\bm{u}_k=(\dt^k\bm{u})|_{t=0}$ with a function $\bm{u}$ satisfying $|\bm{u}(s,t)|=1$, 
then we can easily obtain the identity of the lemma by differentiating $|\bm{u}(s,t)|^2=1$ $j$-times with respect to $t$ and then by putting $t=0$. 
However, we cannot use a priori the existence of such a function $\bm{u}$ and $\{\bm{u}_k\}_{k=0}^j$ were defined by \eqref{TPBVPj} and \eqref{EqUj}, 
so that the lemma is not so trivial. 
\end{remark}

\begin{proof}[Proof of Lemma \ref{lem:Str1}]
We use the identity 
\begin{align*}
\dt^2(|\bm{u}|^2) 
&= 2(\ddot{\bm{u}}-(\tau\bm{u})'')\cdot\bm{u} + 2(\tau''-|\bm{u}'|^2\tau+|\dot{\bm{u}}|^2 ) \\
&\quad\;
 + 2\tau'(|\bm{u}|^2-1)' + \tau(|\bm{u}|^2-1)'' + 2\tau''(|\bm{u}|^2-1),
\end{align*}
which holds for any smooth functions $\bm{u}=\bm{u}(s,t)$ and $\tau=\tau(s,t)$. 
Plugging $\bm{u}=\sum_{k=0}^{j+2}\frac{t^k}{k!}\bm{u}_k(s)$ and $\tau=\sum_{k=0}^j\frac{t^k}{k!}\tau_k(s)$ into the above identity, 
differentiating the resulting equation $j$-times with respect to $t$, putting $t=0$, and then using \eqref{TPBVPj} and \eqref{EqUj} we obtain 
\begin{align*}
\sum_{j_1+j_2=j+2}\frac{j!}{j_1!j_2!}\bm{u}_{j_1}\cdot\bm{u}_{j_2}
 = \sum_{j_1+j_2=j}\frac{j!}{j_1!j_2!}\Biggl\{ 
& 2\tau_{j_1}'\Biggl( \sum_{k_1+k_2=j_2}\frac{j_2!}{k_1!k_2!}\bm{u}_{k_1}\cdot\bm{u}_{k_2}-\delta_{j_20}\Biggr)' \\
&+ \tau_{j_1}\Biggl( \sum_{k_1+k_2=j_2}\frac{j_2!}{k_1!k_2!}\bm{u}_{k_1}\cdot\bm{u}_{k_2}-\delta_{j_20}\Biggr)'' \\
&+ 2\tau_{j_1}''\Biggl( \sum_{k_1+k_2=j_2}\frac{j_2!}{k_1!k_2!}\bm{u}_{k_1}\cdot\bm{u}_{k_2}-\delta_{j_20} \Biggr) \Biggr\}.
\end{align*}
Using this inductively, we obtain the desired identity. 
\end{proof}

We recall that $\bm{x}_{j+2}$ for $j\geq0$ was defined by \eqref{EqUj}. 

\begin{lemma}\label{lem:Str2}
Under the same assumptions of Lemma \ref{lem:Str1}, for any non-negative integer $j$ we have 
\[
\sum_{j_1+j_2=j}\frac{j!}{j_1!j_2!}\bm{x}_{j_1+2}(s)\cdot\bm{u}_{j_2}(s)=\tau_j'(s)+\bm{g}\cdot\bm{u}_j(s)
\]
for $0<s<1$ as long as $\{\bm{x}_{k+2}\}_{k=0}^j$, $\{\bm{u}_k\}_{k=0}^j$, and $\{\tau_k\}_{k=0}^j$ can be defined 
by \eqref{TPBVPj} and \eqref{EqUj}. 
\end{lemma}

\begin{remark}\label{re:Str2}
If $\bm{x}_j$, $\bm{u}_j$, and $\tau_j$ are defined by $\bm{x}_j=(\dt^j\bm{x})|_{t=0}$, $\tau_j=(\dt^j\tau)|_{t=0}$, and $\bm{u}_j=\bm{x}_j'$ with 
a solution $(\bm{x},\tau)$ of \eqref{Eq}, then we can easily obtain the identity of the lemma as follows. 
Taking an inner product of the first equation in \eqref{Eq} with $\bm{u}$ and using the second equation in \eqref{Eq}, 
we obtain $\ddot{\bm{x}}\cdot\bm{u}=\tau'+\bm{g}\cdot\bm{u}$. 
Differentiating this $j$-times with respect to $t$ and then putting $t=0$, we get the identity of the lemma. 
\end{remark}

\begin{proof}[Proof of Lemma \ref{lem:Str2}]
It follows from \eqref{EqUj} and Lemma \ref{lem:Str1} that 
\begin{align*}
\sum_{j_1+j_2=j}\frac{j!}{j_1!j_2!}\bm{x}_{j_1+2}\cdot\bm{u}_{j_2}
&= \sum_{j_1+j_2=j}\frac{j!}{j_1!j_2!}\Biggl( \sum_{k_0+k_1=j_1}\frac{j_1!}{k_0!k_1!}(\tau_{k_0}\bm{u}_{k_1})' + \delta_{j_10}\bm{g} \Biggr)\cdot\bm{u}_{j_2} \\
&= \sum_{j_0+j_1+j_2=j}\frac{j!}{j_0!j_1!j_2!}(\tau_{j_0}\bm{u}_{j_1})'\cdot\bm{u}_{j_2} + \bm{g}\cdot\bm{u}_j \\
&= \sum_{j_0+j_1=j}\frac{j!}{j_0!j_1!}\Biggl\{ \tau_{j_0}'\sum_{k_1+k_2=j_1}\frac{j_1!}{k_1!k_2!}\bm{u}_{k_1}\cdot\bm{u}_{k_2} \\
&\makebox[5em]{}
 + \tau_{j_0}\frac12\Biggl(\sum_{k_1+k_2=j_1}\frac{j_1!}{k_1!k_2!}\bm{u}_{k_1}\cdot\bm{u}_{k_2}\Biggr)' \Biggr\} + \bm{g}\cdot\bm{u}_j \\
&= \sum_{j_0+j_1=j}\frac{j!}{j_0!j_1!}\tau_{j_0}'\delta_{j_10} + \bm{g}\cdot\bm{u}_j \\
&= \tau_j'+\bm{g}\cdot\bm{u}_j,
\end{align*}
which is the desired identity. 
\end{proof}

We go back to consider the compatibility condition at order $j+2$ with $j\geq0$, that is, $\bm{x}_{j+2}(1)=\bm{0}$. 
It follows from Lemma \ref{lem:Str2} that $\sum_{k=0}^j\binom{j}{k}\bm{x}_{k+2}(1)\cdot\bm{u}_{j-k}(1)=0$ so that 
$\bm{x}_{j+2}(1)\cdot\bm{u}_0(1)=-\sum_{k=0}^{j-1}\binom{j}{k}\bm{x}_{k+2}(1)\cdot\bm{u}_{j-k}(1)$. 
Therefore, we have 
\[
\bm{x}_{j+2}(1)=P(\bm{u}_0(1))\bm{x}_{j+2}(1)-\sum_{k=0}^{j-1}\binom{j}{k}(\bm{x}_{k+2}(1)\cdot\bm{u}_{j-k}(1))\bm{u}_0(1).
\]
We recall that we are constructing the initial data $(\bm{u}_0,\bm{u}_1)$ in the form \eqref{defmodu}, so that we have 
\[
\begin{cases}
 \ds^{2j+1}\bm{u}_0 = \frac{1}{|\bm{v}_0|}P(\bm{v}_0)\ds^{2j+1}\bm{v}_0 + [\ds^{2j},D\bm{F}(\bm{v}_0)]\bm{v}_0', \\
 \ds^{2j+1}\bm{u}_1 = \frac{1}{|\bm{v}_0|}P(\bm{v}_0)\ds^{2j+1}\bm{v}_1 + [\ds^{2j},D\bm{F}(\bm{v}_0)]\bm{v}_1'
  + \ds^{2j}(D^2\bm{F}(\bm{v}_0)[\bm{v}_1,\bm{v}_0']).
\end{cases}
\]
Therefore, by Lemma \ref{lem:ExpX} we get 
\[
\begin{cases}
 \bm{x}_{2j+2}(1) = P(\bm{v}_0(1))\Bigl\{ \frac{\tau_0(1)^{j+1}}{|\bm{v}_0(1)|}(\ds^{2j+1}\bm{v}_0)(1) 
  + \tau_0(1)^{j+1}P(\bm{v}_0(1))( [\ds^{2j},D\bm{F}(\bm{v}_0)]\bm{v}_0' )(1) \\
 \phantom{\bm{x}_{2j+2}(1) = }
  + \mathcal{P}_{2j+2}((\ds^{\leq 2j}\bm{u}_0)(1),(\ds^{\leq 2j-1}\bm{u}_1)(1),\tau_0(1),\ldots,\tau_{2j}(1)) \Bigr\} \\
 \phantom{\bm{x}_{2j+2}(1) = }
  - \sum_{k=0}^{2j-1}\binom{2j}{k}(\bm{x}_{k+2}(1)\cdot\bm{u}_{2j-k}(1))\bm{u}_0(1), \\
 \bm{x}_{2j+3}(1) = P(\bm{v}_0(1))\Bigl\{ \frac{\tau_0(1)^{j+1}}{|\bm{v}_0(1)|}(\ds^{2j+1}\bm{v}_1)(1) \\
 \phantom{\bm{x}_{2j+2}(1) = }
  + \tau_0(1)^{j+1}P(\bm{v}_0(1))( [\ds^{2j},D\bm{F}(\bm{v}_0)]\bm{v}_1' + \ds^{2j}(D^2\bm{F}(\bm{v}_0)[\bm{v}_1,\bm{v}_0']) )(1) \\
 \phantom{\bm{x}_{2j+2}(1) = }
  + \mathcal{P}_{2j+3}((\ds^{\leq 2j+1}\bm{u}_0)(1),(\ds^{\leq 2j}\bm{u}_1)(1),\tau_0(1),\ldots,\tau_{2j+1}(1)) \Bigr\} \\
 \phantom{\bm{x}_{2j+2}(1) = }
  - \sum_{k=0}^{2j}\binom{2j+1}{k}(\bm{x}_{k+2}(1)\cdot\bm{u}_{2j+1-k}(1))\bm{u}_0(1).
\end{cases}
\]
In view of these expressions, we define $\bm{X}_{j+2}(\bm{v}_0,\bm{v}_1)$ by 
\[
\begin{cases}
 \bm{X}_{2j+2}(\bm{v}_0,\bm{v}_1) = \frac{\tau_0(1)^{j+1}}{|\bm{v}_0(1)|}(\ds^{2j+1}\bm{v}_0)(1) 
  + \tau_0(1)^{j+1}P(\bm{v}_0(1))( [\ds^{2j},D\bm{F}(\bm{v}_0)]\bm{v}_0' )(1) \\
 \phantom{\bm{x}_{2j+2}(1) = }
  + \mathcal{P}_{2j+2}((\ds^{\leq 2j}\bm{u}_0)(1),(\ds^{\leq 2j-1}\bm{u}_1)(1),\tau_0(1),\ldots,\tau_{2j}(1)), \\
 \bm{X}_{2j+3}(\bm{v}_0,\bm{v}_1) = \frac{\tau_0(1)^{j+1}}{|\bm{v}_0(1)|}(\ds^{2j+1}\bm{v}_1)(1) \\
 \phantom{\bm{x}_{2j+2}(1) = }
  + \tau_0(1)^{j+1}P(\bm{v}_0(1))( [\ds^{2j},D\bm{F}(\bm{v}_0)]\bm{v}_1' + \ds^{2j}(D^2\bm{F}(\bm{v}_0)[\bm{v}_1,\bm{v}_0']) )(1) \\
 \phantom{\bm{x}_{2j+2}(1) = }
  + \mathcal{P}_{2j+3}((\ds^{\leq 2j+1}\bm{u}_0)(1),(\ds^{\leq 2j}\bm{u}_1)(1),\tau_0(1),\ldots,\tau_{2j+1}(1))
\end{cases}
\]
rather than $\bm{X}_{j+2}(\bm{v}_0,\bm{v}_1)=\bm{x}_{j+2}(1)$ for $j=0,1,2,\ldots$. 
We note that if $\bm{v}_0$ and $\bm{v}_1$ satisfy the constraints $|\bm{v}_0(s)|=1$ and $\bm{v}_0(s)\cdot\bm{v}_1(s)=0$, 
then these two definitions coincide. 
Moreover, if $\bm{v}_0$ and $\bm{v}_1$ are constructed so that $\bm{X}_{k+2}(\bm{v}_0,\bm{v}_1)=\bm{0}$ for $k=0,1,\ldots,j$, 
then for the initial data $(\bm{u}_0,\bm{u}_1)$ given by \eqref{defmodu} we have $\bm{x}_{k+2}(1)=\bm{0}$ for $k=0,1,\ldots,j$.

Now, we have resolved the drawback caused by the degeneracy of the map $\bm{F}$, namely, if we use these modified functionals 
$\bm{X}_{j+2}$ for $j=0,1,2,\ldots$, then the calculations given in the previous section still work to construct desired approximate initial data. 
The proof of Proposition \ref{prop:AppID2} is complete. 
\hfill$\Box$

%----------------------------------------------------------------------------------------------------------------------
%----------------------------------------------------------------------------------------------------------------------
\section{Remark on the a priori estimate}\label{sect:APE}
In this last section, we sketch a proof of Proposition \ref{prop:APE} on the improvement of the a priori estimate of the solutions. 
The proof is carried out by the induction on $m$. 
The case $m=4$ was actually shown in \cite[Theorem 2.1]{IguchiTakayama2024}, so that there exists a positive time $T=T(M_*,c_0)$ such that 
the solution satisfies the stability condition \eqref{SC} and 
$\opnorm{ \bm{x}(t) }_4 + \opnorm{ \tau'(t) }_{3,*,\epsilon} \lesssim 1$ for $0\leq t\leq T$ and $\epsilon>0$. 
Therefore, it is sufficient to show that $\opnorm{ \bm{x}(t) }_m + \opnorm{ \tau'(t) }_{m-1,*} \lesssim 1$ holds for the same time interval 
$0\leq t\leq T$ in the case $m\geq5$. 
In the following, we focus on the case $m=5$, because the case $m\geq6$ can be handled similarly and more easily.

We note that the estimates in the case $m=4$ imply $|\tau(s,t)|+|\dot{\tau}(s,t)|\lesssim s$ and $|\ddot{\tau}(s,t)|\lesssim s^{1-\epsilon}$ for $\epsilon>0$. 
By the embedding $\|s^\epsilon u\|_{L^\infty}\lesssim \|u\|_{X^1}$, we have also $\|s^\epsilon\dot{\bm{x}}'(t)\|_{L^\infty}\lesssim 1$ for $\epsilon>0$, 
so that $\|\dot{\bm{x}}'(t)\|_{L^p} \lesssim 1$ for $p\in[1,\infty)$. 
We will use these estimates in the following without any comments. 
We introduce an energy function $E(t)$ by 
\[
E(t) = \|\dt^4\bm{x}(t)\|_{X^1}^2 + \|\dt^3\bm{x}(t)\|_{X^2}^2. 
\]

\begin{lemma}\label{lem:EstXTau_bis}
It holds that $\opnorm{ \bm{x}(t) }_{5,*}+\opnorm{ \tau'(t) }_{3,*} \lesssim 1+E(t)^\frac12$. 
\end{lemma}

\begin{proof}
Integrating the hyperbolic equations for $\bm{x}$ in \eqref{Eq} with respect to $s$ over $[0,s]$ and using the boundary condition $\tau|_{s=0}=0$ in \eqref{BC}, 
we obtain $\tau\bm{x}'=\int_0^s(\ddot{\bm{x}}-\bm{g})\mathrm{d}\sigma$, so that 
\[
\bm{x}'=\mu^{-1}(\mathscr{M}\ddot{\bm{x}}-\bm{g}),
\]
where $\mathscr{M}$ is the averaging operator defined by \eqref{AvOp} and $\mu=\mathscr{M}\tau'$. 
By this expression, we obtain $\|\dot{\bm{x}}'(t)\|_{L^\infty}\lesssim 1+ \|\dt^3\bm{x}(t)\|_{L^\infty} \lesssim 1+E(t)^\frac12$. 
Therefore, by following exactly the calculations in the proof of \cite[Lemma 8.2]{IguchiTakayama2024}, 
we get $\opnorm{ \tau'(t) }_{3,*} \lesssim 1+E(t)^\frac12$. 
To evaluate $\opnorm{ \bm{x}(t) }_{5,*}$ we follow the calculations in \cite[Section 10.1]{IguchiTakayama2024}, 
but this time we also use the tame estimates in Lemmas \ref{lem:tame1}--\ref{lem:tame2} to evaluate the nonlinear terms. 
Then, we obtain $\opnorm{ \bm{x}(t) }_{5,*} \lesssim 1+\opnorm{ \tau'(t) }_{3,*}$. 
\end{proof}

\begin{lemma}\label{lem:EstX_special_bis}
For any $\epsilon>0$, it holds that $\|s^\epsilon\dt^3\tau'(t)\|_{L^\infty} \lesssim 1+E(t)^\frac12$. 
Particularly, $\|\dt^3\tau'(t)\|_{L^p} \lesssim 1+E(t)^\frac12$ holds for $1\leq p<\infty$. 
\end{lemma}

\begin{proof}
Differentiating \eqref{BVP2} $3$-times with respect to $t$ we obtain 
\[
\begin{cases}
 -(\dt^3\tau)''+|\bm{x}''|^2\dt^3\tau = h_3 &\mbox{in}\quad (0,1)\times(0,T), \\
 \dt^3\tau=0 &\mbox{on}\quad \{s=0\}\times(0,T), \\
 \dt^3\tau'= -\bm{g}\cdot\dt^3\bm{x}' &\mbox{on}\quad \{s=1\}\times(0,T),
\end{cases}
\]
where $h_3=\dt^3(|\dot{\bm{x}}'|^2) - [\dt^3,|\bm{x}''|^2]\tau$. 
Therefore, by Lemma \ref{lem:EstSolBVP3} we get 
\[
\|s^\epsilon\dt^3\tau'(t)\|_{L^\infty} \lesssim |\dt^3\bm{x}'(1,t)|+\|s^\epsilon h_3(t)\|_{L^1}
\]
for $0<\epsilon\leq1$. 
By the Sobolev embedding theorem, the first term in the right-hand side can be easily evaluated by $\|\dt^3\bm{x}(t)\|_{X^2}$. 
The most subtle term in $h_3$ is $\dot{\bm{x}}'\cdot\dt^3\dot{\bm{x}}'$, which can be evaluated as 
\begin{align*}
\|s^\epsilon(\dot{\bm{x}}'\cdot\dt^3\dot{\bm{x}}')\|_{L^1} 
&\leq \|s^{-\frac12(1-\epsilon)}\|_{L^2}\|s^\frac{\epsilon}{2}\dot{\bm{x}}'\|_{L^\infty}\|s^\frac12\dt^4\bm{x}'\|_{L^2} \\
&\lesssim \|\dt^4\bm{x}\|_{X^1}.
\end{align*}
The other terms in $h_3$ can be easily handled without the weight $s^\epsilon$ and we obtain 
$\|s^\epsilon h_3(t)\|_{L^1}\lesssim 1+\opnorm{ \bm{x}(t) }_{5,*}$. 
These estimates and Lemma \ref{lem:EstXTau_bis} give the desired one. 
\end{proof}

We proceed to evaluate the energy function $E(t)$. 
Differentiating \eqref{Eq}, \eqref{BC}, and \eqref{HP2} $3$-times with respect to $t$, we see that $\bm{y}=\dt^3\bm{x}$ and $\nu=\dt^3\tau$ satisfy 
\[
\begin{cases}
 \ddot{\bm{y}}-(\tau\bm{y}')'-(\nu\bm{x}')' = \bm{f} &\mbox{in}\quad (0,1)\times(0,T), \\
 \bm{x}'\cdot\bm{y}'=f &\mbox{in}\quad (0,1)\times(0,T), \\
 \bm{y}=\bm{0} &\mbox{on}\quad \{s=1\}\times(0,T),
\end{cases}
\]
and 
\[
\begin{cases}
 -\nu''+|\bm{x}''|^2\nu = 2\dot{\bm{x}}'\cdot\dot{\bm{y}}' - 2(\bm{x}''\cdot\bm{y}'')\tau + h &\mbox{in}\quad (0,1)\times(0,T), \\
 \nu = 0 &\mbox{on}\quad \{s=0\}\times(0,T), \\
 \nu' = -\bm{g}\cdot\bm{y}' &\mbox{on}\quad \{s=1\}\times(0,T),
\end{cases}
\]
where $\bm{f} = \bigl( [ \dt^3; \tau,\bm{x}'] \bigr)'$, $f = -\frac12[\dt^3; \bm{x}',\bm{x}']$, and 
$h = [\dt^3;\dot{\bm{x}}',\dot{\bm{x}}']-\tau[\dt^3;\bm{x}'',\bm{x}''] - [\dt^3;\tau,|\bm{x}''|^2]$. 
Here, $[P;\bm{u},\bm{v}] = P(\bm{u}\cdot\bm{v})-\bm{u}\cdot(P\bm{v})-(P\bm{u})\cdot\bm{v}$ denotes the symmetric commutator. 
We recall that $E(t)=\|\dot{\bm{y}}(t)\|_{X^1}^2 + \|\bm{y}(t)\|_{X^2}^2$. 
Therefore, by \cite[Proposition 6.2]{IguchiTakayama2024}, we obtain 
\[
E(t) \lesssim E(0) + S_1(0) + \int_0^t S_2(t')\mathrm{d}t', 
\]
where 
\[
\begin{cases}
 S_1(t) = \|\bm{f}\|_{L^2}^2 + \|sh\|_{L^1}^2,\\
 S_2(t) = \|\dot{\bm{f}}\|_{L^2}^2 + \|s^{\frac12-\epsilon}\dot{f}\|_{L^2}^2 + |\dot{f}|_{s=1}|^2
 + \|s\dot{h}\|_{L^1}^2 + \|s^{\frac12+\epsilon}h\|_{L^2}^2
\end{cases}
\]
with $0<\epsilon<\frac12$. 
By the analysis in \cite[Section 9]{IguchiTakayama2024}, 
under the assumption $\|\bm{x}_0^\mathrm{in}\|_{X^5}+\|\bm{x}_1^\mathrm{in}\|_{X^4} \lesssim 1$ we have $\opnorm{ \bm{x}(0) }_5\lesssim 1$ 
so that we obtain easily $E(0) + S_1(0) \lesssim 1$. 
The most subtle term in the evaluation of $S_2(t)$ is $((\dt^3\tau)\dot{\bm{x}}')'$ in $\dot{\bm{f}}$. 
Since $\dt^3\tau=s\mathscr{M}(\dt^3\tau')$, this term can be evaluated as 
\begin{align*}
\|((\dt^3\tau)\dot{\bm{x}}')'\|_{L^2}
&\leq \|\dt^3\tau'\|_{L^4}\|\dot{\bm{x}}'\|_{L^4} + \|s^\frac12\mathscr{M}(\dt^3\tau')\|_{L^\infty}\|s^\frac12\dot{\bm{x}}''\|_{L^2} \\
&\lesssim \|\dt^3\tau'\|_{L^4} + \|s^\frac12\dt^3\tau'\|_{L^\infty} \\
&\lesssim 1+E(t)^\frac12,
\end{align*}
where we used Lemmas \ref{lem:AvOp} and \ref{lem:EstX_special_bis}. 
The other terms in $S_2(t)$ can be easily handled and we obtain $S_2(t) \lesssim 1+E(t)$. 
Thus, we get $E(t) \lesssim 1+\int_0^tE(t')\mathrm{d}t'$ so that $E(t)\lesssim 1$, 
which together with Lemma \ref{lem:EstXTau_bis} implies $\opnorm{ \bm{x}(t) }_{5,*}+\opnorm{ \tau'(t) }_{3,*} \lesssim 1$ for $0\leq t\leq T$.

Finally, it is sufficient to use the first equation in \eqref{Eq} to evaluate $\|\dt^5\bm{x}(t)\|_{L^2}$, 
and the estimate for $\opnorm{ \tau'(t) }_{4,*}$ follows directly from \cite[Lemma 8.1]{IguchiTakayama2024}. 
As a result, we obtain $\opnorm{\bm{x}(t)}_5+\opnorm{\tau'(t)}_{4,*} \lesssim 1$ for $0\leq t\leq T$. 
\hfill$\Box$

%----------------------------------------------------------------------------------------------------------------------
%----------------------------------------------------------------------------------------------------------------------

\bigskip
Tatsuo Iguchi \par
{\sc Department of Mathematics} \par
{\sc Faculty of Science and Technology, Keio University} \par
{\sc 3-14-1 Hiyoshi, Kohoku-ku, Yokohama, 223-8522, Japan} \par
E-mail: \texttt{iguchi@math.keio.ac.jp}

\bigskip
Masahiro Takayama \par
{\sc Department of Mathematics} \par
{\sc Faculty of Science and Technology, Keio University} \par
{\sc 3-14-1 Hiyoshi, Kohoku-ku, Yokohama, 223-8522, Japan} \par
E-mail: \texttt{masahiro@math.keio.ac.jp}

\end{document}